\documentclass[oneside,english]{amsart}
\usepackage[T1]{fontenc}
\usepackage[latin9]{inputenc}
\usepackage[letterpaper]{geometry}
\usepackage{verbatim}
\usepackage{float}
\usepackage{amsthm}
\usepackage{amstext}
\usepackage{graphicx}
\usepackage{amssymb}
\usepackage{esint}

\makeatletter

\floatstyle{ruled}
\newfloat{algorithm}{tbp}{loa}
\floatname{algorithm}{Algorithm}

\numberwithin{equation}{section}
\numberwithin{figure}{section}
\theoremstyle{plain}
\newtheorem{thm}{Theorem}
  \theoremstyle{remark}
  \newtheorem{rem}[thm]{Remark}
\newenvironment{lyxcode}
{\par\begin{list}{}{
\setlength{\rightmargin}{\leftmargin}
\setlength{\listparindent}{0pt}
\raggedright
\setlength{\itemsep}{0pt}
\setlength{\parsep}{0pt}
\normalfont\ttfamily}%
 \item[]}
{\end{list}}
  \theoremstyle{remark}
  \newtheorem*{rem*}{Remark}
  \theoremstyle{plain}
  \newtheorem{lem}[thm]{Lemma}
  \theoremstyle{plain}
  \newtheorem{prop}[thm]{Proposition}

\makeatother

\usepackage{babel}

\begin{document}

\title[Con-eigenvalue algorithm for optimal rational approximations]{fast and accurate con-eigenvalue algorithm for optimal rational approximations }

\author{T. S. Haut and G. Beylkin}

\address{Department of Applied Mathematics, University of Colorado, Boulder,
CO 80309-0526, United States}

\thanks{This research was partially supported by NSF grant DMS-100995 and
DOE/ORNL grant 4000038129.}
\begin{abstract}
The need to compute small con-eigenvalues and the associated con-eigenvectors
of positive-definite Cauchy matrices naturally arises when constructing
rational approximations with a (near) optimally small $L^{\infty}$
error. Specifically, given a rational function with $n$ poles in
the unit disk, a rational approximation with $m\ll n$ poles in the
unit disk may be obtained from the $m$th con-eigenvector of an $n\times n$
Cauchy matrix, where the associated con-eigenvalue $\lambda_{m}>0$
gives the approximation error in the $L^{\infty}$ norm. Unfortunately,
standard algorithms do not accurately compute small con-eigenvalues
(and the associated con-eigenvectors) and, in particular, yield few
or no correct digits for con-eigenvalues smaller than the machine
roundoff. We develop a fast and accurate algorithm for computing con-eigenvalues
and con-eigenvectors of positive-definite Cauchy matrices, yielding
even the tiniest con-eigenvalues with high relative accuracy. The
algorithm computes the $m$th con-eigenvalue in $\mathcal{O}\left(m^{2}n\right)$
operations and, since the con-eigenvalues of positive-definite Cauchy
matrices decay exponentially fast, we obtain (near) optimal rational
approximations in $\mathcal{O}\left(n\left(\log\delta^{-1}\right)^{2}\right)$
operations, where $\delta$ is the approximation error in the $L^{\infty}$
norm. We derive error bounds demonstrating high relative accuracy
of the computed con-eigenvalues and the high accuracy of the unit
con-eigenvectors. We also provide examples of using the algorithm
to compute (near) optimal rational approximations of functions with
singularities and sharp transitions, where approximation errors close
to machine precision are obtained. Finally, we present numerical tests
on random (complex-valued) Cauchy matrices to show that the algorithm
computes all the con-eigenvalues and con-eigenvectors with nearly
full precision. 
\end{abstract}
\maketitle

\section{Introduction}

We present an algorithm for computing with high relative accuracy
the con-eigenvalue decomposition of positive-definite Cauchy matrices,\begin{equation}
Cu_{m}=\lambda_{m}\overline{u_{m}},\,\,\,\, C_{ij}=\frac{\sqrt{\alpha_{i}}\sqrt{\overline{\alpha_{j}}}}{1-\gamma_{i}\overline{\gamma_{j}}},\,\, i,j=1,\ldots,n,\label{eq:con-eigenvalue problem-1}\end{equation}
where $\gamma_{i}$ and $\alpha_{i}$ are complex numbers and $\left|\gamma_{i}\right|<1$.The
con-eigenvalue $\lambda_{m}$ is only defined up to an arbitrary phase,
which we choose so that $\lambda_{m}>0$. Although the con-eigenvalue
decomposition (see e.g. \cite{HOR-JOH:1990}) is less well-known than
the eigenvalue decomposition or the singular value decomposition,
it arises naturally in constructing optimal approximations using exponentials
or rational functions \cite{AD-AR-KR:1968,AD-AR-KR:1968a,AD-AR-KR:1971,DE-GE-KA:1981,YOUNG:1983,BEY-MON:2005,BEY-MON:2009}.
For example, for a real-valued rational function $f(z)$,\begin{equation}
f(z)=\sum_{i=1}^{n}\frac{\alpha_{i}}{z-\gamma_{i}}+\sum_{i=1}^{n}\frac{\overline{\alpha_{i}}z}{1-\overline{\gamma_{i}}z}+\alpha_{0},\label{eq:function form for reduction problem}\end{equation}
we may construct a rational approximation $g(z)$ with $m$ poles
and with an error, \[
\max_{x\in[0,1]}\left|f\left(e^{2\pi ix}\right)-g\left(e^{2\pi ix}\right)\right|\approx\lambda_{m},\]
by solving the con-eigenvalue problem (\ref{eq:con-eigenvalue problem-1})
(see Section~\ref{sub:Reduction-procedure} for more detail). Ordering
the con-eigenvalues, $\lambda_{1}\geq\ldots\geq\lambda_{n}>0$, the
number of poles $m$ of the approximant $g(z)$ corresponds to the
index of the con-eigenvalue $\lambda_{m}$ and leads to a near optimal
approximation in the $L^{\infty}$-norm with the error close to $\lambda_{m}$.
The form (\ref{eq:function form for reduction problem}) ensures that
$f\left(e^{2\pi ix}\right)$ is real-valued and periodic; complex-valued
functions may also be handled using this form by splitting the real
and imaginary parts and performing additional reductions (see \cite{BEY-MON:2009}). 

Current algorithms compute an approximate con-eigenvalue $\widehat{\lambda_{m}}$
with an error no better than $\left|\lambda_{m}-\widehat{\lambda_{m}}\right|/\left|\lambda_{1}\right|=\mathcal{O}\left(\epsilon\right)$,
and an approximate unit con-eigenvector $\widehat{u_{m}}$ with an
error no better than \[
\left\Vert u_{m}-\widehat{u_{m}}\right\Vert _{2}=\mathcal{O}\left(\epsilon\right)/\text{absgap}_{m},\,\,\,\,\,\text{absgap}_{m}\equiv\min_{p\neq m}\left|\lambda_{m}-\lambda_{p}\right|/\left|\lambda_{1}\right|,\]
where $\epsilon$ denotes the machine roundoff. This implies that
a computed con-eigenvalue smaller than $\left|\lambda_{1}\right|\epsilon$
may have few or no correct digits. Hence, in order to obtain a rational
approximation with accuracy $\lambda_{m}\lesssim10^{-7}$, we may
be forced to use at least quadruple precision. Since quadruple precision
is typically not supported by the hardware, it slows down the computation
by an unpleasant factor (between $30$ and $100$). Another undesirable
feature of current algorithms to solve (\ref{eq:con-eigenvalue problem-1})
is the $\mathcal{O}\left(n^{3}\right)$ complexity for finding the
$m\ll n$ poles of $g(z)$, where $n$ is the original number of poles
of $f(z)$. 

Although the construction of optimal rational approximations in the
$L^{\infty}$-norm has a long history (starting with the seminal papers
\cite{AD-AR-KR:1968,AD-AR-KR:1968a,AD-AR-KR:1971}), the difficulties
mentioned above limit practical applications of such approximations
to situations where the problem size is relatively small and a low
accuracy is acceptable. In this regard, we view our results as a stepping
stone toward a wider use of optimal $L^{\infty}$-approximations in
numerical analysis (see \cite{HA-BE-MO:2012}).

We develop a fast and accurate algorithm for con-eigenvalue/con-eigenvector
computations of positive-definite Cauchy matrices that addresses both
of the difficulties mentioned above. Our algorithm computes the $m$th
con-eigenvalue/con-eigenvector in $\mathcal{O}\left(m^{2}n\right)$
operations (see Section~\ref{sec:Accuracy-and-perturbation}). Since
the con-eigenvalues of positive definite Cauchy matrices decay exponentially
fast, for a given desired accuracy $\|f\left(e^{2\pi ix}\right)-g\left(e^{2\pi ix}\right)\|_{\infty}\approx\delta$,
the number of poles $m$ in the approximant $g(z)$ is $\mathcal{O}\left(\log\delta^{-1}\right)$.
Therefore, the complexity of our algorithm is $\mathcal{O}\left(n\left(\log\delta^{-1}\right)^{2}\right)$,
i.e., it is essentially linear in the number of original poles $n$
and, thus, is mostly controlled by the number of poles of the final
optimal approximation. 

The con-eigenvalue algorithm achieves high relative accuracy, i.e.,
the computed con-eigenvalue $\widehat{\lambda_{m}}$ satisfies $\left|\lambda_{m}-\widehat{\lambda_{m}}\right|/\left|\lambda_{m}\right|=\mathcal{O}\left(\epsilon\right)$,
and the computed unit con-eigenvector $\widehat{u_{m}}$ satisfies
\[
\left\Vert u_{m}-\widehat{u_{m}}\right\Vert _{2}=\mathcal{O}\left(\epsilon\right)/\mbox{relgap}_{m},\,\,\,\,\mbox{relgap}_{m}\equiv\min_{l\neq m}\left|\lambda_{m}-\lambda_{l}\right|/\left(\lambda_{l}+\lambda_{m}\right),\]
(see Theorems~\ref{thm:high accuracy of algorithm-1}~and~\ref{thm:truncation theorem}
for the exact statement). In contrast to the usual perturbation theory
for general matrices, we show that small perturbations of the poles
$\gamma_{m}$ and residues $\alpha_{m}$ (determining the Cauchy matrix
$C=C(\alpha,\gamma)$ in (\ref{eq:con-eigenvalue problem-1})) lead
to correspondingly small perturbations in the con-eigenvalues and
con-eigenvectors, as long as the poles are well separated in a relative
sense and are not too close to the unit circle.

In many applications, the function $f\left(e^{2\pi ix}\right)$ has
sharp transitions, so that the poles are clustered close to the unit
circle and each other. In such cases, it is natural to maintain the
poles of $f\left(z\right)$ in the form $\gamma_{j}=\exp\left(-\tau_{j}\right)$,
where $\mathcal{R}e\left(\tau_{j}\right)>0$ and $0\le\mathcal{I}m\left(\tau_{j}\right)<2\pi$,
so that $\mathcal{R}e\left(\tau_{j}\right)$ are well-separated in
a relative sense. The reduction algorithm produces new poles of the
same form, where even the smallest exponents are computed with high
relative accuracy. This allows us to develop a numerical calculus
that includes functions with singularities and sharp transitions.
We address this issue further in Section~\ref{sec:Examples-of-optimal}. 

Our approach is inspired by papers \cite{DEM-VES:1992,DRMAC:1998,D-G-E-S-V-D:1999,DEMMEL:1999,HIGHAM:2000},
which develop algorithms and theory for highly accurate SVDs of certain
structured matrices. Generally speaking, high relative accuracy is
achieved when it is possible to avoid catastrophic cancellation resulting
from subtracting two close floating point numbers (when the outcome
of such cancellation is significant relative to the final result).
We refer to \cite{D-D-H-K:2008} for a comprehensive analysis of when
efficient and accurate algorithms are possible using floating point
arithmetic. Classes of matrices for which highly accurate SVD or eigenvalue
algorithms exist include bi-diagonal matrices \cite{DEM-KAH:1990,D-D-L-T:1991,FER-PAR:1994},
acyclic matrices \cite{DEM-CRA:1992}, graded positive-definite matrices
\cite{DEM-VES:1992}, scaled diagonally dominant matrices \cite{BAR-DEM:1988},
totally positive matrices \cite{KOEV:2005}, certain indefinite matrices
\cite{SLAPNI:2003}, and Cauchy matrices (as well as, more generally,
matrices with displacement rank one) \cite{DEMMEL:1999}. For such
matrices, recent algorithmic advances (see \cite{DRM-VES:2007,DRM-VES:2007a})
make the cost of achieving high relative accuracy comparable to that
of alternative (and less accurate) SVD methods. 

The con-eigenvalue algorithm considered here is based on computing
the eigenvalue decomposition of the product, $\overline{C}C$, of
positive-definite Cauchy matrices $\overline{C}$ and $C$, and is
similar to the algorithm in \cite{D-G-E-S-V-D:1997} for the generalized
eigenvalue decomposition, as well as the algorithm in \cite{DRMAC:1998}
for the product SVD decomposition. We also rely on the algorithm in
\cite{DEMMEL:1999} for computing, with high relative accuracy, the
Cholesky decomposition (with complete pivoting) $C=\left(PL\right)D^{2}\left(PL\right)^{*}$
of a positive-definite Cauchy matrix $C$. However, since we are interested
in computing only con-eigenvalues of some approximate size $\delta$,
we stop Demmel's Cholesky algorithm once the diagonal elements $D_{ii}$
are small with respect to $\delta$ and the desired precision. Since
the diagonal elements $D_{ii}$ decay exponentially fast, this allows
us to accurately compute con-eigenvalues of size $\delta$ (and the
associated con-eigenvectors) in $\mathcal{O}\left(n\left(\log\delta^{-1}\right)^{2}\right)$
operations. We also modify the Cholesky decomposition algorithm in
\cite{DEMMEL:1999} to yield high relative accuracy for Cauchy matrices
$C_{ij}=\sqrt{\alpha_{i}}\sqrt{\overline{\alpha_{j}}}/\left(1-\gamma_{i}\overline{\gamma_{j}}\right)$,
with $\gamma_{i}=\exp\left(-\tau_{j}\right)$, where the real parts
of the exponents, $\mathcal{R}e(\tau_{j})$, may be extremely small
in magnitude. We observe that the error bounds developed in \cite{DRMAC:1998}
are not applicable to our problem since the condition number of a
Cauchy matrix cannot be appreciably reduced by scaling the rows and
columns. In contrast, the error bounds developed in this paper yield
high relative accuracy for all the computed con-eigenvalues larger
than $\delta$ (and high accuracy for the con-eigenvectors), as long
as the $n$ leading principal minors of $L^{\text{T}}L$ are well-conditioned,
and the relative gap between the con-eigenvalues is not too small
(we have always observed this to hold in practice). In particular,
if $\delta$ is chosen small enough, the full con-eigenvalue decomposition
is obtained with high relative accuracy. The derivation of our error
bounds makes crucial use of the component-wise perturbation theory
developed in \cite{DEM-VES:1992} for the singular vectors of graded
matrices (see also \cite{MATHIA:1997}), as well as the component-wise
error analysis in \cite{DEM-VES:1992} and \cite{MATHIA:1996} for
the one-sided Jacobi method. We also use the error analysis given
in \cite{HIGHAM:2000} for the Householder QR method. We note that
although our error estimates are much more pessimistic than what we
observe in practice, they provide a framework for understanding the
high accuracy of the con-eigenvalue algorithm of this paper.

It has been an established practice, in both numerical analysis and
signal processing, to use $L^{2}$-type methods for representing functions.
On the other hand, it has been understood for some time that nonlinear
approximations may be far superior in achieving high accuracy with
a minimal number of terms (see e.g., \cite{NEWMAN:1964}). However,
in spite of many interesting results (see e.g., \cite{KUNG:1980,TREFET:1981,DE-GE-KA:1981,TRE-GUT:1983,TREFET:1983,YOUNG:1983,BEY-MON:2005,BEY-MON:2010,DEU-TRE:2011}),
the widespread use of nonlinear approximations has been limited by
a lack of efficient and accurate algorithms for computing them (particularly
for functions with sharp changes or singularities). Our algorithms
provide the necessary tools for computing optimal nonlinear approximations
via rational functions, and come with guaranteed accuracy bounds.
We believe that these new accurate algorithms may greatly extend the
practical use of $L{}^{\infty}$ approximations in numerical analysis
(see \cite{HA-BE-MO:2012}) and signal processing (see \cite{BE-LE-MO:2011p}). 

In Section~\ref{sub:Reduction-procedure} we describe the reduction
problem for rational functions, and connect its solution to a con-eigenvalue
problem for positive definite Cauchy matrices. We then present new
algorithms for solving the con-eigenvalue problem with high relative
accuracy. We follow up in Section~\ref{sec:Examples-of-optimal}
with examples of using the reduction algorithm to construct and use
optimal rational approximations for functions with singularities and
sharp transitions. In Section~\ref{sub:Experimental-verification}
we verify the accuracy of the con-eigenvalue algorithm by comparing
the con-eigenvalue decomposition of randomly generated Cauchy matrices
with that obtained via standard algorithms in extended precision.
In Section~\ref{sec:Accuracy-and-perturbation}, we prove that the
con-eigenvalue algorithm achieves high relative accuracy and that
the con-eigenvalue decomposition is stable with respect to small perturbations
of the parameters defining the Cauchy matrix. Finally, Section~\ref{sec:Comparison-with-related}
compares the reduction algorithm of this paper with other algorithms
in the literature for constructing optimal rational approximations.
For the convenience of the reader we also provide relevant background
material in Section~\ref{sec:Background-on-algorithms}. The proof
of a technical proposition may be found in Appendix.

\section{Accurate con-eigenvalue decomposition (an informal derivation)}

\subsection{\label{sub:Reduction-procedure}Constructing optimal rational approximations
via a con-eigenvalue problem}

In order to motivate our con-eigenvalue algorithm, let us explain
how the accurate computation of small con-eigenvalues and associated
con-eigenvectors allows us to construct optimal rational approximations. 

We consider an algorithm to find a rational approximation $r(e^{2\pi ix})$
to $f(e^{2\pi ix})$ in (\ref{eq:function form for reduction problem})
with a specified number of poles and with a (nearly) optimally small
error in the $L^{\infty}$-norm. The algorithm is based on a theorem
of Adamyan, Arov, and Krein (referred to below as the AAK Theorem)
\cite{AD-AR-KR:1971}. We note that the formulation given below in
terms of a con-eigenvalue problem is similar to the approach taken
in \cite{DE-GE-KA:1981} and \cite{BEY-MON:2005}. 

Given a target accuracy $\delta$ for the error in the $L^{\infty}$-norm,
the steps for computing the rational approximant $r(z)$, \[
r(z)=\sum_{i=1}^{m}\frac{\beta_{i}}{z-\eta_{i}}+\sum_{i=1}^{m}\frac{\overline{\beta_{i}}z}{1-\overline{\eta_{i}}z}+\alpha_{0},\]
 are as follows:
\begin{enumerate}
\item Compute a con-eigenvalue $0<\lambda_{m}\leq\delta$ and corresponding
con-eigenvector $u$ of the Cauchy matrix $C_{ij}=C_{ij}(\gamma_{i},\alpha_{j})$,
\begin{equation}
Cu=\lambda_{m}\overline{u},\,\,\,\mbox{where}\,\,\, u=\left(\begin{array}{c}
u_{1}\\
u_{2}\\
\vdots\\
u_{n}\end{array}\right),\,\,\,\, C_{ij}=\frac{a_{i}b_{j}}{x_{i}+y_{j}},\,\,\, i,j=1,\dots,n,\label{eq:con-eigenvalue problem}\end{equation}
and $a_{i}=\sqrt{\alpha_{i}}/\gamma_{i}$, $b_{j}=\sqrt{\overline{\alpha}_{j}}$,
$x_{i}=\gamma_{i}^{-1}$, $y_{j}=-\overline{\gamma_{j}}.$ The con-eigenvalues
of $C$ are labeled in non-increasing order, $\lambda_{1}\geq\lambda_{2}\geq\dots\geq\lambda_{n}$. 
\item Find the (exactly) $m$ zeros $\eta_{j}$ in the unit disk of the
function \begin{equation}
v(z)=\frac{1}{\lambda_{m}}\sum_{i=1}^{n}\frac{\sqrt{\overline{\alpha_{i}}}\,\overline{u_{i}}}{1-\overline{\gamma_{i}}z}.\label{eq: eq for function with zeros}\end{equation}
The fact that there are exactly $m$ zeros in the unit disk, corresponding
to the index $m$ of the con-eigenvalue $\lambda_{m}$, is a consequence
of the AAK theorem. The poles of $r(z)$ are given by the zeros $\eta_{j}$
of $v(z)$.
\item Find the residues $\beta_{m}$ of $r(z)$ by solving the $m\times m$
linear system\begin{equation}
\sum_{i=1}^{m}\frac{1}{1-\eta_{i}\overline{\eta_{j}}}\,\beta_{i}=\sum_{i=1}^{n}\frac{\alpha_{i}}{1-\gamma_{i}\overline{\eta_{j}}}.\label{eq:least squares for weights}\end{equation}

\end{enumerate}
The $L^{\infty}$-error of the resulting rational approximation $r(e^{2\pi ix})$
satisfies $\left\Vert f-r\right\Vert _{\infty}\approx\lambda_{m}$,
and is close to the best error in the $L^{\infty}$-norm achievable
by rational functions with no more than $m$ poles in the unit disk.
Hence, we are led to the problem of computing, to high relative accuracy,
small con-eigenvalues and the associated con-eigenvectors of positive-definite
Cauchy matrices.

In many applications it is natural (and advisable) to maintain the
poles $\gamma_{j}$ in the form $\gamma_{j}=\exp\left(-\tau_{j}\right)$
(see e.g., \cite{BEY-MON:2005,BEY-MON:2010}). As we explain in Section~\ref{sec:Examples-of-optimal},
this is particularly important if the function $f(e^{2\pi ix})$ has
singularities or sharp transitions. The advantage of this form is
that, on a logarithmic scale, the nodes are well separated. In such
cases, our algorithm computes the new poles $\eta_{i}=\exp\left(-\zeta_{i}\right)$
with nearly full precision in the exponents $\zeta_{i}$, i.e., $\left|\hat{\zeta_{i}}-\zeta_{i}\right|/\left|\zeta_{i}\right|$
is close to machine precision even if $\zeta_{i}$ is close to zero.
\begin{rem}
In practice, finding the new poles $\eta_{i}$ using the formula for
$v(z)$ in (\ref{eq: eq for function with zeros}) is ill-advised,
since evaluating $v(z)$ in this form could result in loss of significant
digits through catastrophic cancellation. Indeed, it turns out (see
\cite[Section 6]{BEY-MON:2005} and \cite{HA-BE-MO:2012}) that the
values of the con-eigenvector components satisfy $u_{i}=\sqrt{\alpha_{i}}v\left(\gamma_{i}\right)$,
$i=1,\ldots,n$. It then follows that the sum (\ref{eq: eq for function with zeros})
must suffer cancellation of about $\log_{10}\left(\lambda_{m}^{-1}\right)$
digits if $v\left(\gamma_{i}\right)$ and $v\left(z\right)$ are of
comparable size (note that $\lambda_{m}$ controls the approximation
error and, thus, is necessarily small). On the other hand, the function
values $v\left(\gamma_{i}\right)=u_{i}/\sqrt{\alpha_{i}}$, $i=1,\ldots,n$,
along with the $n$ poles $1/\overline{\gamma_{i}}$ of $v(z)$, completely
determine (\ref{eq: eq for function with zeros}). Since the poles
$\gamma_{i}$ of $f(z)$ are often close to the poles $\eta_{i}$
of $r(z)$, we have observed that evaluating $v(z)$ by using rational
interpolation via continued fractions with the known values $v\left(\gamma_{i}\right)$
allows us to obtain the new poles $\eta_{i}$ with nearly full precision.
In particular, an approximation $\widetilde{v}(z)$ to $v(z)$ is
computed via continued fractions, \begin{equation}
\widetilde{v}(z)=\frac{a_{1}}{1+a_{2}\left(z-\gamma_{1}\right)/\left(1+a_{3}\left(z-\gamma_{2}\right)/\left(1+\cdots\right)\right)},\label{eq:continued fraction-1}\end{equation}
where the coefficients $a_{j}$ are determined from the interpolation
conditions $\widetilde{v}(\gamma_{i})=v\left(\gamma_{i}\right)$.
If the poles $\gamma_{i}$ are given in the form $\gamma_{i}=\exp\left(-\tau_{i}\right)$,
we find that Newton's method on $\widetilde{v}\left(\exp\left(-\eta\right)\right)$
yields the new poles $\eta_{i}=\exp\left(-\zeta_{i}\right)$ with
nearly full relative accuracy even when $\text{Re}\left(\zeta_{i}\right)\ll1$;
see Section~\ref{sec:Examples-of-optimal} for more details (achieving
high relative accuracy also requires slightly modifying the recursion
formulas for the continued fraction coefficients $a_{i}$). A more
detailed description of the root-finding algorithm may be found in
\cite{HA-BE-MO:2012}.
\end{rem}

\subsection{Accurate con-eigenvalue decompositions of positive-definite matrices
with RRDs\label{sub:Derivation-of-algorithm}}

The con-eigenvalue problem for a positive-definite Cauchy matrix $C_{ij}=a_{i}b_{j}/\left(x_{i}+y_{j}\right)$
reduces to an eigenvalue problem, \begin{equation}
\overline{C}Cu=\lambda\overline{C}\bar{u}=\left|\lambda\right|^{2}u.\label{eq:con-eigenvalue problem 3}\end{equation}
We first discuss a somewhat more general problem of computing accurate
eigenvalues and eigenvectors of matrices of the form $\overline{A}A$,
where we assume that $A$ has a factorization $A=XD^{2}X^{*}$, with
$X$ a (well-conditioned) $n\times m$ matrix ($m\leq n$) and $D$
an $m\times m$ diagonal matrix with positive, non-increasing diagonal
entries. The rectangular form of the factorization, $m\le n$, will
be important in the sequel. 

Let us define the $m\times m$ matrix $G=D\left(X^{\text{T}}X\right)D$,
and consider its SVD, $G=W\Sigma V^{*}$. Then $G^{*}G=V\Sigma^{2}V^{*}$,
and the $i$th right singular vector ($1\leq i\leq m$), $v_{i}=V(:,i)$,
satisfies $\left(DX^{*}\overline{X}D\right)\left(DX^{\text{T}}XD\right)v_{i}=\Sigma_{ii}^{2}v_{i}$.
It then follows that $z_{i}=XDv_{i}$ is an eigenvector of $A\overline{A}$
with eigenvalue $\Sigma_{ii}^{2}$, since\begin{eqnarray*}
A\overline{A}z_{i} & = & \left(XD^{2}X^{*}\right)\left(\overline{X}D^{2}X^{\text{T}}\right)z_{i}=\\
 & = & XD\left(DX^{*}\overline{X}D\right)\left(DX^{\text{T}}XD\right)v_{i}=\Sigma_{ii}^{2}XDv_{i}=\Sigma_{ii}^{2}z_{i}.\end{eqnarray*}
and, thus, $\overline{z_{i}}=\overline{X}D\overline{v_{i}}$ is an
eigenvector of $\overline{A}A$. To summarize: given the decomposition
$A=XD^{2}X^{*}$, an eigenvector $z_{i}$ ($i\leq m$) of $\overline{A}A$
is given by $\overline{z_{i}}=\overline{X}\left(D\overline{v_{i}}\Sigma_{ii}^{-1/2}\right)$,
where $v_{i}$ is the $i$th right singular vector of the $m\times m$
matrix $G=D\left(X^{\text{T}}X\right)D$. Here $\Sigma_{ii}$ is the
$i$th singular value of $G$, and the $i$th con-eigenvalue of $A$.
Let us now present an algorithm for accurately computing the con-eigenvalues
and con-eigenvectors of $A$ (its derivation also relies on the background
material collected in Section~\ref{sec:Background-on-algorithms}).

\begin{algorithm}[H]
\caption{\label{Flo:con-eig comp of XDX^*}$\text{ConEig\_RRD}\left(X,D\right)$
computes accurate con-eigenvalue decomposition of $XD^{2}X^{*}$.
Input: rank-revealing factors $X$ and $D$ (of dimensions $n\times m$
and $m\times m$), where the diagonal of $D>0$ is decreasing. Output:
$m$ con-eigenvalues/con-eigenvectors of $XDX^{*}$, contained in
$\Sigma$ and $T$.}

$\left(\Sigma,T\right)\leftarrow\text{ConEig\_RRD}\left(X,D\right)$
\begin{lyxcode}
1.~Form~$G=D\left(X^{\text{T}}X\right)D$

2.~Compute~QR~factors~$\left(Q,R\right)\leftarrow\text{Householder\_QR}$~of~$G$~($G=QR$),~with~optional~pivoting~(see~Section~\ref{sub:QR decomposition})

3.~Compute~the~SVD~factors~$\left(U_{l},\Sigma,U_{r}\right)\leftarrow\text{Jacobi}\left(R\right)$~of~$R$~($R=U_{l}\Sigma U_{r}^{*}$),~using~one-sided~Jacobi,~applied~from~the~left~(see~Section~\ref{sub:Modified-one-sided-Jacobi})

4.~Compute~$R_{1}=D^{-1}RD^{-1}$,~$X_{1}=D^{-1}U_{l}\Sigma^{1/2}$,~and~$Y_{1}=R_{1}^{-1}X_{1}$~(see~(\ref{eq:relation between right/left sing vecs})~below)

5.~Form~the~matrix~of~con-eigenvectors~$T=\overline{XY_{1}}$,~and~output~con-eigenvalues~$\Sigma$~and~con-eigenvectors~$T$
\end{lyxcode}

\end{algorithm}

Importantly, for Cauchy matrices ($A=C$) the elements of $D$ decay
exponentially fast, and it would appear that computing the con-eigenvectors
$\overline{z_{i}}=\overline{X}D\overline{v_{i}}/\Sigma_{ii}^{1/2}$
might lead to wildly inaccurate results even if the right singular
vector of $G$, $v_{i}$, is computed accurately. However, as we show
in Section~\ref{sec:Accuracy-and-perturbation}, Algorithm~\ref{Flo:con-eig comp of XDX^*}
achieves high accuracy despite the extreme ill-conditioning of $D$.
The key reason is that the right singular vector $v_{i}$, corresponding
to the singular value $\Sigma_{ii}$, scales like $\left|v_{i}\left(j\right)\right|\leq c_{V}\min\left(D_{jj}/\Sigma_{ii}^{1/2},\Sigma_{ii}^{1/2}/D_{jj}\right)$,
and the computed singular vector $\widehat{v_{i}}$ is accurate relative
to the scaling in $D$ and $\Sigma$ in the sense that \[
\left|v_{i}\left(j\right)-\widehat{v_{i}}\left(j\right)\right|\leq\min\left\{ \frac{D_{jj}}{\sqrt{\Sigma_{ii}}},\frac{\sqrt{\Sigma_{ii}}}{D_{jj}}\right\} \mathcal{O}\left(\epsilon\right).\]
For Cauchy matrices, the quantity $\min\left(D_{jj}/\Sigma_{ii}^{1/2},\Sigma_{ii}^{1/2}/D_{jj}\right)$
decreases exponentially fast away from the diagonal $i=j$.

Let us give an informal explanation of the reasons why Algorithm~\ref{Flo:con-eig comp of XDX^*}
yields accurate results. As discussed in Section~\ref{sub:QR decomposition},
the QR Householder algorithm computes an accurate rank-revealing decomposition
of $G=QR$. It turns out (see Lemma~\ref{lem:R_0, for main alg proof})
that $R$ may be factored as $R=D^{2}R_{0}$, where $R_{0}$ is graded
relative to $D$ in the sense that$\left\Vert DR_{0}D^{-1}\right\Vert $
and $\left\Vert DR_{0}^{-1}D^{-1}\right\Vert $ are not too large,
as long as the $n$ leading principal minors of $X^{\text{T}}X$ are
well-conditioned. Therefore, from the discussion in Section~\ref{sub:Modified-one-sided-Jacobi}
(see in particular Theorem~\ref{thm:Jacobi bound one-sided scaling}),
the one-sided Jacobi algorithm computes the $i$th left singular vector
$u_{i}$ of $R$ accurately relative to the scaling $\min\left\{ D_{jj}/\Sigma_{ii}^{1/2},\Sigma_{ii}^{1/2}/D_{jj}\right\} $.
It follows that $D^{-1}u_{i}\Sigma_{ii}^{1/2}$ may also be computed
accurately. Finally, since the $i$th right singular vector $v_{i}$
of $R$ (and $G$) satisfies \begin{eqnarray}
Dv_{i}\Sigma_{ii}^{-1/2} & = & DR^{-1}u_{i}\Sigma_{ii}^{1/2}\nonumber \\
 & = & \left(DR_{0}D^{-1}\right)^{-1}\left(D^{-1}u_{i}\Sigma_{ii}^{1/2}\right),\label{eq:relation between right/left sing vecs}\end{eqnarray}
the con-eigenvector $\overline{z_{i}}=\overline{X}\left(D\overline{v_{i}}\Sigma_{ii}^{-1/2}\right)$
may be computed accurately, as long as $DR_{0}D^{-1}$ is computed
accurately and is well-conditioned (we show this is the case if $n$
leading principal minors of $X^{\text{T}}X$ are well-conditioned).
The last step in Algorithm~\ref{Flo:con-eig comp of XDX^*} uses
the approach in \cite{DRM-VES:2007a} for computing highly accurate
right singular vectors via solving a triangular linear system of equations.
\begin{rem}
To obtain optimal rational approximations (see Section~\ref{sub:Reduction-procedure}),
we need to compute small con-eigenvalues (and the associated con-eigenvectors)
of Cauchy matrices of the slightly different form, $C_{ij}=\sqrt{\alpha_{i}}\sqrt{\overline{\alpha}_{j}}/\left(1-\gamma_{i}\overline{\gamma_{j}}\right)$,
i.e., with $a_{i}=\sqrt{\alpha_{i}}/\gamma_{i}$, $b_{j}=\sqrt{\overline{\alpha}_{j}}$,
$x_{i}=\gamma_{i}^{-1}$, and $y_{j}=-\overline{\gamma_{j}}$. The
same reasoning as in \cite{DEMMEL:1999} shows that the Cholesky computation
of $C$ (see Section~\ref{sub:Accurate-factorization-of}) is performed
with high relative accuracy, as long as the differences $\gamma_{j}^{-1}-\overline{\gamma_{i}}$
are computed with high relative accuracy. As explained in the next
section, $\gamma_{j}^{-1}-\overline{\gamma_{i}}$ may be accurately
computed if $\gamma_{i}$ is of the form $\gamma_{i}=\exp\left(-\tau_{i}\right)$,
where the exponents $\tau_{i}$ are known accurately (see Section~\ref{sec:Examples-of-optimal}
for examples).%
{}
\end{rem}
~~
\begin{rem}
Computing the normalized eigenvector $u$ via (\ref{eq:con-eigenvalue problem 3})
determines the con-eigenvector, the solution of (\ref{eq:con-eigenvalue problem}),
only up to an unknown phase factor $e^{-i\phi/2}$. Indeed, given
any solution $\lambda$ and $u$ of (\ref{eq:con-eigenvalue problem 3})
and an arbitrary phase factor $e^{-i\phi}$, it is easy to see that
$\lambda e^{-i\phi}$ and $ue^{-i\phi/2}$ also satisfy (\ref{eq:con-eigenvalue problem}).
Let us now determine the phase $\phi$ so that the con-eigenvalues
$\lambda$ are positive. To do so, we compute the usual inner product
$\left(C\left(ue^{-i\phi/2}\right),ue^{-i\phi/2}\right)=\lambda\left(\overline{u}e^{i\phi/2},ue^{-i\phi/2}\right)$
and choose $\phi$ so that $\lambda>0$. Since $C$ is a positive-definite
matrix, it follows that $\left(\overline{u}e^{i\phi/2},ue^{-i\phi/2}\right)>0$.
From this we obtain the phase factor as $e^{i\phi}=\left(u,\overline{u}\right)/\left|\left(u,\overline{u}\right)\right|$.
\end{rem}

\subsection{Accurate con-eigenvalue decompositions of positive-definite Cauchy
matrices\label{sub:Accurate-con-eigenvalue-decompositions}}

If $A=C$ is a positive-definite Cauchy matrix, then the modified
GECP algorithm in \cite{DEMMEL:1999} computes the Cholesky decomposition
$C=\left(PL\right)D^{2}\left(PL\right)^{*}$ with high relative accuracy
(see Section~\ref{sub:Highly-accurate-SVDs}). Therefore, Algorithm~\ref{Flo:con-eig comp of XDX^*}
for the eigenvalue problem of $\overline{C}C$ may be used, with $X=PL$,
to compute all the eigenvalues and eigenvectors (and, therefore, the
con-eigenvectors and con-eigenvalues of $C$).

For our purposes, we are only interested in computing a single con-eigenvector
with associated con-eigenvalue of approximate size $\delta$ (see
Section~\ref{sub:Reduction-procedure}). However, the diagonal elements
of $D$ may be many orders of magnitude smaller than $\delta$, and
it is then natural to expect that, by computing a partial Cholesky
decomposition of $C$, we may obtain the $i$th con-eigenvector in
much fewer than $\mathcal{O}\left(n^{3}\right)$ operations. In this
case, we stop Demmel's algorithm for the Cholesky decomposition of
$C$ once the diagonal elements $D_{ii}^{2}$ are small with respect
to the product of $\delta^{2}$ and the machine round-off $\epsilon$,
that is, as soon as $D_{mm}^{2}\leq\delta^{2}\epsilon$ for some $m$
(notice that complete pivoting ensures that the diagonal elements
$D_{ii}$ are non-increasing). We then obtain $C\approx\widetilde{C}=\left(\widetilde{P}\widetilde{L}\right)\widetilde{D}^{2}\left(\widetilde{P}\widetilde{L}\right)^{*}$,
where $\widetilde{P}$ is an $m\times n$ matrix, $\widetilde{L}$
is an $n\times m$ matrix and $\widetilde{D}$ is a diagonal $m\times m$
matrix. Algorithms~\ref{Flo:algorithm for pivot order-1}~and~\ref{Flo:Cauchy factorization alg-1}
contain pseudo-code for computing $\widetilde{L}$, $\widetilde{D}$,
and $\widetilde{P}$. In the pseudo-code $I\left(n,m\right)$ denotes
the first $m\leq n$ columns of the $n\times n$ identity matrix.

\begin{algorithm}[H]
\caption{\label{Flo:algorithm for pivot order-1}$\text{Pivot\_Order}\left(a,b,x,y,\delta\right)$
pre-computes pivot order for Cholesky factorization of $n\times n$
positive-definite Cauchy matrix $C_{ij}=a_{i}b_{j}/\left(x_{i}+y_{j}\right).$
Input: $a$, $b$, $x$, and $y$ defining $C_{ij}=a_{i}b_{j}/\left(x_{i}+y_{j}\right)$,
and target size $\delta$ of con-eigenvalue. Output: correctly pivoted
vectors $a$, $b$, $x$, and $y$, truncation size $m$, and $m\times n$
permutation matrix $\widetilde{P}$}

$\left(a,b,x,y,\widetilde{P},m\right)\leftarrow\text{Pivot\_Order}\left(a,b,x,y,\delta\right)$
\begin{lyxcode}
Form~vector~$g_{i}:=a_{i}b_{i}/(x_{i}+y_{i})$,~$i=1,\ldots,n$

Set~cutoff~for~GECP~termination:~$\eta:=\epsilon\delta^{2}$

Initialize~permutation~matrix~($n\times n$~identity):~$\widetilde{P}=I\left(n,n\right)$

Compute~correctly~pivoted~vectors:

~$m:=1$

~~~while~$\left|g\left(m\right)\right|\geq\eta$~or~$m=n-1$
\begin{lyxcode}
~~Find~$m\leq l\leq n$~such~that~$\left|g(l)\right|=\max\left|g\left(m:n\right)\right|$

~~Swap~elements:

~~~$g(l)\leftrightarrow g(m)$,~$x(l)\leftrightarrow x(m)$~,~$y(l)\leftrightarrow y(m)$

~~~$a(l)\leftrightarrow a(m)$,$b(l)\leftrightarrow b(m)$

~~Swap~rows~of~permutation~matrix:

~~~$\widetilde{P}(l,:)\leftrightarrow\widetilde{P}(m,:)$

~~Update~diagonal~of~Schur~complement:

~~~$g(m+1:n):=\left(x\left(m+1:n\right)-x(m)\right)/\left(y\left(m+1:n\right)-y(m)\right)g(m+1:n)$

~~Increment~iteration~count:

~~~$m:=m+1$
\end{lyxcode}

Output~$a,b,x,y,\widetilde{P}\left(1:m,n\right),m$
\end{lyxcode}

\end{algorithm}

\begin{algorithm}[H]
\caption{\label{Flo:Cauchy factorization alg-1}$\text{Cholesky\_Cauchy}\left(x,y,a,b,\delta\right)$
computes partial Cholesky factorization of positive-definite Cauchy
matrix $C_{ij}=a_{i}b_{j}/\left(x_{i}+y_{j}\right)$. Input: $a$,
$b$, $x$, and $y$ defining $C_{ij}=a_{i}b_{j}/\left(x_{i}+y_{j}\right)$,
and target size $\delta$ of con-eigenvalue. Output: $n\times m$
matrix $\widetilde{L}$, $m\times m$ matrix $\widetilde{D}$, and
permutation $m\times n$ matrix $\widetilde{P}$ in partial Cholesky
factorization.}

$\left(\widetilde{L},\widetilde{D},\widetilde{P}\right)\leftarrow\text{Cholesky\_Cauchy}\left(a,b,x,y,\delta\right)$
\begin{lyxcode}
Compute~pivoted~vectors~and~matrix~size~$m$~(Algorithm~\ref{Flo:algorithm for pivot order-1}):

~~$\left(a,b,x,y,\widetilde{P},m\right)\leftarrow\text{Pivot\_Order}\left(a,b,x,y,\delta\right)$

Initialize~generators:

~~$\alpha:=a$,~$\beta:=b$

Compute~first~column~of~Schur~complement:

~~$G\left(:,1\right):=\alpha*\beta/\left(x+y\right)$

~~for~$k=2,m$
\begin{quotation}
Update generators:\end{quotation}
\begin{lyxcode}
~~~~~$\alpha\left(k:n\right):=\alpha\left(k:n\right)*\left(x\left(k:n\right)-x\left(k-1\right)\right)/\left(x\left(k:n\right)+y\left(k-1\right)\right)$

\textrm{~~~~}~~~\textrm{$\beta\left(k:n\right):=\beta\left(k:n\right)*\left(y\left(k:n\right)-y\left(k-1\right)\right)/\left(y\left(k:n\right)+x\left(k-1\right)\right)$}

~~~Extract~$k$th~column~for~Cholesky~factors:

~~~~~$G\left(k:n,k\right):=\alpha\left(k:n\right)*\beta\left(k:n\right)/\left(x\left(k:n\right)+y\left(k:n\right)\right)$~
\end{lyxcode}
Output~partial~Cholesky~factors:

~~$\widetilde{D}=\text{diag}\left(G(1:n,1:m\right)^{1/2}$,~$\widetilde{L}=\text{tril}\left(G(1:n,1:m)\right)\widetilde{D}^{-2}+I\left(n,m\right)$,~$\widetilde{P}$
\end{lyxcode}

\end{algorithm}

Once the partial Cholesky decomposition $C\approx\widetilde{C}=\left(\widetilde{P}\widetilde{L}\right)\widetilde{D}^{2}\left(\widetilde{P}\widetilde{L}\right)^{*}$
is computed, Algorithm~\ref{Flo:con-eig comp of XDX^*} for the eigenvalue
problem of $\overline{\widetilde{C}}\widetilde{C}$ may then be used,
with $X=\widetilde{P}\widetilde{L}$ and $D=\widetilde{D}$, to compute
accurate con-eigenvalues and con-eigenvectors of $\widetilde{C}$
(see Theorem~\ref{thm:truncation theorem}). Since the con-eigenvalues
decay exponentially fast, the complexity of this algorithm is $\mathcal{O}\left(n\left(\log(\delta\epsilon)^{-1}\right)^{2}\right)$
operations. Therefore, when used in the reduction procedure outlined
in Section~\ref{sub:Reduction-procedure}, the near optimal rational
approximation may be obtained by computing the SVD of a matrix that
is roughly twice the size of the optimal number of poles. The pseudo-code
is given in Algorithm~\ref{Flo:Accurate-con-eigenvalue-decomposition}. 

\begin{algorithm}[H]
\caption{\label{Flo:Accurate-con-eigenvalue-decomposition}$\text{Con\_Eigvector}\left(a,b,x,y,\delta\right)$
computes accurate con-eigenvalue decomposition of positive-definite
Cauchy matrix $C_{ij}=a_{i}b_{j}/\left(x_{i}+y_{j}\right).$ Input:
$a$, $b$, $x$, and $y$ defining $C_{ij}=a_{i}b_{j}/\left(x_{i}+y_{j}\right)$,
and target size $\delta$ of con-eigenvalue. Output: con-eigenvalues
lager than $\delta$, and associated con-eigenvectors.}

$\left(\Sigma,T\right)\leftarrow\text{Con\_Eigvector}\left(a,b,x,y,\delta\right)$
\begin{lyxcode}
1.~Compute~partial~Cholesky~factors~$\left(L,D,P\right)\leftarrow\text{Cholesky\_Cauchy}\left(a,b,x,y,\delta\right)$~(Algorithm~\ref{Flo:Cauchy factorization alg-1})~and~set~$X=PL$

2.~Compute~con-eigenvalues~and~con-eigenvectors~$\left(\Sigma,T\right)\leftarrow\text{ConEig\_RRD}\left(X,D\right)$~using~Algorithm~\ref{Flo:con-eig comp of XDX^*}

3.~Select~largest~$l$~such~that~$\Sigma_{ll}\geq\delta$~and~output~$\Sigma\left(1:l,1:l\right)$,~$T\left(1:n,1:l\right)$
\end{lyxcode}

\end{algorithm}

\begin{rem}
\label{rem:In-applications-involving-functions with singuularities}In
applications involving functions $f\left(e^{2\pi ix}\right)$ with
singularities or sharp transitions, the poles $\gamma_{i}$ are given
in the form $\gamma_{i}=\exp\left(-\tau_{i}\right)$, where $\mathcal{R}e\tau_{j}>0$
and $0\le\mathcal{I}m\tau_{j}<2\pi$ and the exponents $\tau_{i}$
are known with high relative accuracy. Indeed, this form naturally
arises either via a discretization of an integral (see \cite{BEY-MON:2005,BEY-MON:2010})
or as a result of an intermediate computation as in \cite{HA-BE-MO:2012}.
This leads us to modify Algorithms~\ref{Flo:algorithm for pivot order-1}~and~\ref{Flo:Cauchy factorization alg-1}
so that high relative accuracy is achieved for poles of this form.
In particular, we modify formulas (\ref{eq:schur complement recursion}),
(\ref{eq:recursion for schur complement}) and (\ref{eq:recursion for alpha's, beta's})
in Section~\ref{sec:Background-on-algorithms}. For example, the
formula for $\alpha_{i}^{(k)}$ in (\ref{eq:recursion for alpha's, beta's})
involves computing \begin{eqnarray*}
\frac{x_{j}-x_{k-1}}{x_{j}+y_{k-1}} & = & \frac{\gamma_{j}^{-1}-\gamma_{k-1}^{-1}}{\gamma_{j}^{-1}-\overline{\gamma_{k-1}}}=\frac{1-\exp\left(-\tau_{j}+\tau_{k-1}\right)}{1-\exp\left(-\tau_{j}-\overline{\tau_{k-1}}\right)}.\end{eqnarray*}
The simple modification is to use the Taylor expansion $1-\exp\left(z\right)\approx z+z^{2}/2+\dots$
if $\left|z\right|$ is small. The other formulas in (\ref{eq:schur complement recursion}),
(\ref{eq:recursion for schur complement}) and (\ref{eq:recursion for alpha's, beta's})
are modified in a similar fashion, allowing the LDU factorization
of $C$ to be computed with high relative accuracy.

In Section \ref{sec:Examples-of-optimal}, we consider a case where
the absolute values of many poles agree with $1$ to twelve digits
(i.e., the poles $\gamma_{i}$ satisfy$\left|\gamma_{i}\right|\approx0.999999999999xxxx$).
\end{rem}

\section{Examples of optimal rational approximations\label{sec:Examples-of-optimal}}

In this section, we consider some applications of the reduction algorithm.

\subsection{Optimal rational approximations of functions with singularities }

Using the reduction algorithm, as well as tools developed in \cite{BEY-MON:2005,BEY-MON:2010},
we construct a (near) optimal rational approximation of a (piecewise
smooth) function $f$ with a finite number of isolated integrable
singularities. For simplicity, we assume that singularities of $f$
are at two points, $0$ and $x_{0}$. 

Performing integration by parts $L$ times on the expression for the
Fourier coefficients,\[
\hat{f}_{n}=\int_{0}^{1}f(x)e^{2\pi inx}dx=\int_{0}^{x_{0}}f(x)e^{2\pi inx}dx+\int_{x_{0}}^{1}f(x)e^{2\pi inx}dx,\]
we obtain 

\begin{eqnarray*}
\widehat{f}_{n} & =h_{n} & +\frac{\left(-1\right)^{L}}{\left(2\pi in\right)^{L}}\int_{0}^{x_{0}}f^{(L)}(x)e^{2\pi inx}dx+\frac{\left(-1\right)^{L}}{\left(2\pi in\right)^{L}}\int_{x_{0}}^{1}f^{(L)}(x)e^{2\pi inx}dx,\end{eqnarray*}
where \[
h_{n}=\sum_{p=1}^{L}\frac{\left(-1\right)^{p}}{\left(2\pi in\right)^{p}}\left(e^{2\pi inx_{0}}F^{(p-1)}\left(x_{0}\right)+F^{(p-1)}\left(0\right)\right),\]
$F^{(p)}\left(x\right)=f^{(p)}\left(x^{+}\right)-f^{(p)}\left(x^{-}\right)$
and $x^{+}$, $x^{-}$ indicate directional limits. As the first step
in constructing a (near) optimal rational approximation to $f$, we
subtract the leading $L$ terms of the asymptotic expansion of $\widehat{f}_{n}$
and consider $g_{n}=\widehat{f}_{n}-h_{n}$. Since $g_{n}$ decays
like $O\left(1/n^{L+1}\right)$, it is sufficient to use the algorithm
in \cite{BEY-MON:2005,BEY-MON:2010} to construct an approximation
\begin{equation}
\left|g_{n}-\sum_{m=1}^{M}w_{m}e^{-\mu_{m}n}\right|\le\epsilon,\,\,\,\, n\geq1.\label{eq:g_n}\end{equation}
This algorithm requires quadruple precision for computing small singular
values of a Hankel matrix but, due to the fast decay of $g_{n}$,
the matrix is small so that the computational cost is insignificant.
An alternative method for obtaining (\ref{eq:g_n}) based on rational
representations of B-splines requires only double precision and will
appear elsewhere \cite{D-B-H-M:2012p}. For $h_{n}$ we use a discretization
of the integral representation for $1/n^{p}$ in \cite{BEY-MON:2010}
to obtain \begin{equation}
\left|\frac{1}{n^{p}}-\sum_{m=-M_{1}}^{M_{2}}a_{m,p}e^{-\tau_{m}n}\right|\le\epsilon,\,\,\,\,1\le p\le L,\,\,\,\,1\leq n,\label{eq:integral for 1/n^p}\end{equation}
where $\tau_{m}=e^{hm}$, $a_{m,p}=\frac{h}{\left(p-1\right)!}e^{phm}$
and $h$ is the step size used in the discretization. Results in \cite{BEY-MON:2010}
imply that there are at most $\mathcal{O}\left(\left(\log\epsilon^{-1}\right)^{2}\right)$
terms in the approximation of $1/n^{p}$ for a given accuracy $\epsilon$,
for all $n\geq1$. Note that when $m<0$ the nodes $\gamma_{m}=e^{-e^{hm}}\approx1-e^{hm}$
are very close to one. 

Thus, we arrive at \begin{equation}
\left|h_{n}-\sum_{m=-M_{1}}^{M_{2}}a_{m}e^{-\left(\tau_{m}+2\pi ix_{0}\right)n}-\sum_{m=-M_{1}}^{M_{2}}b_{m}e^{-\tau_{m}n}\right|\le2\epsilon,\label{eq:h_n}\end{equation}
where\[
a_{m}=\sum_{p=1}^{L}\frac{1}{\left(-2\pi i\right)^{p}}F^{(p-1)}\left(x_{0}\right)a_{m,p},\,\,\,\, b_{m}=\sum_{p=1}^{L}\frac{1}{\left(-2\pi i\right)^{p}}F^{(p-1)}\left(0\right)a_{m,p}.\]
Combining the approximations (\ref{eq:g_n}) and (\ref{eq:h_n}),
we obtain the suboptimal approximation \begin{equation}
\left|\widehat{f}_{n}-\sum_{m=1}^{M}w_{m}e^{-\mu_{m}n}-\sum_{m=-M_{1}}^{M_{2}}a_{m}e^{-\left(\tau_{m}+2\pi ix_{0}\right)n}-\sum_{m=-M_{1}}^{M_{2}}b_{m}e^{-\tau_{m}n}\right|\le3\epsilon,\label{eq:approx of f_n, total}\end{equation}
where the number of terms is excessive (for the accuracy $3\epsilon$).
We now use the reduction algorithm on (\ref{eq:approx of f_n, total})
to obtain a nearly optimal number of terms to approximate the Fourier
coefficients $f_{n}$ for $n\ge1$. This, in turn leads to a near
optimal rational approximation to $f(x)$ with a nearly equioscillating
error. 

As an example, we apply this procedure to \begin{equation}
f(x)=\begin{cases}
\sin(4/3\pi x), & 0\le x\le3/4\\
0 & 3/4<x\le1\end{cases}\label{eq:BrokenSine}\end{equation}
Choosing the parameters $M_{1}=200$, $M_{2}=10$, and $h=.316707$
in (\ref{eq:approx of f_n, total}) (see \cite{BEY-MON:2010} for
how to select the parameters) yields a sub-optimal approximation containing
$426$ pairs of conjugate-reciprocal poles $\gamma_{j}=e^{-\tau_{j}}$,
which approximates $f\left(x\right)$ in the $L^{\infty}$ norm with
error $\approx5\times10^{-14}$. We note that many of the poles are
extremely close to the unit disk (the magnitudes $\left|\gamma_{i}\right|\approx.999999999999xxxx$
of over a dozen poles agree with $1$ to twelve digits).

We apply the reduction algorithm using the approximation error $\delta=10^{-13}$
(thus, the Cholesky decomposition algorithm \ref{Flo:Cauchy factorization alg-1}
is truncated once the diagonal elements are smaller than $\epsilon\delta^{2}$,
where $\epsilon$ denotes the machine roundoff). As explained in Remark~\ref{rem:In-applications-involving-functions with singuularities},
Algorithms~\ref{Flo:algorithm for pivot order-1}~and~\ref{Flo:Cauchy factorization alg-1}
are modified to accurately compute the partial Cholesky decomposition
for poles in the form $\gamma_{j}=e^{-\tau_{j}}$. After applying
the reduction algorithm with approximation error $\delta=10^{-13}$,
the resulting rational approximation contains $92$ pairs of conjugate-reciprocal
poles (i.e., about $46$ poles per singularity). The resulting error
is shown in Figure~\ref{fig:rational approx half-sine}.

We note that the only step of the reduction procedure where quadruple
precision is used is in computing the residues $\beta_{j}$ (see Step
$3$ of Section~\ref{sub:Reduction-procedure}). However, using the
techniques described in the background Section~\ref{sub:Accurate-factorization-of}
to factor the $m\times m$ Cauchy matrix, this step takes only $\mathcal{O}\left(m^{2}\right)$
operations, and so does not impact the overall speed of the algorithm
(recall that $m$ denotes the number of reduced poles).

\begin{figure}
\begin{centering}
\includegraphics[scale=0.55]{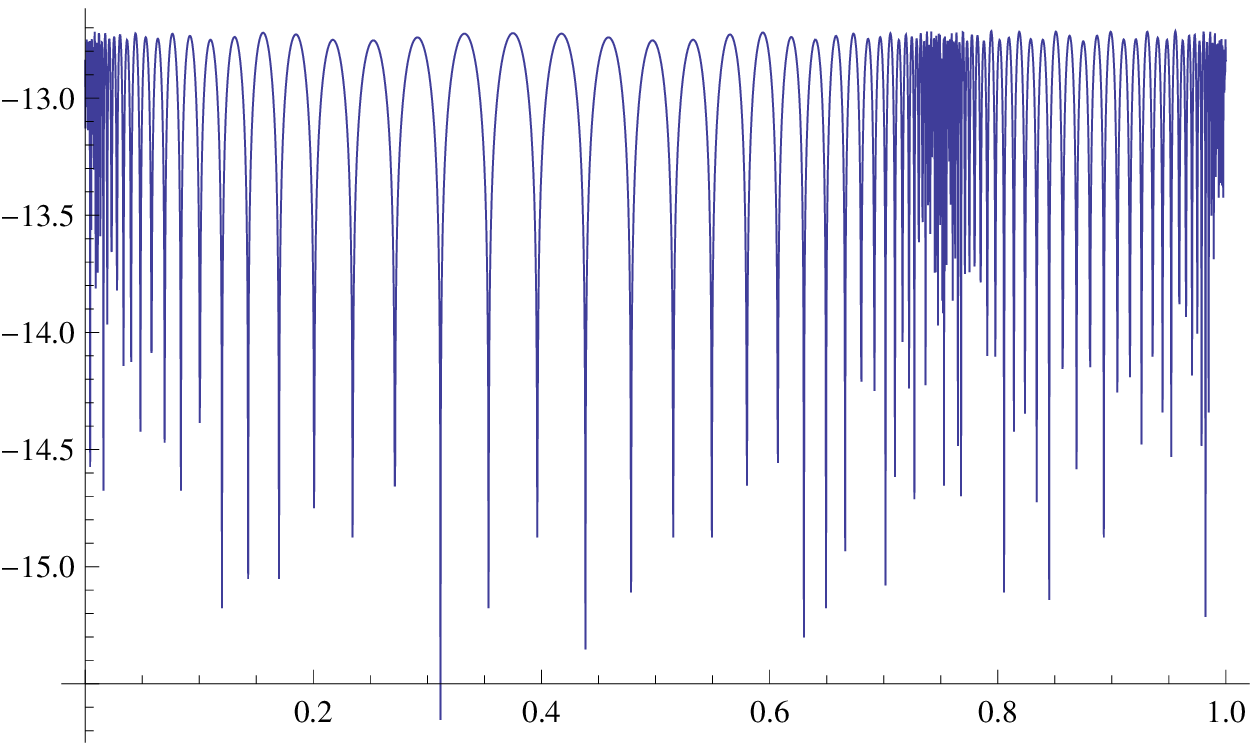} \includegraphics[scale=0.55]{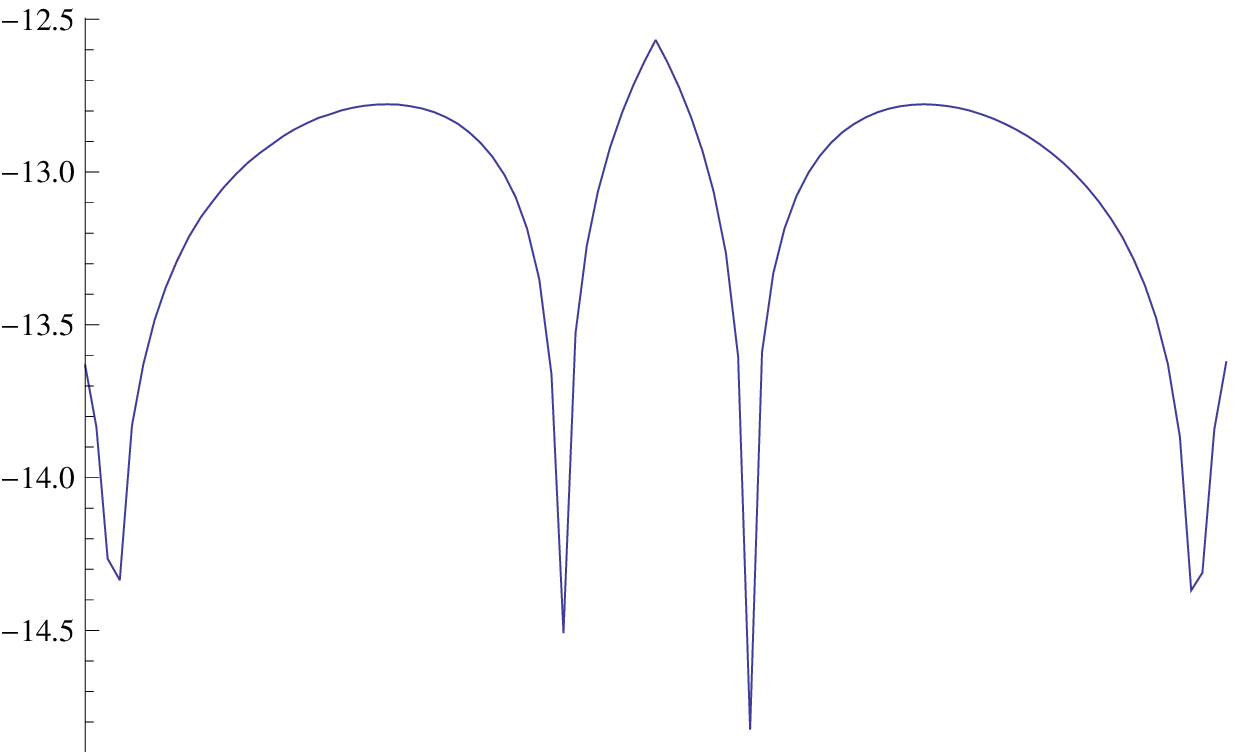}
\par\end{centering}

\begin{centering}
~~~~~~~~~~~~~~~~~~~~~~~~~~~(a)~~~~~~~~~~~~~~~~~~~~~~~~~~~~~~~~~~~~~~~~~~~~~~~~~(b)
~~~~~~~~~~~~~~~~~~~~~~~
\par\end{centering}

\caption{\label{fig:rational approx half-sine} (a) Error of the rational approximation
to $f\left(x\right)$ in (\ref{eq:BrokenSine}). (b) A zoom on a neighbourhood
around one of the singularities $x\in\left(3/4-10^{-12},3/4+10^{-12}\right)$.}

\end{figure}

We find that the exponents, $\eta_{i}$, of the near optimal poles
$\zeta_{i}=\exp\left(-\eta_{i}\right)$ are computed with high relative
accuracy, i.e., \[
\left|\text{Re}\left(\eta_{i}\right)-\text{Re}\left(\widehat{\eta_{i}}\right)\right|\leq\left|\text{Re}\left(\eta_{i}\right)\right|\delta_{1},\,\,\,\,\left|\eta_{i}-\widehat{\eta_{i}}\right|\leq\left|\eta_{i}\right|\delta_{2},\]
where $\delta_{1}\leq1.48\times10^{-13}$ and $\delta_{2}\leq14.87\times10^{-13}$.
As a gauge we used the poles $\zeta_{i}$ obtained in Mathematica$\,^{TM}$
via extended precision arithmetic. We note that the real parts of
some of the exponents $\eta_{i}$ are of size $\left|\text{Re}\left(\eta_{i}\right)\right|\approx10^{-12}$.
{}

\subsection{Solving viscous Burgers' equation}

In \cite{HA-BE-MO:2012} we use the reduction algorithm to solve viscous
Burgers' equation, \begin{equation}
u_{t}-uu_{x}=\nu u_{xx},\,\,\, u(x,0)=u_{0}(x),\,\,\, u(0,t)=u(1,t),\,\,\, x\in[0,1],\,\,\, t\ge0.\label{eq:Burgers}\end{equation}
The solution of this equation develops a shock (or a sharp transition)
on an interval of size $\mathcal{O}\left(\nu\right)$. We approximate
solutions to (\ref{eq:Burgers}) using rational functions of the form
\[
u\left(x,t\right)=\sum_{j=1}^{M_{0}}\frac{\alpha_{j}\left(t\right)}{e^{-2\pi ix}-\gamma_{j}\left(t\right)}+\sum_{j=1}^{M_{0}}\frac{\overline{\alpha_{j}\left(t\right)}}{e^{2\pi ix}-\overline{\gamma_{j}\left(t\right)}}+\alpha_{0}.\]
The key idea is to develop a numerical calculus using the reduction
algorithm. Although operators such as multiplication and convolution
increase the number of poles in the representation, the reduction
algorithm is employed at each stage to keep the number of poles near
optimally small. Overall, about $10^{6}$ applications of the reduction
algorithm were employed to compute the solutions illustrated below,
thus confirming its robustness and efficiency.

Figure~\ref{fig:moving front, t=00003D10^-5, all } shows the computed
solutions $u(x,h_{t}j)$ to (\ref{eq:Burgers}), with the viscosity
$\nu=10^{-5}$, the step size $h_{t}$ and the initial condition $u_{0}(x)=\sin(2\pi x)+1/2\sin(4\pi x)$.
In our reduction procedure, we used the step size of $h_{t}=10^{-5}$
and the error tolerance $\delta=10^{-9}$ (to match the error of our
time discretization). The solution $u(x,h_{t}j)$ is shown for time
steps $t_{j}=h_{t}j$, $j=10^{2},\,10^{4},\,2\times10^{4},\,3\times10^{4},\,5\times10^{4}$.
We see that the solution $u(x,t)$ develops two moving sharp transition
regions, which approach each other and eventually merge into a single
one about $x\approx1/2$. The rational representations of $u(x,t_{j})$
have $4$, $11$, $33$, $29$, and $19$ pairs of conjugate-reciprocal
poles, respectively. It also demonstrates that the transition regions
of $u(x,t)$ occur within intervals of width of $\mathcal{O}\left(\nu\right)$.

\begin{figure}[h]
\begin{centering}
\includegraphics[scale=0.6]{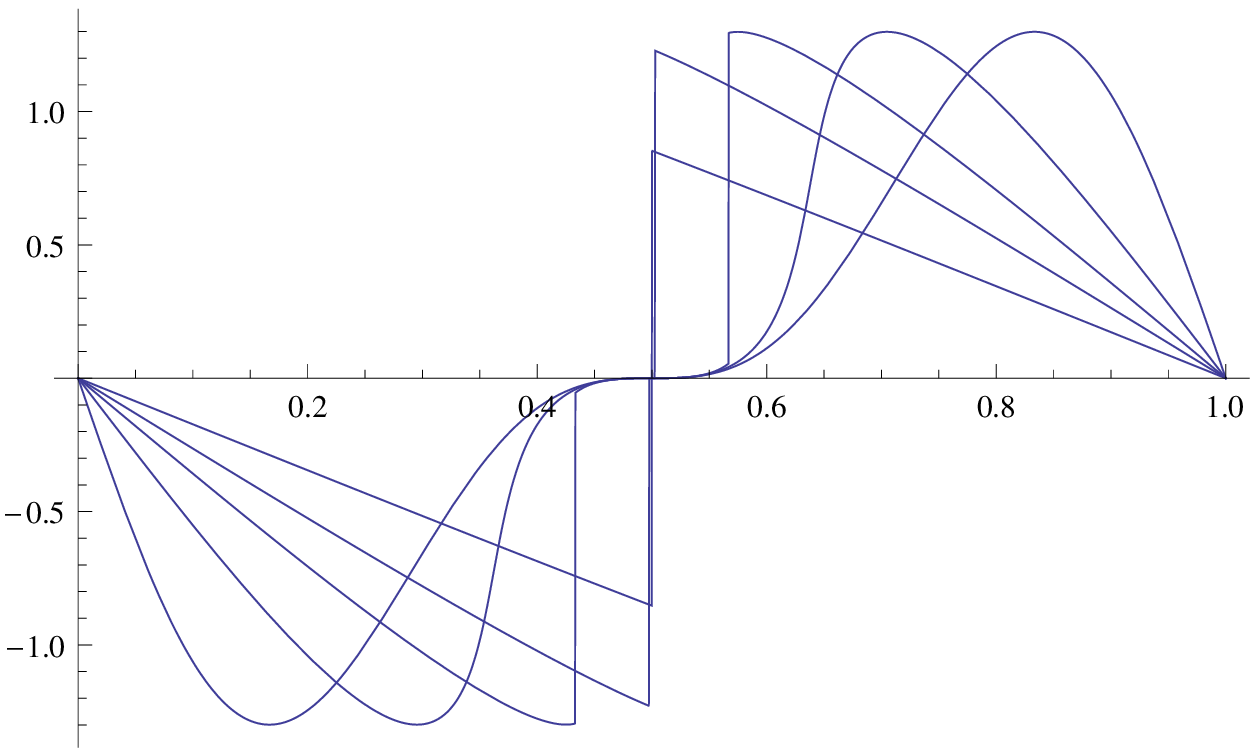}\includegraphics[scale=0.6]{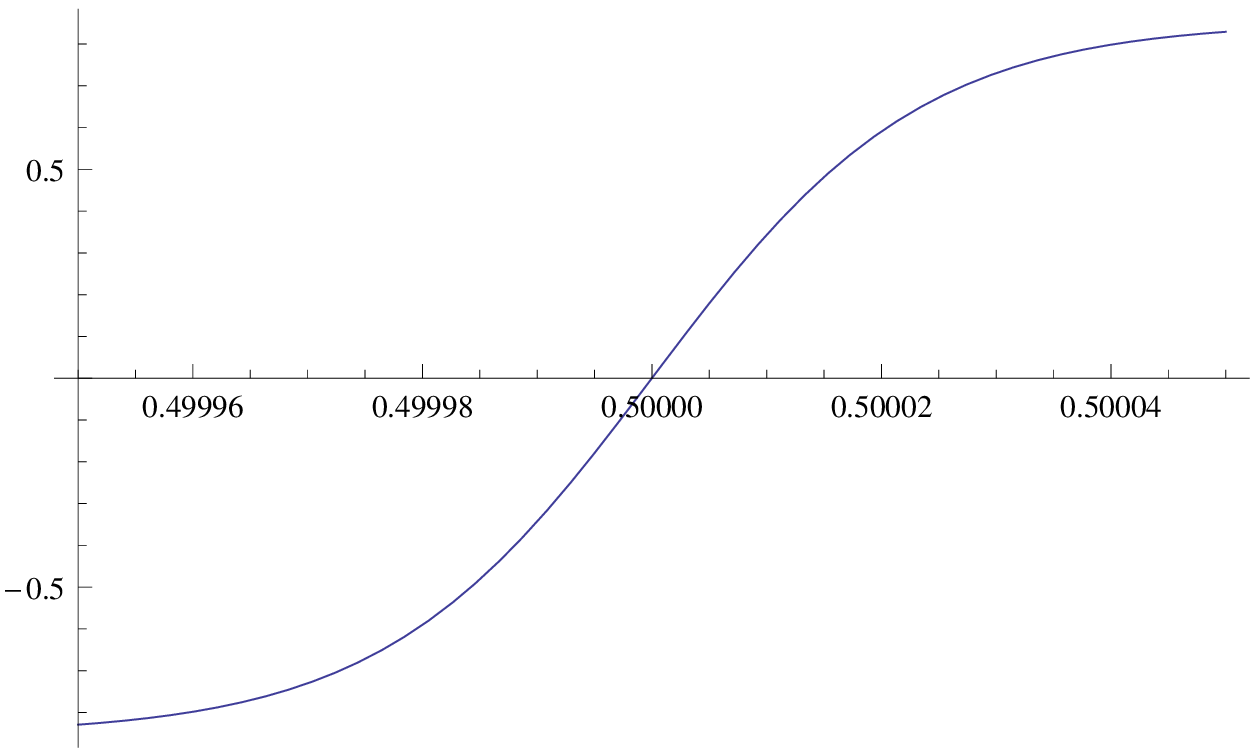}
\par\end{centering}

\caption{\label{fig:moving front, t=00003D10^-5, all }(a) Solution $u(x,t_{j})$,
for $t_{j}=10^{-3}$, $.1$, $.2$, $.3$, and $.5$. (b) $u(x,t_{j})$
in the transition region $\left(1/2-10^{-5},1/2+10^{-5}\right)$,
for $t_{j}=0.4$ (from \cite{HA-BE-MO:2012}). These solutions are
represented with $4$, $11$, $33$, $29$, and $19$ pairs of conjugate-reciprocal
poles. }

\end{figure}

\section{Accuracy verification\label{sub:Experimental-verification}}

We test the accuracy of Algorithm~\ref{Flo:Accurate-con-eigenvalue-decomposition}
on $500$ random Cauchy matrices, $C_{ij}=\left(\alpha_{i}\overline{\alpha_{j}}\right)/\left(1-\gamma_{i}\overline{\gamma_{j}}\right)$,
$i,j=1,\dots,120$. The complex poles $\gamma_{j}=\rho_{j}e^{2\pi i\phi_{j}}$
and residues $\alpha_{j}=\zeta_{j}e^{2\pi i\psi_{j}}$ are generated
by taking $\rho_{j}$, $\phi_{j}$, and $\psi_{j}$ from the uniform
distribution on $\left(0,1\right)$, and taking $\zeta_{j}$ from
the uniform distribution on $\left(0,10\right)$. For each randomly
generated matrix, we first compute, as a gauge, $\overline{C}C=Z\Sigma Z^{-1}$
using the in-built Mathematica$\,^{TM}$ eigenvalue solver with $300$
digits of precision, and compare the result with $\widehat{Z}$ and
$\widehat{\Sigma}$ computed via Algorithm~\ref{Flo:Accurate-con-eigenvalue-decomposition}
using standard double precision. We then evaluate the maximum relative
error in the con-eigenvalues $\lambda_{j}=\Sigma_{jj}$, $\max_{j}\left|\lambda_{j}-\widehat{\lambda_{j}}\right|/\left|\lambda_{j}\right|$,
and the maximum error in the computed con-eigenvectors, $\max_{j}\left\Vert Z\left(:,j\right)-\widehat{Z}\left(:,j\right)\right\Vert _{2}/\left\Vert Z\left(:,j\right)\right\Vert _{2}$.
We first scale $\widehat{Z}\left(:,j\right)$ by the complex-valued
constant $Z\left(i_{0},j\right)/\widehat{Z}\left(i_{0},j\right)$,
$i_{0}=\max_{1\leq i\leq n}\left|Z\left(i,j\right)\right|$, since
$Z\left(:,j\right)$ and $\widehat{Z}\left(:,j\right)$ are defined
only up to an arbitrary complex-valued factor.

Figures~\ref{Flo:poles/con-eig values} and \ref{Flo:error in con-eigs/con-vecs}
summarize the result of a typical run. Figure~\ref{Flo:poles/con-eig values}(a)
shows the distribution of the poles $\gamma_{j}$ inside the unit
disk and Figure~\ref{Flo:poles/con-eig values}(b) displays $\log_{10}\lambda_{j}^{2}$
as a function of the index~$j$. Figures~\ref{Flo:error in con-eigs/con-vecs}(a)~~\ref{Flo:error in con-eigs/con-vecs}(b)
show the relative errors in the con-eigenvalues $\left|\lambda_{j}-\widehat{\lambda_{j}}\right|/\left|\lambda_{j}\right|$
and the normalized con-eigenvectors $\left\Vert z_{j}-\widehat{z_{j}}\right\Vert _{2}/\left\Vert z_{j}\right\Vert _{2}$,
both as functions of the index $j$. 

\begin{figure}[h]
(a)~~ \includegraphics[scale=0.45]{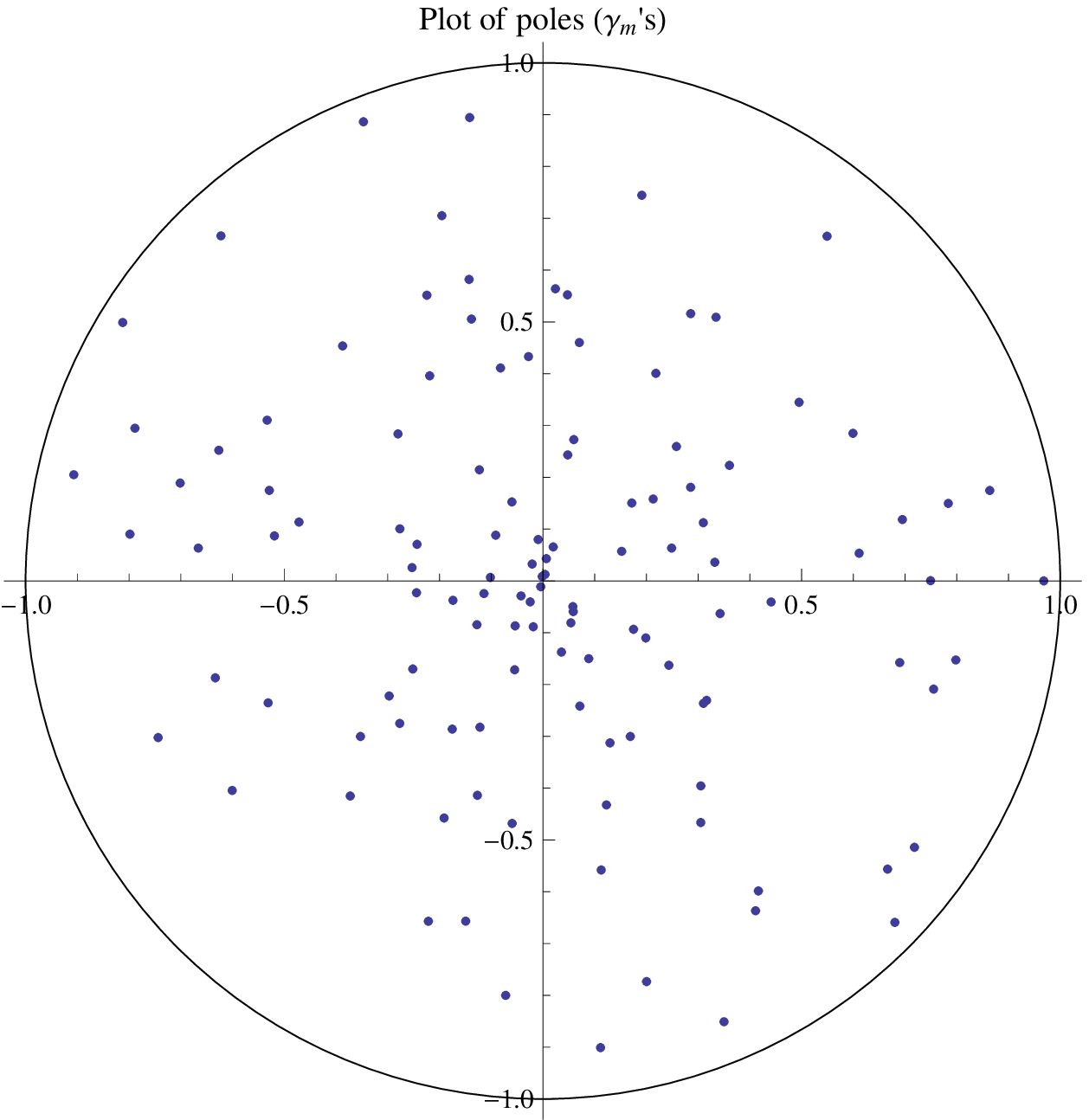} ~~~~~~~~~(b)~~
\includegraphics[scale=0.65]{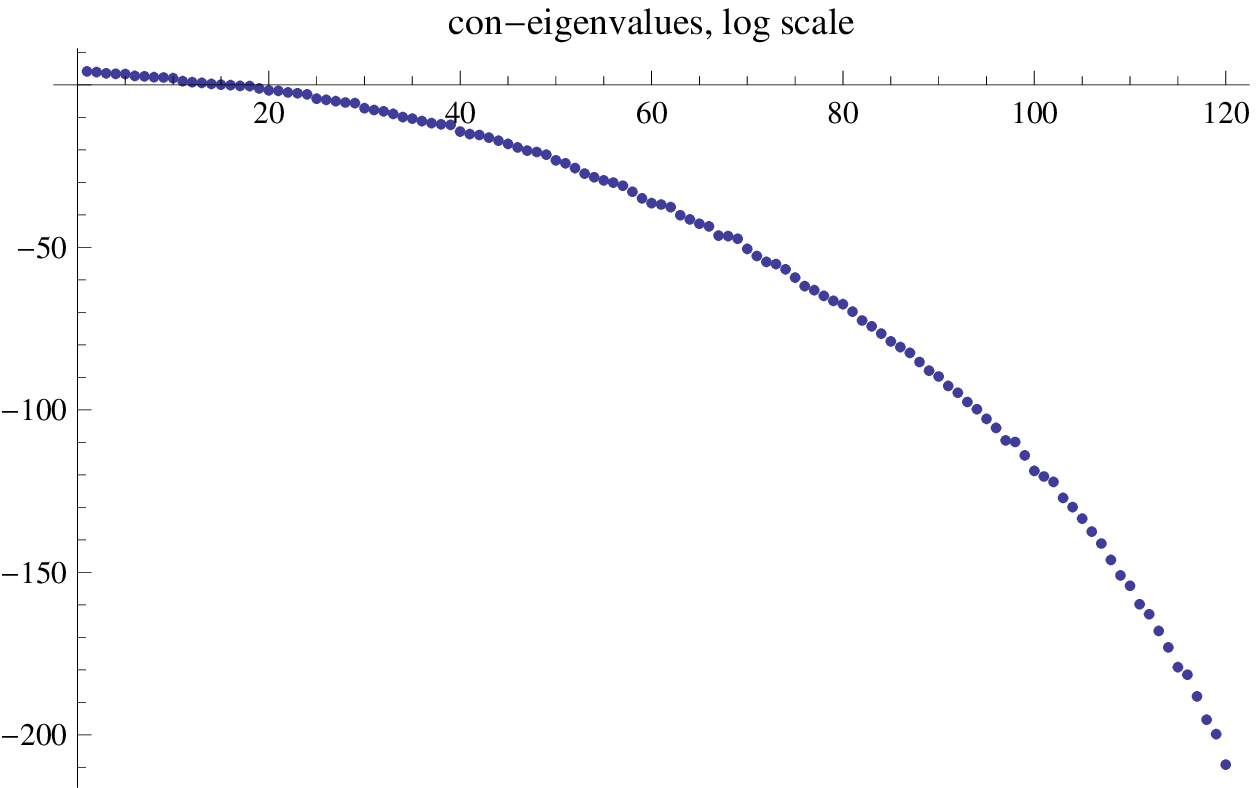}

\caption{\label{Flo:poles/con-eig values}(a) Distribution of poles $\gamma_{j}$
determining Cauchy matrix $C$ in a typical run. (b) Exponential decay
of the eigenvalues $\lambda_{j}^{2}$ of $\overline{C}C$ as a function
of the index $j$ using $\log_{10}$ scale. }

\end{figure}

\begin{figure}[h]
(a)~~ \includegraphics[scale=0.55]{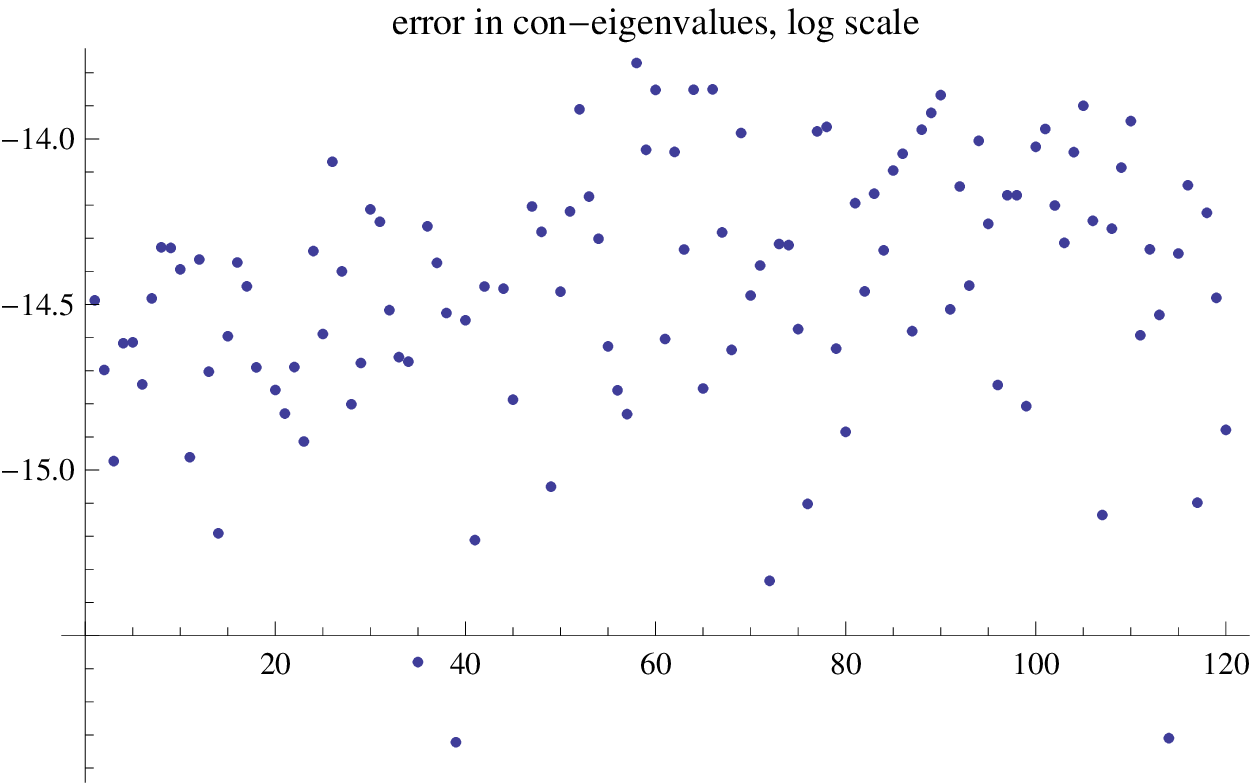} ~~~~~~~(b)\includegraphics[scale=0.55]{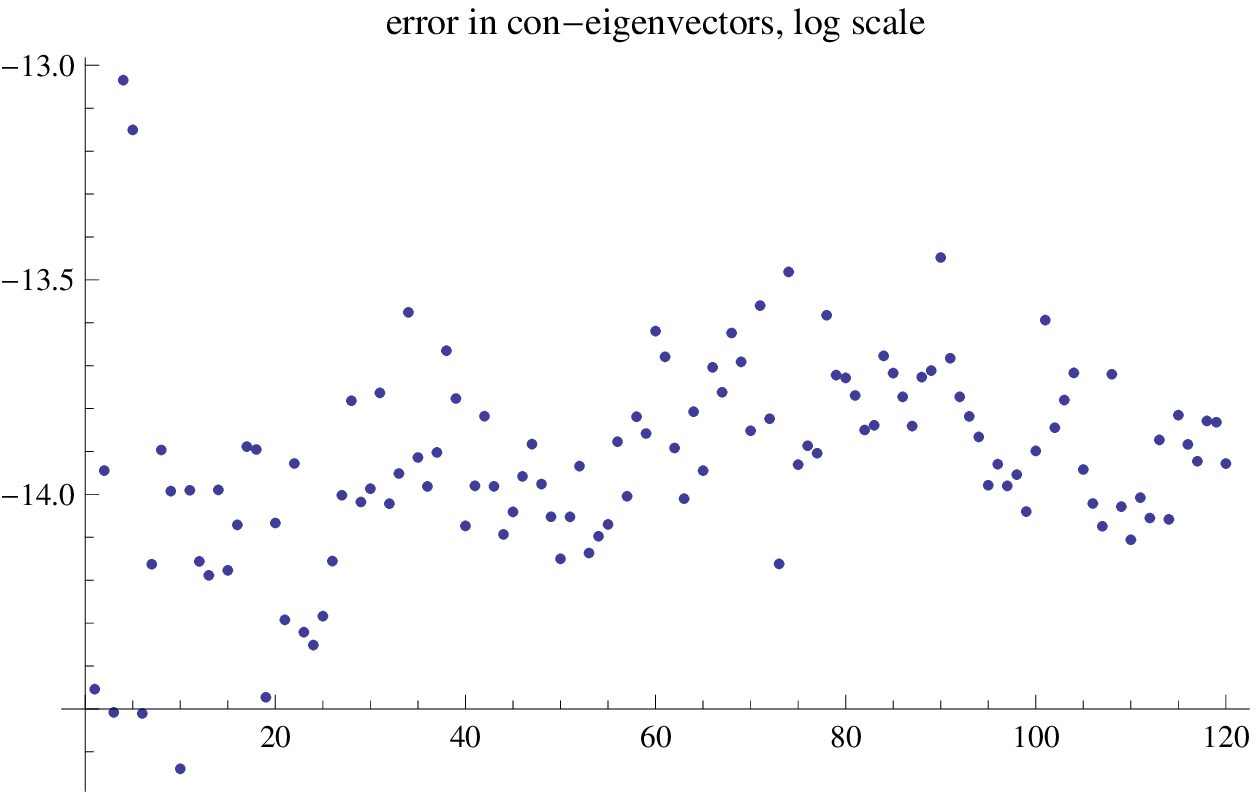}\caption{\label{Flo:error in con-eigs/con-vecs}(a) Relative error in the $j$th
con-eigenvalue, $\left|\lambda_{j}-\widehat{\lambda_{j}}\right|/\left|\lambda_{j}\right|$,
as a function of the index $j$. (b) The error in the $j$th con-eigenvector,
$\left\Vert z_{j}-\widehat{z_{j}}\right\Vert _{2}/\left\Vert z_{j}\right\Vert _{2}$,
$z_{j}=Z\left(:,j\right)$, as a function of the index $j$.}

\end{figure}

In Figures~\ref{Flo:error in con-eig values, various indices} and
\ref{Flo:error in con-eig vectors, various indices} for each of the
$500$ random Cauchy matrices, we plot the error in the computed con-eigenvalues
$\left|\widehat{\lambda_{j}}-\lambda_{j}\right|/\left|\lambda_{j}\right|$
and con-eigenvectors $\left\Vert \widehat{z_{j}}-z_{j}\right\Vert _{2}/\left\Vert z_{j}\right\Vert _{2}$for
$j=1,40,80,120$ (note the exponential decay of $\lambda_{j}$). We
see that the con-eigenvalues and the con-eigenvectors are computed
with nearly full precision for all the Cauchy matrices. In fact, the
largest errors $\left|\widehat{\lambda_{j}}-\lambda_{j}\right|/\left|\lambda_{j}\right|$
and $\left\Vert \widehat{z_{j}}-z_{j}\right\Vert _{2}/\left\Vert z_{j}\right\Vert _{2}$
in the computed con-eigenvalues and con-eigenvectors, for any of the
$500$ Cauchy matrices and any $1\leq j\leq n$, are $5.13\times10^{-12}$
and $5.35\times10^{-12}$.

\begin{figure}[h]
\begin{centering}
(a) ~~\includegraphics[scale=0.55]{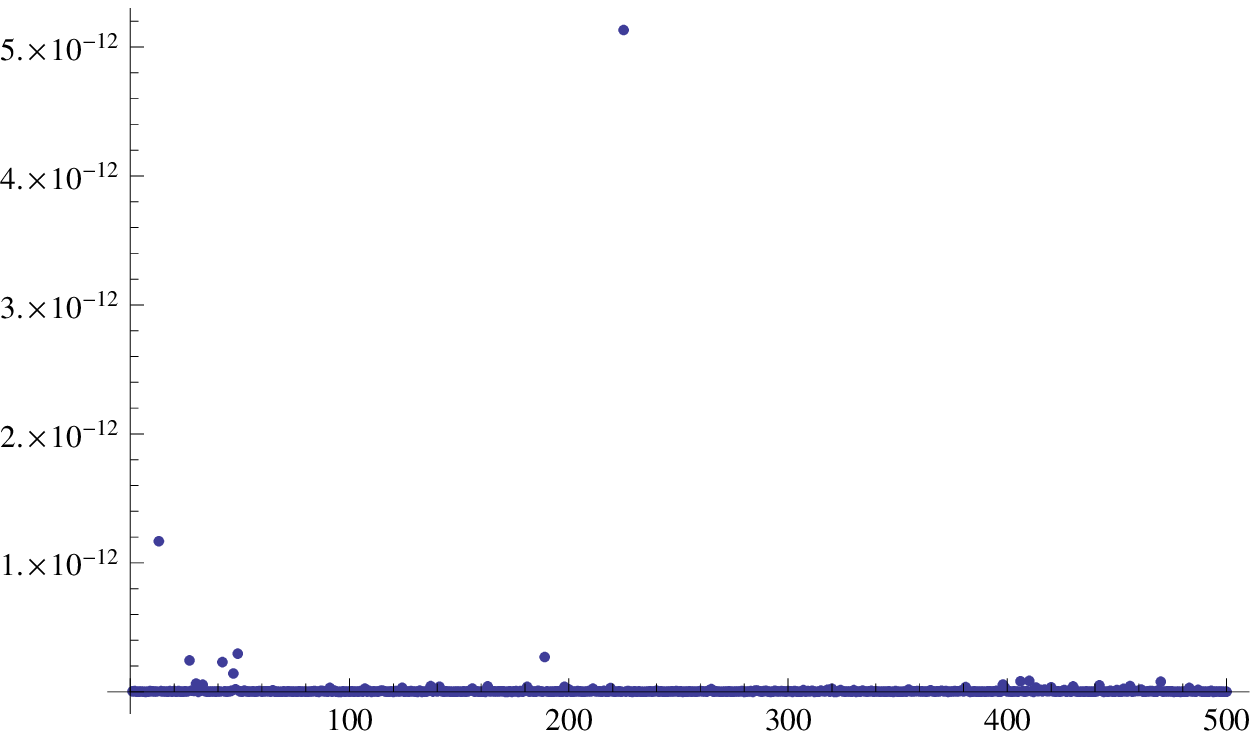} ~~~~~~(b)~~\includegraphics[scale=0.55]{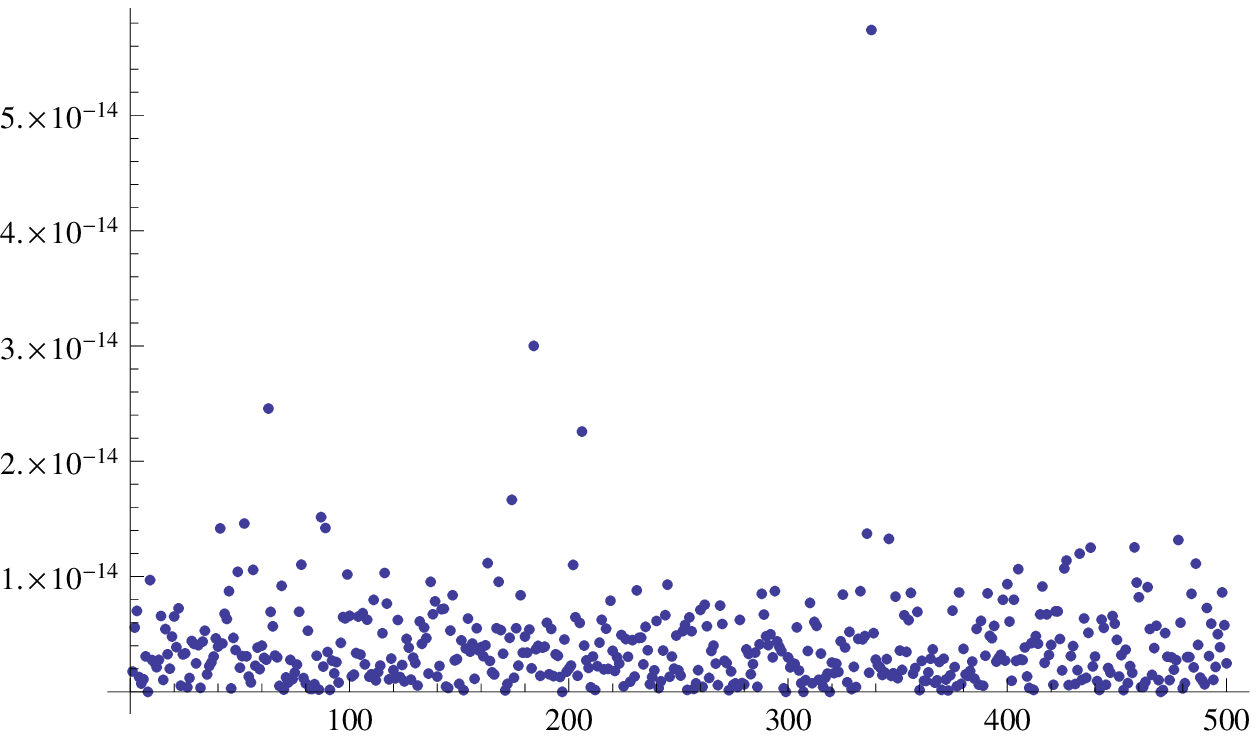}
\par\end{centering}

\begin{centering}
(c)~~\includegraphics[scale=0.55]{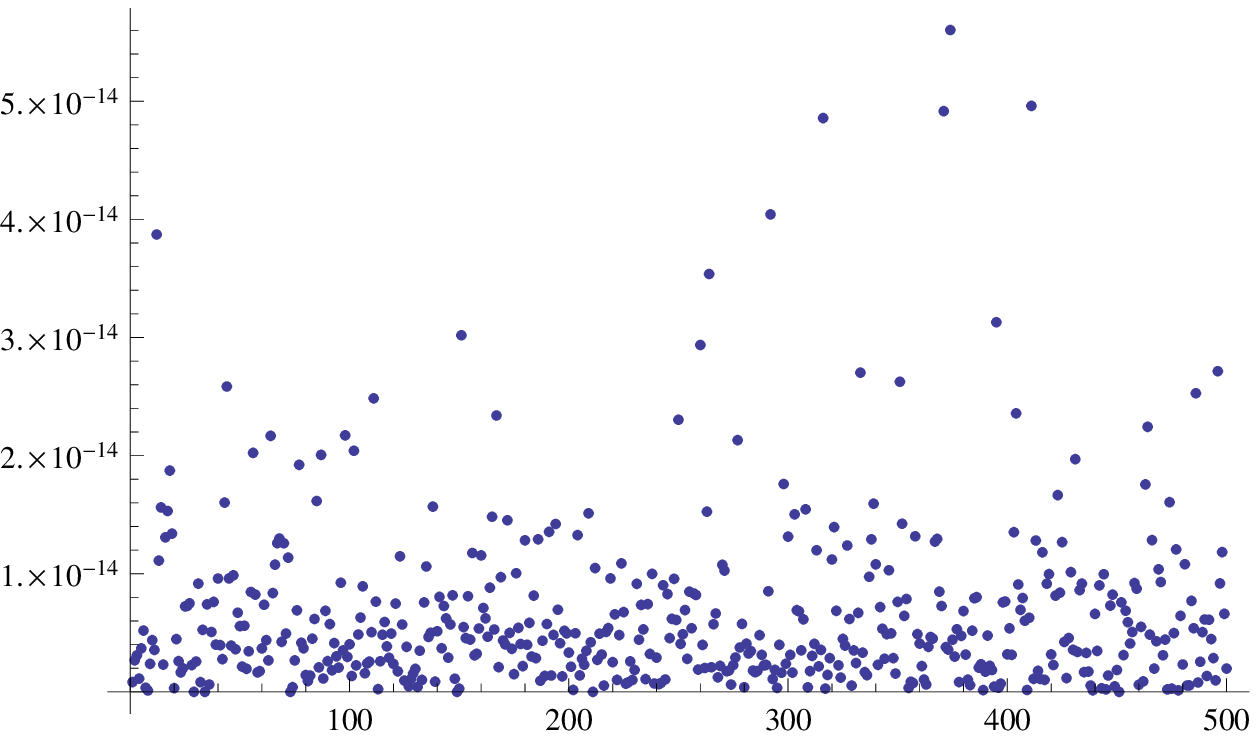}~~~~~~
(d)~~\includegraphics[scale=0.55]{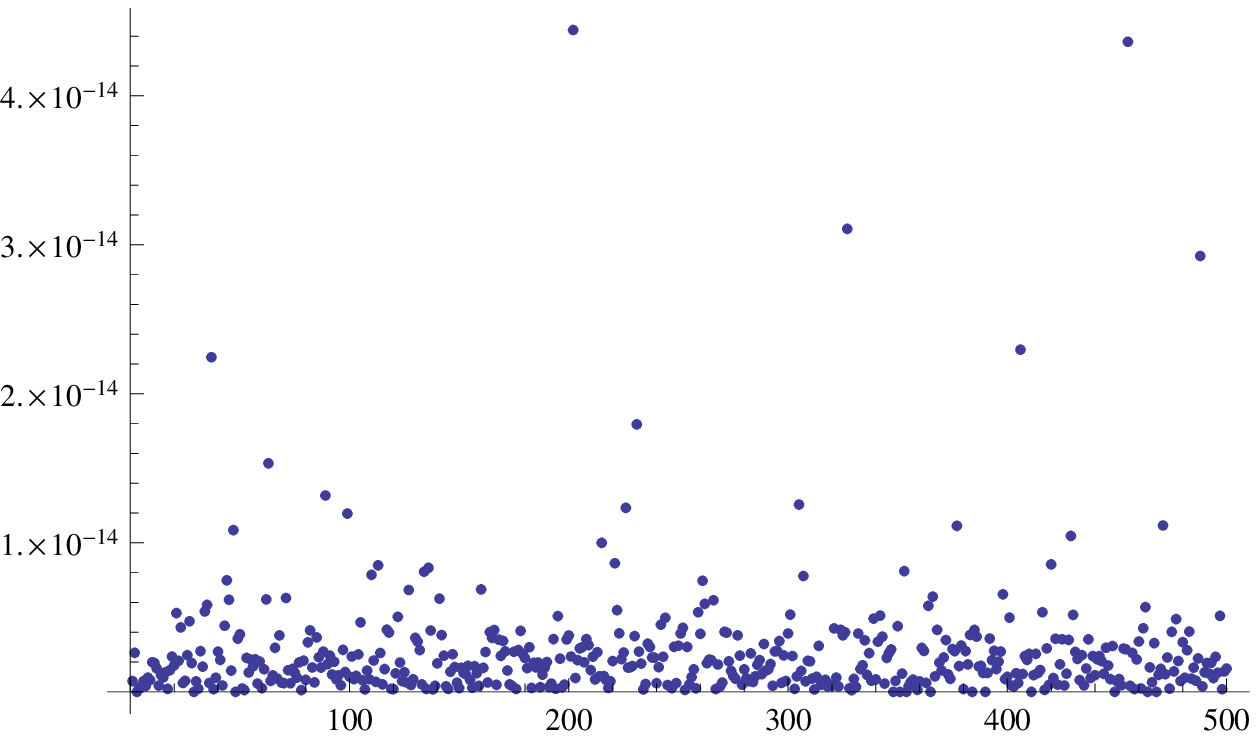}
\par\end{centering}

\caption{\label{Flo:error in con-eig values, various indices}Relative error
in the computed con-eigenvalues, $\left|\widehat{\lambda_{j}}-\lambda_{j}\right|/\left|\lambda_{j}\right|$,
for $j=1,40,80,120$ ((a), (b), (c), and (d), respectively), plotted
for each of the $500$ random Cauchy matrices.}

\end{figure}

\begin{figure}[h]
\begin{centering}
(a)~~\includegraphics[scale=0.55]{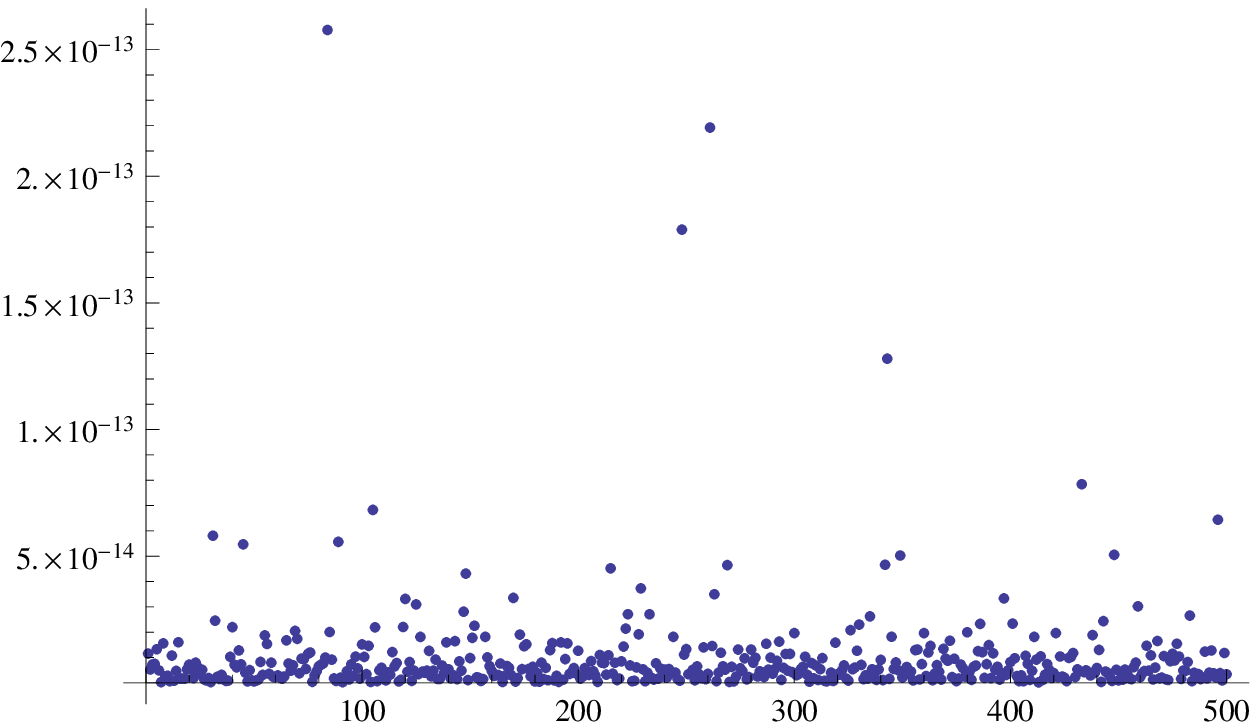}~~~~~~
(b)~~\includegraphics[scale=0.55]{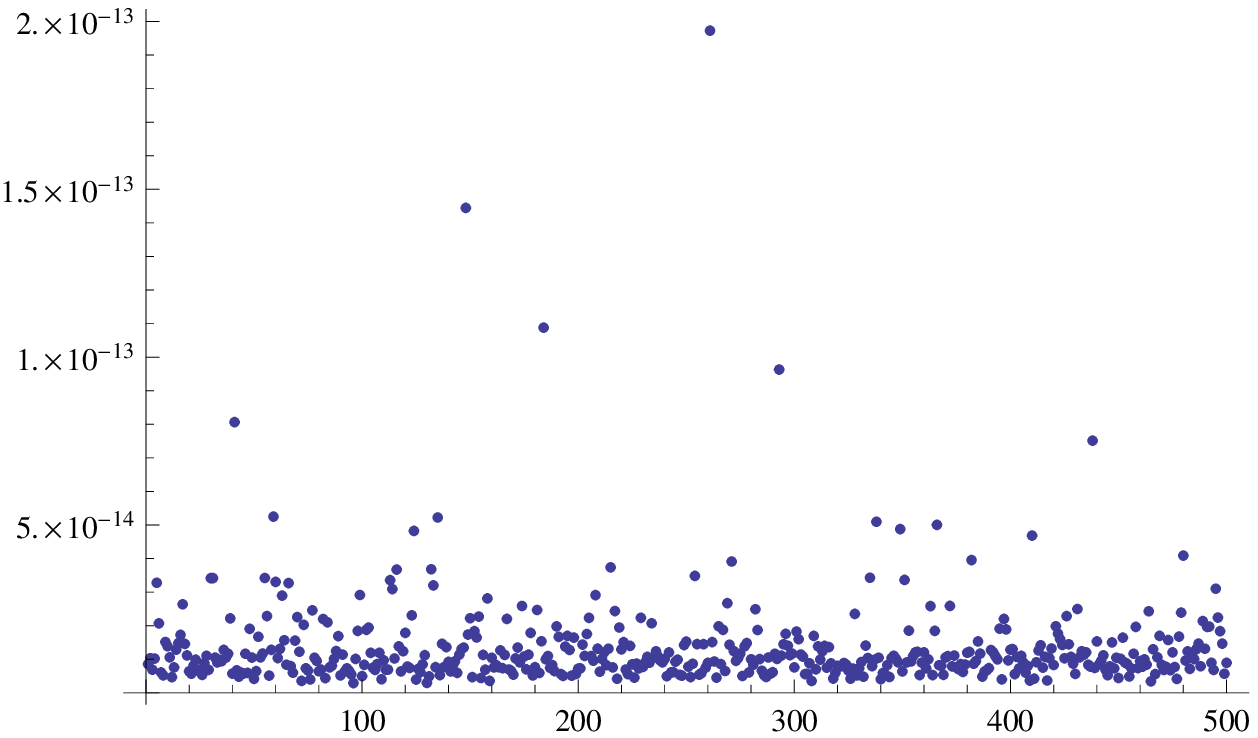}
\par\end{centering}

\begin{centering}
(c)~~ \includegraphics[scale=0.55]{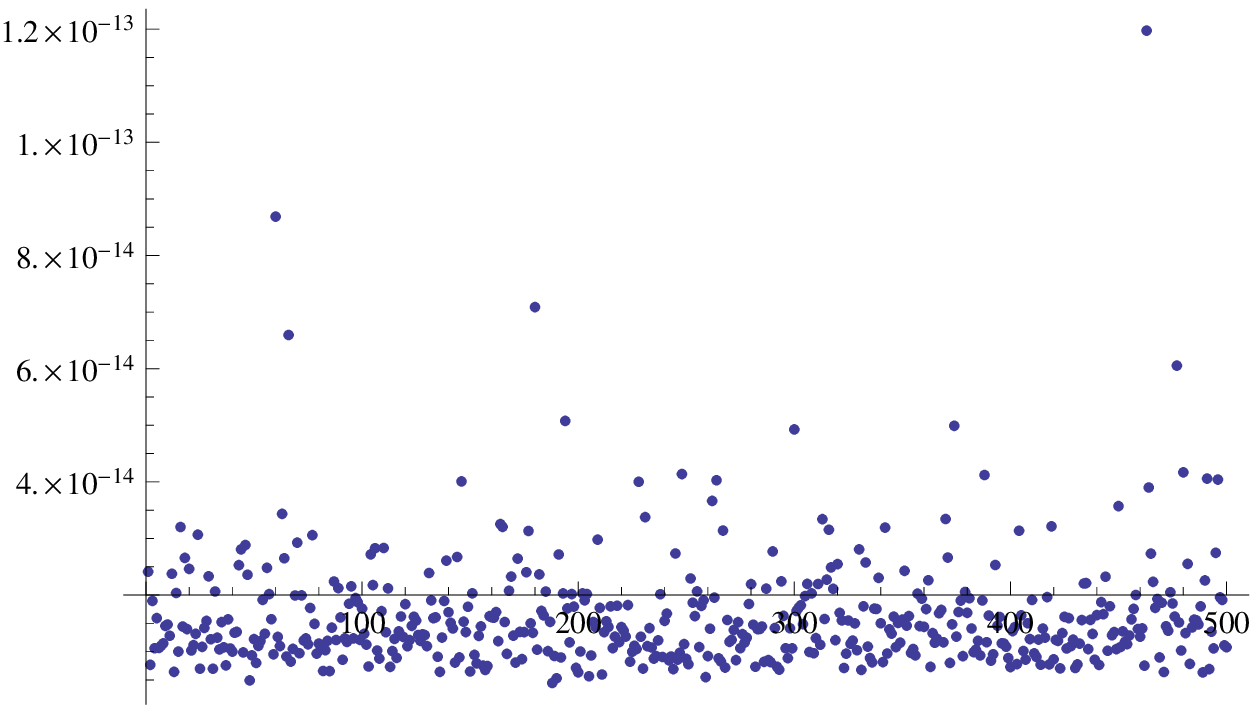}~~~~~~
(d)~~ \includegraphics[scale=0.55]{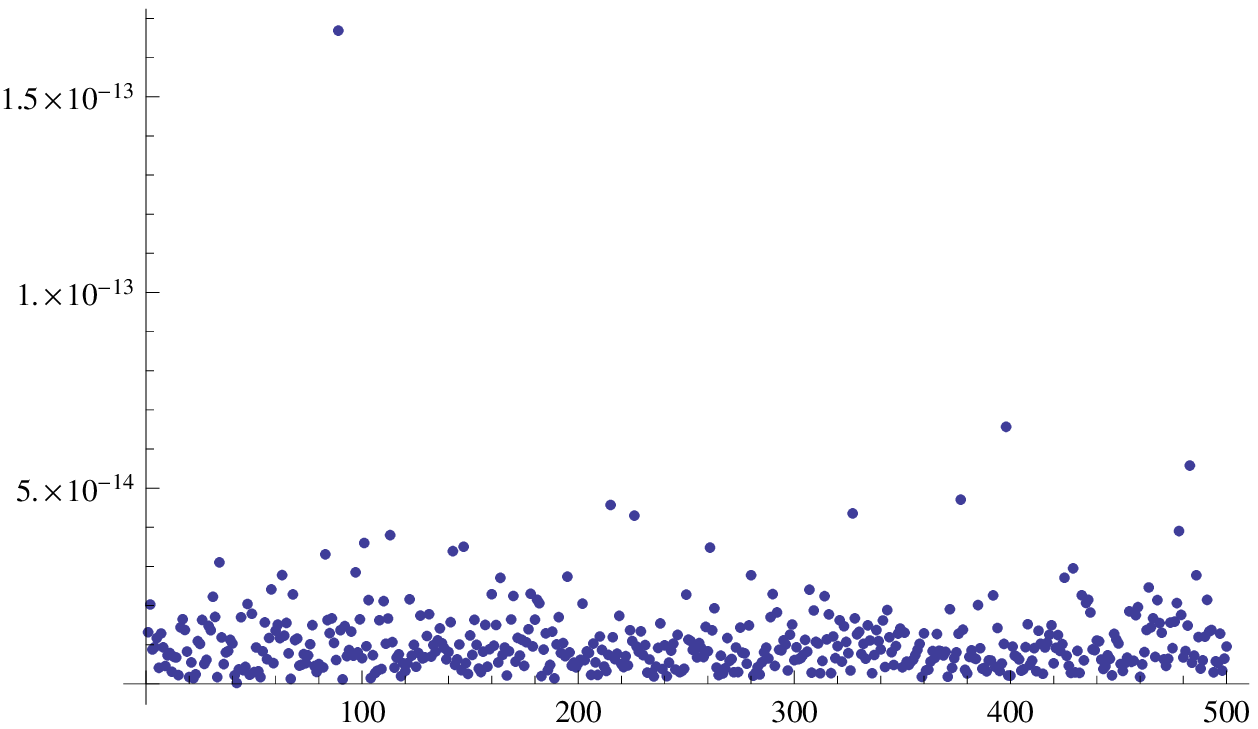}
\par\end{centering}

\caption{\label{Flo:error in con-eig vectors, various indices}Relative error
in the computed con-eigenvectors, $\|\widehat{z_{j}}-z_{j}\|_{2}/\|z_{j}\|_{2}$,
for $j=1,40,80,120$ ((a), (b), (c), and (d), respectively), plotted
for each of the $500$ random Cauchy matrices.}

\end{figure}

\section{Accuracy and perturbation theory\label{sec:Accuracy-and-perturbation}}

We show that Algorithm~\ref{Flo:Accurate-con-eigenvalue-decomposition}
of the previous section achieves high relative accuracy. We also demonstrate
that small perturbations of $a_{i}$, $b_{j}$, $x_{i}$, and $y_{j}$
determining $C$ lead to small relative perturbations of the con-eigenvalues
and small perturbations of the angles between subspaces spanned by
the con-eigenvectors, as long as the parameters $x_{i}$ and $y_{j}$
are not too close in a relative sense. 

For two (complex) floating point numbers $x$ and $y$, we denote
by $\text{fl}\left(x\odot y\right)$ the result of applying the operation
$x\odot y$ in floating point arithmetic, where $\odot$ is one of
the four basic operations, $\odot\in\left\{ +,-,\times,\div\right\} $.
We use that $\text{fl}\left(x\odot y\right)=\left(x\odot y\right)\left(1+\delta\right)$,
where $\left|\delta\right|\leq c\epsilon+\mathcal{O}\left(\epsilon^{2}\right)$,
$\epsilon$ denote the machine round-off, and $c$ is a small constant
(cf. \cite{HIGHAM:2002}). We will also abuse notation by letting
$\text{fl}\left(XY\right)$ denote the result of multiplying matrices
$X$ and $Y$ in floating point arithmetic. Finally, throughout this
paper the notation $\left\Vert \cdot\right\Vert $ always denotes
the Frobenius norm.

In Theorems~\ref{thm:perturbation for con-eig decomp}-\ref{thm:truncation theorem}
below we always assume that the con-eigenvalues are simple, although
this is not a crucial restriction. In the statements and proofs of
these theorems, the implicit constant factor implied by the notation
$\mathcal{O}\left(\eta\right)$ and $\mathcal{O}\left(\epsilon\right)$
(here $\epsilon,\eta\ll1$) depends only on the size $n$ of the matrix
$C$. We also use the notation $\mathcal{O}\left(1\right)$ to denote
a quantity that depends only on the size $n$. We note that all these
implicit constants may be tracked more carefully and are modest-sized
functions of $n$.

The bounds in the theorems below depend on the Cholesky factors in
the decomposition $C=\left(PL\right)D^{2}\left(PL\right)^{*}$. In
particular, let $B=L^{\text{T}}L$, and consider the (symmetric) LU
factorization $B=L_{B}L_{B}^{\text{T}}$, where $L_{B}$ is unit lower
triangular. Then the estimates in Theorems~\ref{thm:perturbation for con-eig decomp}~-~\ref{thm:truncation theorem}
depend on the quantities \begin{eqnarray}
\mu_{0}\left(L_{B}\right) & = & \left\Vert L_{B}^{-1}\right\Vert ^{2}\kappa\left(L_{B}\right),\label{eq:constants for bounds in proof}\\
\mu_{1}\left(L_{B}\right) & = & \max\left\{ \left\Vert L_{B}^{-1}\right\Vert ^{2},\left\Vert L_{B}\right\Vert ^{2}\right\} \kappa\left(L_{B}\right),\nonumber \\
\mu_{2}\left(L_{B}\right) & = & \left\Vert L_{B}^{-1}\right\Vert ^{2}\mu_{1}\left(L_{B}\right)\kappa^{3}\left(L_{B}\right),\nonumber \end{eqnarray}
where the condition number $\kappa\left(L_{B}\right)=\left\Vert L_{B}\right\Vert \left\Vert L_{B}^{-1}\right\Vert $
is not too large if $B$ and its $n$ leading principal submatrices
are well-conditioned. The estimates in Theorems~\ref{thm:high accuracy of algorithm-1}-\ref{thm:truncation theorem}
also depend on \begin{equation}
\mu_{3}\left(L_{B}\right)=\left\Vert L^{-1}\right\Vert \left(\rho\mu\psi\mu_{2}\left(L_{B}\right)+\nu\kappa^{3}\left(L_{B}\right)\right),\label{eq:constants for bounds, 2}\end{equation}
where $\rho$, $\mu$, and $\psi$ are {}``pivot growth'' factors
associated with the QR factorization (see Section~\ref{sub:QR decomposition}),
and the factor $\nu$ is associated with the one-sided Jacobi algorithm
(see (\ref{eq:def of nu})).
\begin{rem*}
There are simple formulas for $L_{ij}$ and $\left(L^{-1}\right)_{ij}$
(\cite{CHO:1968}) in terms of the parameters $a_{i}$, $b_{j}$,
$x_{i}$ and $y_{j}$ defining the Cauchy matrix $C$, and it is possible
that the bounds below may be improved by using this additional structure. \end{rem*}
\begin{thm}
\label{thm:perturbation for con-eig decomp}Suppose that the parameters
defining the positive-definite Cauchy matrix $C=C(a,b,x,y)$ are perturbed
to $\tilde{a}=a+\delta a$, $\tilde{b}=b+\delta b$, $x=x+\delta x$,
and $y=y+\delta y$. Let us define\[
\eta=\left(1/\eta_{1}+1/\eta_{2}+1/\eta_{3}\right)\max\left\{ \left\Vert \delta a\right\Vert _{\infty},\left\Vert \delta b\right\Vert _{\infty},\left\Vert \delta x\right\Vert _{\infty},\left\Vert \delta y\right\Vert _{\infty}\right\} ,\]
where \[
\eta_{1}=\min_{i\neq j}\frac{\left|x_{i}-x_{j}\right|}{\left|x_{j}\right|+\left|x_{i}\right|},\,\,\,\,\eta_{2}=\min_{i\neq j}\frac{\left|y_{i}-y_{j}\right|}{\left|y_{j}\right|+\left|y_{i}\right|},\,\,\,\,\eta_{3}=\min_{i\neq j}\frac{\left|x_{i}+y_{j}\right|}{\left|x_{i}\right|+\left|y_{j}\right|}.\]
Let $C=LDL^{*}$ denote the Cholesky factorization of $C$, and let
$\widetilde{C}=C(\widetilde{a},\widetilde{b},\widetilde{x},\widetilde{y})$
denote the Cauchy matrix corresponding to the perturbed parameters.
Finally, let $z_{i}$, $\widetilde{z}_{i}$ denote the con-eigenvectors
of $C$ and $\widetilde{C}$, corresponding to con-eigenvalues $\lambda_{i}$
and $\widetilde{\lambda_{i}}$ .

Then the relative difference in the con-eigenvalues $\lambda_{i}$
and $\widetilde{\lambda_{i}}$ is bounded as\[
\left|\frac{\lambda_{i}-\widetilde{\lambda_{i}}}{\lambda_{i}}\right|\leq\mu_{0}\left(L_{B}\right)\mathcal{O}\left(\eta\right),\]
and the acute angle between the con-eigenvectors $z_{i}$ and $\widetilde{z}_{i}$
is bounded by \[
\sin\left(\angle z_{i},\widetilde{z_{i}}\right)\leq\kappa\left(L\right)\left(\frac{\mu_{2}\left(L_{B}\right)}{\text{relgap}{}_{i}}+\mu_{0}\left(L_{B}\right)\mu_{1}\left(L_{B}\right)\right)\mathcal{O}\left(\eta\right).\]
Here $\mu_{0}\left(L_{B}\right)$, $\mu_{1}\left(L_{B}\right)$ and
$\mu_{2}\left(L_{B}\right)$ are defined in (\ref{eq:constants for bounds in proof}),
and \[
\mbox{relgap}_{i}=\min_{j\neq i}\frac{\left|\lambda_{i}-\lambda_{j}\right|}{\left|\lambda_{i}\right|+\left|\lambda_{j}\right|}.\]

\end{thm}
Next we state
\begin{thm}
\label{thm:high accuracy of algorithm-1}Suppose that Algorithm~\ref{Flo:Accurate-con-eigenvalue-decomposition}
is used to compute the full con-eigenvalue decomposition of a positive-definite
Cauchy matrix $C$. Suppose also that $C$ has the Cholesky factorization
$C=\left(PL\right)D^{2}\left(PL\right)^{*}$, where $P$ is the permutation
matrix that encodes complete pivoting. 

Then the relative error between the computed con-eigenvalue $\widehat{\lambda_{i}}$
and the exact $\lambda_{i}$ is bounded as \[
\frac{\left|\widehat{\lambda_{i}}-\lambda_{i}\right|}{\left|\lambda_{i}\right|}\leq\left(\rho\mu\psi\mu_{0}\left(L_{B}\right)+\nu\right)\mathcal{O}\left(\epsilon\right),\]
where $\rho$, $\mu$, and $\psi$ are {}``pivot growth'' factors
associated with the QR factorization (see Section~\ref{sub:QR decomposition}),
and the factor $\nu$ is associated with the one-sided Jacobi algorithm
(see (\ref{eq:def of nu})).

Letting $z_{i}$, $\widehat{z_{i}}$ denote exact and computed con-eigenvectors
of $C$, the acute angle between $z_{i}$ and $\widehat{z_{i}}$ then
satisfies \[
\sin\left(\angle\widehat{z_{i}},z_{i}\right)\leq\kappa\left(L\right)\left(\frac{\mu_{3}\left(L_{B}\right)}{\text{relgap}{}_{i}}+\left\Vert L^{-1}\right\Vert ^{2}\kappa^{3}\left(L_{B}\right)\right)\mathcal{O}\left(\epsilon\right),\]
where $\mbox{relgap}_{i}$ is defined as in Theorem~\ref{thm:perturbation for con-eig decomp}
and $\mu_{3}\left(L_{B}\right)$ is defined in \ref{eq:constants for bounds, 2}.
\end{thm}
~~~
\begin{thm}
\label{thm:truncation theorem}Suppose Algorithm~\ref{Flo:Accurate-con-eigenvalue-decomposition}
is used to compute $m$ approximate con-eigenvalues and con-eigenvectors
of a positive-definite Cauchy matrix $C$. Suppose also that $C$
has the Cholesky factorization $C=\left(PL\right)D^{2}\left(PL\right)^{*}$,
where $P$ is the permutation matrix that encodes complete pivoting.
Assuming that $D_{mm}^{2}\leq\lambda_{i}\epsilon$ for some $1\leq i\leq m$,
the following error bound holds for the computed con-eigenvalue $\widehat{\lambda_{i}}$,
\[
\frac{\left|\widehat{\lambda_{i}}-\lambda_{i}\right|}{\left|\lambda_{i}\right|}\leq\left(\rho\mu\psi\mu_{0}\left(L_{B}\right)+\nu+\left\Vert C\right\Vert \mu_{1}^{2}\left(L_{B}\right)\right)\mathcal{O}\left(\epsilon\right),\]
 and the acute angle between $z_{i}$ and $\widehat{z_{i}}$ is bounded
by \[
\sin\left(\angle\widehat{z_{i}},z_{i}\right)\leq\kappa\left(L\right)\left(\frac{\mu_{3}\left(L_{B}\right)+\left\Vert C\right\Vert \mu_{1}^{2}\left(L_{B}\right)}{\text{relgap}{}_{i}}+\left\Vert L^{-1}\right\Vert ^{2}\kappa^{3}\left(L_{B}\right)\right)\mathcal{O}\left(\epsilon\right).\]
In the above estimates, $\rho$, $\mu$, and $\psi$ are {}``pivot
growth'' factors associated with the QR factorization (see Section~\ref{sub:QR decomposition}),
and the factor $\nu$ is associated with the one-sided Jacobi algorithm
(see (\ref{eq:def of nu})).\end{thm}
\begin{rem}
We note that the constants in the theorems above are significantly
more pessimistic than we actually observe in numerical experiments.
Indeed, while the bounds on the con-eigenvectors depend only on the
well-conditioned matrices $\kappa\left(L_{B}\right)$ and $k\left(L\right)$
(and, in particular, are independent of the exponentially decaying
diagonal matrix $D$), they still scale like $\kappa^{9}$, where
$\kappa=\max\left\{ \kappa\left(L_{B}\right),\kappa\left(L\right)\right\} $;
the bounds on the con-eigenvalues are better---they scale like $\kappa^{3}$.
However, in practice Algorithm~\ref{Flo:Accurate-con-eigenvalue-decomposition}
achieves nearly full precision for all the con-eigenvalues and con-eigenvectors.
While it is likely that better estimates can be obtained, those presented
here elucidate the basic mechanism behind the high accuracy that we
observe in our experiments.
\end{rem}

\subsection{Perturbation theorem\label{sub:Perturbation-theorem}}

In this section, we prove Theorem~\ref{thm:perturbation for con-eig decomp}.
We start by formulating several preliminary results. Lemma~\ref{lem:perturbation, factorization of Cauchy}
describes how perturbations of the vectors $a$, $b$, $x$, and $y$
defining the Cauchy matrix $C=C(a,b,x,y)$ change the factors $L$
and $D$ in the Cholesky decomposition $C=LDL^{*}$ (see \cite{DEMMEL:1999}
for a proof). %
{}
\begin{lem}
\label{lem:perturbation, factorization of Cauchy}Suppose the data
defining the Cauchy matrix $C=C(a,b,x,y)$ is perturbed to $\tilde{a}=a+\delta a$,
$\tilde{b}=b+\delta b$, $x=x+\delta x$, and $y=y+\delta y$. Let
us define\[
\eta=\left(1/\eta_{1}+1/\eta_{2}+1/\eta_{3}\right)\max\left\{ \left\Vert \delta a\right\Vert _{\infty},\left\Vert \delta b\right\Vert _{\infty},\left\Vert \delta x\right\Vert _{\infty},\left\Vert \delta y\right\Vert _{\infty}\right\} ,\]
where \[
\eta_{1}=\min_{i\neq j}\frac{\left|x_{i}-x_{j}\right|}{\left|x_{j}\right|+\left|x_{i}\right|},\,\,\,\,\eta_{2}=\min_{i\neq j}\frac{\left|y_{i}-y_{j}\right|}{\left|y_{j}\right|+\left|y_{i}\right|},\,\,\,\,\eta_{3}=\min_{i\neq j}\frac{\left|x_{i}+y_{j}\right|}{\left|x_{i}\right|+\left|y_{j}\right|}.\]
Then $C=C(a,b,x,y)$ and $\widetilde{C}=\widetilde{C}(\widetilde{a},\widetilde{b},\widetilde{x},\widetilde{y})$
have Cholesky factorizations $C=LDL^{*}$ and $\widetilde{C}=\widetilde{L}\widetilde{D}\widetilde{L}^{*}$,
where $L$, $\widetilde{L}$ are unit lower triangular matrices, $D$,
$\widetilde{D}$ are diagonal matrices with positive entries, and
\[
\left|L_{ij}-\widehat{L}_{ij}\right|=\left|L_{ij}\right|\mathcal{O}\left(\eta\right),\,\,\,\,\left|D_{ii}-\widehat{D_{ii}}\right|=\left|D_{ii}\right|\mathcal{O}\left(\eta\right).\]

\end{lem}
We now state the main result needed to prove Theorem~\ref{thm:perturbation for con-eig decomp}.
Proposition~\ref{pro:perturb for sing vecs of DBD} considers how
a perturbation $B+\delta B$ in $DBD$ affects the singular vectors.
It turns out that, if all the principal minors of $B$ are well-conditioned,
then the errors in the perturbed singular vectors are graded, \[
\left|v_{i}(j)-\widetilde{v}_{i}(j)\right|\sim\min\left\{ \frac{D_{jj}}{\sqrt{\Sigma_{ii}}},\frac{\sqrt{\Sigma_{ii}}}{D_{jj}}\right\} \mathcal{O}\left(\left\Vert \delta B\right\Vert \right).\]
The proof uses techniques developed in \cite{BAR-DEM:1988}, \cite{DEM-VES:1992},
and \cite{MATHIA:1997}.
\begin{prop}
\label{pro:perturb for sing vecs of DBD}Suppose that $G=DBD$ and
$\widetilde{G}=D\left(B+\delta B\right)D$, where $B$ is complex
symmetric ($B=B^{\text{T}}$) and non-singular, and $D$ is a diagonal
matrix with positive, decreasing diagonal elements. Assume also that
$B$ has the LU decomposition $B=L_{B}L_{B}^{\text{T}}$.

Then the $i$th singular values $\Sigma_{ii}$ and $\widetilde{\Sigma_{ii}}$
of $G$ and $\widetilde{G}$ satisfy \[
\left|\frac{\Sigma_{ii}-\widetilde{\Sigma_{ii}}}{\Sigma_{ii}}\right|\leq\mu_{0}\left(L_{B}\right)\mathcal{O}\left(\left\Vert \delta B\right\Vert \right),\]
where $\mu_{0}\left(L_{B}\right)$ is defined in \ref{eq:constants for bounds in proof}.
The $i$th (left or right) singular vectors $u_{i}$ and $\widetilde{u}_{i}$
of $G$ and $\widetilde{G}$, corresponding to (simple) singular values
$\Sigma_{ii}$ and $\widetilde{\Sigma_{ii}}$, may be chosen so that\[
\left|u_{i}(j)\right|\le\mu_{1}\left(L_{B}\right)\min\left\{ \frac{D_{jj}}{\Sigma_{ii}^{1/2}},\frac{\Sigma_{ii}^{1/2}}{D_{jj}}\right\} ,\]
and\[
\left|u_{i}(j)-\widetilde{u}_{i}(j)\right|\le\mu_{2}\left(L_{B}\right)\min\left\{ \frac{D_{jj}}{\Sigma_{ii}^{1/2}},\frac{\Sigma_{ii}^{1/2}}{D_{jj}}\right\} \mathcal{O}\left(\left\Vert \delta B\right\Vert \right),\]
where $\mu_{1}\left(L_{B}\right)$ and $\mu_{2}\left(L_{B}\right)$
are defined in \ref{eq:constants for bounds in proof}. Finally, we
have the following norm-wise bounds,\begin{equation}
\left\Vert B\right\Vert ^{1/2}\le\left\Vert \frac{Du_{i}}{\Sigma_{ii}^{1/2}}\right\Vert ,\,\,\,\left\Vert B^{-1}\right\Vert ^{-1/2}\le\left\Vert D^{-1}u_{i}\Sigma_{ii}^{1/2}\right\Vert ,\label{eq:bound for sing val ratio}\end{equation}
where $\mu_{0}\left(L_{B}\right)$ is defined in \ref{eq:constants for bounds in proof}.\end{prop}
\begin{proof}
See the Appendix for proofs of the component-wise bounds on the left
singular vectors. The bound on the relative difference $\left|\Sigma_{ii}-\widetilde{\Sigma_{ii}}\right|/\Sigma_{ii}$
is proven in \cite{D-G-E-S-V-D:1997}.

We now prove the norm-wise bound on $Du_{i}/\Sigma_{ii}^{1/2}$. To
do so, note that $DBD$ has the SVD $DBD=U\Sigma U^{\text{T}}$, since
$B$ is complex symmetric (i.e., $B$ has a Takagi factorization).
Now suppose that $DBD\overline{u_{i}}=\Sigma_{ii}u_{i}$. Then the
bound for $Du_{i}/\Sigma_{ii}^{1/2}$ follows from \[
\Sigma_{ii}=\left|\left(DBD\overline{u_{i}},u_{i}\right)\right|=\left|\left(BD\overline{u_{i}},Du_{i}\right)\right|\leq\left\Vert B\right\Vert \left\Vert Du_{i}\right\Vert ^{2}.\]
Similarly, note that $D^{-1}B^{-1}D^{-1}u_{i}=\Sigma_{ii}^{-1}\overline{u_{i}}$.
The bound for $D^{-1}u_{i}\Sigma_{ii}^{1/2}$ then follows from\[
\Sigma_{ii}^{-1}=\left|\left(D^{-1}B^{-1}D^{-1}u_{i},\overline{u_{i}}\right)\right|=\left|\left(B^{-1}D^{-1}u_{i},D^{-1}\overline{u_{i}}\right)\right|\leq\left\Vert B^{-1}\right\Vert \left\Vert D^{-1}u_{i}\right\Vert ^{2}.\]

\end{proof}
We now prove Theorem~\ref{thm:perturbation for con-eig decomp}.
\begin{proof}
Recall that the matrix $Z$ of con-eigenvectors satisfies $Z=P\overline{L}\left(D\overline{V}\Sigma^{-1/2}\right)$,
where $C=\left(PL\right)D^{2}\left(PL\right)^{*}$ is the Cholesky
factorization of $C$ (with complete pivoting) and $V$ is the matrix
of right singular vectors of $G=D\left(L^{\text{T}}L\right)D$. Let
$\widetilde{C}=\left(P\widetilde{L}\right)\widetilde{D}^{2}\left(P\widetilde{L}\right)^{*}$
denote the Cholesky factorization of $\widetilde{C}$. From Lemma~\ref{lem:perturbation, factorization of Cauchy}
(see also \cite{DEMMEL:1999}), \begin{equation}
\left\Vert L-\widetilde{L}\right\Vert =\left\Vert L\right\Vert \mathcal{O}\left(\eta\right),\,\,\,\,\left|\frac{D_{ii}-\widetilde{D_{ii}}}{D_{ii}}\right|=\mathcal{O}\left(\eta\right).\label{eq:bound on SVD perturbed, Cauchy}\end{equation}
Defining $\widetilde{G}=\widetilde{D}\left(\widetilde{L}^{\text{T}}\widetilde{L}\right)\widetilde{D}$,
the above bounds yield $\widetilde{G}=D\left(L^{\text{T}}L+E\right)D$,
where $\left\Vert E\right\Vert =\mathcal{O}\left(\eta\right)$ (since
complete pivoting is used in the Cholesky factorization of $C$, $\left\Vert L\right\Vert =\mathcal{O}\left(1\right)$).
Since the $i$th con-eigenvalue $\lambda_{i}$ of $C$ is given by
$\lambda_{i}=\Sigma_{ii}$, the estimate for the relative difference
$\left|\left(\lambda_{i}-\widetilde{\lambda_{i}}\right)/\lambda_{i}\right|$
now follows from Proposition~\ref{pro:perturb for sing vecs of DBD}.

To derive bounds for the singular vectors, we apply Proposition~\ref{pro:perturb for sing vecs of DBD}
to $\widetilde{G}=D\left(L^{\text{T}}L+E\right)D$. In particular,
there exist right unit singular vectors $\widetilde{v}_{i}$ of $\widetilde{G}$
and $v_{i}$ of $G$ such that \begin{equation}
\left|v_{i}(j)\right|\leq\mu_{1}\left(L_{B}\right)d_{ij}\left(D,\Sigma\right),\label{eq:bound on sing vectors, perturbed proof}\end{equation}
and\begin{equation}
\left|v_{i}(j)-\widetilde{v}_{i}(j)\right|\le d_{ij}\left(D,\Sigma\right)\left(\frac{\mu_{2}\left(L_{B}\right)}{\mbox{relgap}{}_{i}}\right)O\left(\left\Vert E\right\Vert \right)\le d_{ij}\left(D,\Sigma\right)\frac{\mu_{2}\left(L_{B}\right)}{\mbox{relgap}{}_{i}}\mathcal{O}\left(\eta\right),\label{eq:bound on sing vector diff, perturbed proof}\end{equation}
where \[
d_{ij}\left(D,\Sigma\right)=\min\left\{ \frac{D_{jj}}{\Sigma_{ii}^{1/2}},\frac{\Sigma_{ii}^{1/2}}{D_{jj}^{2}}\right\} .\]
Here $\mu_{1}\left(L_{B}\right)$ and $\mu_{2}\left(L_{B}\right)$
are defined in (\ref{eq:constants for bounds in proof}). Proposition~\ref{pro:perturb for sing vecs of DBD}
also shows that \begin{equation}
1-\mu_{0}\left(L_{B}\right)\mathcal{O}\left(\eta\right)\leq\frac{\Sigma_{ii}}{\widetilde{\Sigma_{ii}}}\leq1+\mu_{0}\left(L_{B}\right)\mathcal{O}\left(\eta\right).\label{eq:ratio of sing vals, perturb proof}\end{equation}
Finally, defining $w_{i}=Dv_{i}/\Sigma_{ii}^{1/2}$ and $\widetilde{w}_{i}=\widetilde{D}\widetilde{v}_{i}/\widetilde{\Sigma_{ii}}^{1/2}$,
we have \begin{eqnarray*}
\left|w_{i}(j)-\widetilde{w}_{i}(j)\right| & = & \frac{D_{jj}}{\Sigma_{ii}^{1/2}}\left|v_{i}(j)-\frac{\widetilde{D_{jj}}}{D_{jj}}\left(\frac{\Sigma_{ii}}{\widetilde{\Sigma_{ii}}}\right)^{1/2}\widetilde{v}_{i}(j)\right|\\
 & \leq & \frac{D_{jj}}{\Sigma_{ii}^{1/2}}\left(\left|v_{i}(j)-\widetilde{v}_{i}(j)\right|+\left|v_{i}(j)\right|\mu_{0}\left(L_{B}\right)\mathcal{O}\left(\eta\right)\right)\\
 & \leq & \left(\frac{\mu_{2}\left(L_{B}\right)}{\mbox{relgap}{}_{i}}+\mu_{0}\left(L_{B}\right)\mu_{1}\left(L_{B}\right)\right)\mathcal{O}\left(\eta\right),\end{eqnarray*}
where we used (\ref{eq:bound on SVD perturbed, Cauchy}) and (\ref{eq:ratio of sing vals, perturb proof})
in the first inequality, and (\ref{eq:bound on sing vectors, perturbed proof})-(\ref{eq:bound on sing vector diff, perturbed proof})
in the last one. %
{} 

Proposition~\ref{pro:perturb for sing vecs of DBD} also implies
that $1/\left\Vert w_{i}\right\Vert \leq\sigma_{\max}^{1/2}\left(L^{\text{T}}L\right)=\mathcal{O}\left(1\right)$
and $1/\left\Vert \widehat{w_{i}}\right\Vert \leq\mathcal{O}\left(1\right)$
(since complete pivoting is used, $\left\Vert L\right\Vert =\mathcal{O}\left(1\right)$).
The proof now follows upon noting that the con-eigenvectors $z_{i}$
and $\widetilde{z_{i}}$ satisfy $z_{i}=Lw_{i}$, $\widetilde{z_{i}}=\widetilde{L}\widetilde{w_{i}}$,
and using (\ref{eq:bound on SVD perturbed, Cauchy}),\begin{eqnarray*}
\frac{\left\Vert z_{i}-\widetilde{z_{i}}\right\Vert }{\left\Vert z_{i}\right\Vert } & \leq & \frac{\left\Vert Lw_{i}-L\widetilde{w_{i}}\right\Vert }{\left\Vert Lw_{i}\right\Vert }+\left\Vert L\right\Vert \frac{\left\Vert \widetilde{w_{i}}\right\Vert }{\left\Vert L\widetilde{w_{i}}\right\Vert }\mathcal{O}\left(\eta\right)\\
 & \leq & \left\Vert L\right\Vert \left\Vert L^{-1}\right\Vert \left(\frac{\left\Vert w_{i}-\widetilde{w_{i}}\right\Vert }{\left\Vert w_{i}\right\Vert }+\mathcal{O}\left(\eta\right)\right)\\
 & \leq & \kappa\left(L\right)\left(\frac{\mu_{2}\left(L_{B}\right)}{\mbox{relgap}{}_{i}}+\mu_{0}\left(L_{B}\right)\mu_{1}\left(L_{B}\right)\right)\mathcal{O}\left(\eta\right).\end{eqnarray*}

%
{}
\end{proof}

\subsection{\label{sub:High-relative-accuracy}Proof of Theorem~\ref{thm:high accuracy of algorithm-1}
(high relative accuracy of Algorithm~\ref{Flo:Accurate-con-eigenvalue-decomposition})}

We now show that Algorithm~\ref{Flo:Accurate-con-eigenvalue-decomposition}
accurately computes the eigenvectors of $\overline{C}C$ (recall that
$C=\left(PL\right)D^{2}\left(PL\right)^{*}$), as long as the $n$
leading principal minors of $L^{T}L$ are well-conditioned and the
relative gap between the eigenvalues is not too small. 

Before proving Theorem~\ref{thm:high accuracy of algorithm-1}, we
first need several lemmas on graded matrices.
\begin{lem}
\label{lem:graded triangular matrices}Let $D$ be a positive definite
diagonal matrix with decreasing diagonal elements, and let $L$ and
$R$ denote nonsingular lower and upper triangular matrices. Then\[
\left\Vert DLD^{-1}\right\Vert \leq\left\Vert L\right\Vert ,\,\,\,\left\Vert \left(DLD^{-1}\right)^{-1}\right\Vert \leq\left\Vert L^{-1}\right\Vert ,\]
and \[
\left\Vert D^{-1}RD\right\Vert \leq\left\Vert R\right\Vert ,\,\,\,\left\Vert \left(D^{-1}RD\right)^{-1}\right\Vert \leq\left\Vert R^{-1}\right\Vert .\]
\end{lem}
\begin{proof}
Since the diagonal elements of $D$ are decreasing and $L$ is lower
triangular,\[
\left|\left(DLD^{-1}\right)_{ij}\right|\leq\frac{D_{ii}}{D_{jj}}\left|L_{ij}\right|\leq\left|L_{ij}\right|,\]
and \[
\left|\left(DLD^{-1}\right)_{ij}^{-1}\right|\leq\frac{D_{ii}}{D_{jj}}\left|\left(L^{-1}\right)_{ij}\right|\leq\left|\left(L^{-1}\right)_{ij}\right|.\]
Since the Frobenius norm is absolute, the first two inequalities in
Lemma~\ref{lem:graded triangular matrices} follow. The other two
inequalities can be shown in a similar fashion.
\end{proof}
~~~
\begin{lem}
\label{lem:R_0, for main alg proof}Let $D>0$ denote a diagonal matrix
with decreasing diagonal elements, and let $B$ denote a non-singular,
complex symmetric matrix with LU factorization $B=L_{B}L_{B}^{\text{T}}$.
Suppose that $DBD$ has the QR factorization $QR=DBD$. Then the upper
triangular matrix $R_{0}=D^{-2}R$ satisfies\[
\left\Vert DR_{0}D^{-1}\right\Vert \leq\left\Vert L_{B}\right\Vert ^{2},\,\,\,\left\Vert DR_{0}^{-1}D^{-1}\right\Vert \leq\left\Vert L_{B}^{-1}\right\Vert ^{2}.\]
Moreover, the $i$th left singular vector $u_{i}$ of $R$ (corresponding
to singular value $\Sigma_{ii}$) satisfies\begin{equation}
\left\Vert D^{-1}u_{i}\Sigma_{ii}^{1/2}\right\Vert \geq\left\Vert B^{-1}\right\Vert ^{-1/2}\kappa^{-1}\left(L_{B}\right).\label{eq:D^{-1}u_iSigma_ii, lemma}\end{equation}
\end{lem}
\begin{proof}
First note $DBD=\left(DL_{B}D^{-1}\right)\left(DL_{B}^{\text{T}}D\right)$.
Now suppose that $DL_{B}D^{-1}$ has the QR factorization $Q_{1}R_{1}=DL_{B}D^{-1}$.
Then $Q_{1}\left(R_{1}DL_{B}^{\text{T}}D\right)=QR$. Since $R_{1}DL_{B}^{\text{T}}D$
is upper triangular, there is a diagonal matrix $\Omega$ such that
$\left|\Omega_{ij}\right|=1$, $Q=Q_{1}\Omega^{-1}$, and $R=\Omega R_{1}DL_{B}^{\text{T}}D$.
Therefore, we have \[
R_{0}=D^{-2}R=\Omega D^{-2}R_{1}DL_{B}^{\text{T}}D.\]
It follows that $DR_{0}D^{-1}=\Omega\left(D^{-1}R_{1}D\right)L_{B}^{\text{T}}$,
and we obtain \begin{eqnarray*}
\left\Vert DR_{0}D^{-1}\right\Vert  & = & \left\Vert \Omega\left(D^{-1}R_{1}D\right)L_{B}^{\text{T}}\right\Vert \leq\left\Vert R_{1}\right\Vert \left\Vert L_{B}^{\text{T}}\right\Vert \leq\left\Vert L_{B}\right\Vert ^{2}.\end{eqnarray*}
In the first inequality, we use Lemma~\ref{lem:graded triangular matrices},
and in the last one we use $\left\Vert R_{1}\right\Vert =\left\Vert DL_{B}D^{-1}\right\Vert $
and Lemma~\ref{lem:graded triangular matrices}. Similarly, we have
\begin{eqnarray*}
\left\Vert DR_{0}^{-1}D^{-1}\right\Vert  & = & \left\Vert \left(L_{B}^{\text{T}}\right)^{-1}\left(D^{-1}R_{1}^{-1}D\right)\Omega^{-1}\right\Vert \leq\left\Vert R_{1}^{-1}\right\Vert \left\Vert \left(L_{B}^{\text{T}}\right)^{-1}\right\Vert \leq\left\Vert L_{B}^{-1}\right\Vert ^{2},\end{eqnarray*}
where we use Lemma~\ref{lem:graded triangular matrices} in the first
inequality, and $\left\Vert R_{1}^{-1}\right\Vert =\left\Vert DL_{B}^{-1}D^{-1}\right\Vert $
and Lemma~\ref{lem:graded triangular matrices} in the last one.

In order to prove the bound for $D^{-1}u_{i}\Sigma_{ii}^{1/2}$, we
first claim that $\left\Vert D^{-1}QD\right\Vert \leq\kappa\left(L_{B}\right)$.
Indeed, since $Q_{1}R_{1}=DL_{B}D^{-1}$ and $Q=Q_{1}\Omega^{-1}$,
we have \begin{eqnarray*}
\left\Vert D^{-1}QD\right\Vert  & = & \left\Vert D^{-1}Q_{1}D\right\Vert =\left\Vert L_{B}\left(D^{-1}R_{1}^{-1}D\right)\right\Vert \\
 & \leq & \left\Vert L_{B}\right\Vert \left\Vert D^{-1}R_{1}^{-1}D\right\Vert \\
 & \leq & \left\Vert L_{B}\right\Vert \left\Vert R_{1}^{-1}\right\Vert \leq\left\Vert L_{B}\right\Vert \left\Vert DL_{B}^{-1}D^{-1}\right\Vert \leq\left\Vert L_{B}\right\Vert \left\Vert L_{B}^{-1}\right\Vert .\end{eqnarray*}
In the above string of inequalities, we use of Lemma~\ref{lem:graded triangular matrices}
repeatedly.

Now, if $R$ has the SVD $R=U\Sigma V^{*}$, then $D^{-1}B^{-1}D^{-1}$
has the SVD $D^{-1}B^{-1}D^{-1}=V\Sigma^{-1}\left(QU\right)^{*}$
(recall that $QR=DBD$). Therefore, the left and right singular vectors
$u_{i}$ and $v_{i}$ of $R$ satisfy %
{} $\left(D^{-1}B^{-1}D^{-1}\right)\left(Qu_{i}\right)=\Sigma_{ii}^{-1}v_{i}$.
Since $B$ is complex symmetric, we may also assume (without loss
of generality) that $Qu_{i}=\overline{v_{i}}$. The bound on $\left\Vert D^{-1}u_{i}\Sigma_{ii}^{1/2}\right\Vert $
now follows from \begin{eqnarray*}
\Sigma_{ii}^{-1} & = & \left|\left(B^{-1}\left(D^{-1}\overline{v_{i}}\right),D^{-1}v_{i}\right)\right|\leq\left\Vert B^{-1}\right\Vert \left\Vert D^{-1}v_{i}\right\Vert \left\Vert D^{-1}\overline{v_{i}}\right\Vert \\
 & = & \left\Vert B^{-1}\right\Vert \left\Vert D^{-1}v_{i}\right\Vert ^{2}=\left\Vert B^{-1}\right\Vert \left\Vert D^{-1}Qu_{i}\right\Vert ^{2}\\
 & = & \left\Vert B^{-1}\right\Vert \left\Vert \left(D^{-1}QD\right)D^{-1}u_{i}\right\Vert ^{2}\\
 & \leq & \left\Vert B^{-1}\right\Vert \kappa^{2}\left(L_{B}\right)\left\Vert D^{-1}u_{i}\right\Vert ^{2}.\end{eqnarray*}
In the last inequality, we used the bound $\left\Vert D^{-1}QD\right\Vert \leq\kappa\left(L_{B}\right)$,
as shown in the previous paragraph.\end{proof}
\begin{lem}
\label{lem:R0 for main alg proof, 2}Let $DBD$, $Q$, $R$, and $L_{B}$
be as in Lemma~\ref{lem:R_0, for main alg proof}, and define $R_{1}=D^{-1}RD^{-1}$.
Then $\kappa\left(R_{1}\right)\leq\kappa^{2}\left(L_{B}\right)$,
and the $i$th right singular vector $v_{i}$ of $R$ (corresponding
to singular value $\Sigma_{ii}$) satisfies\[
Dv_{i}\Sigma_{ii}^{-1/2}=R_{1}^{-1}\left(D^{-1}u_{i}\Sigma_{ii}^{1/2}\right),\]
where \[
\left\Vert R_{1}\right\Vert \leq\left\Vert L_{B}\right\Vert ^{2},\,\,\,\left\Vert R_{1}^{-1}\right\Vert \leq\left\Vert L_{B}^{-1}\right\Vert ^{2}.\]
\end{lem}
\begin{proof}
Define $R_{0}=D^{-2}R$ and $R_{1}=D^{-1}RD^{-1}$. Then since $R_{1}=DR_{0}D^{-1}$,
the bounds for $\left\Vert R_{1}\right\Vert $ and $\left\Vert R_{1}^{-1}\right\Vert $
follow from Lemma~\ref{lem:R_0, for main alg proof}. Using (\ref{eq:relation between right/left sing vecs}),
we also have that \begin{eqnarray*}
Dv_{i}\Sigma_{ii}^{-1/2} & = & DR^{-1}u_{i}\Sigma_{ii}^{1/2}\\
 & = & \left(DR_{0}^{-1}D^{-1}\right)\left(D^{-1}u_{i}\Sigma_{ii}^{1/2}\right)\\
 & = & \left(DR_{0}D^{-1}\right)^{-1}\left(D^{-1}u_{i}\Sigma_{ii}^{1/2}\right)\\
 & = & R_{1}^{-1}\left(D^{-1}u_{i}\Sigma_{ii}^{1/2}\right).\end{eqnarray*}

\end{proof}
We now prove Theorem~\ref{thm:high accuracy of algorithm-1}.
\begin{proof}
First consider Step~$1$ of Algorithm~\ref{Flo:Accurate-con-eigenvalue-decomposition}.
From \cite{DEMMEL:1999}, the computed Cholesky factors $\widehat{D}$
and $\widehat{L}$ of $C$ satisfy \begin{equation}
\left|D_{ij}-\widehat{D_{ij}}\right|\leq\left|D_{ij}\right|\mathcal{O}\left(\epsilon\right),\,\,\,\,\left\Vert L-\widehat{L}\right\Vert \leq\mathcal{O}\left(\epsilon\right).\label{eq:bound on computed SVD of Cauchy}\end{equation}
We now examine the error in applying Algorithm~\ref{Flo:con-eig comp of XDX^*}
to compute the con-eigenvectors and con-eigenvalues. 

In Step~$1$ of Algorithm~\ref{Flo:con-eig comp of XDX^*}, the
computed matrix $\widehat{G}$ satisfies $\widehat{G}=\text{fl}\left(\widehat{D}\left(\widehat{L}^{\text{T}}\widehat{L}\right)\widehat{D}\right)=D\left(L^{\text{T}}L+E_{0}\right)D$,
where $\left\Vert E_{0}\right\Vert \leq\mathcal{O}\left(\epsilon\right)$
(recall that complete pivoting is used, so that $\left\Vert L\right\Vert =\mathcal{O}\left(1\right)$).

In Step~$2$ of Algorithm~\ref{Flo:con-eig comp of XDX^*}, a computed
upper triangular factor $\widehat{R}$ of $\widehat{G}$ is obtained
using the Householder QR algorithm. By Theorem~\ref{thm:backward error for QR}
in Section~\ref{sub:QR decomposition}, there is an orthogonal matrix
$Q$ such that \begin{equation}
Q\widehat{R}=D\left(L^{\text{T}}L+E_{2}\right)D,\label{eq:first perturbed matrix, proof of accuracy}\end{equation}
where $E_{2}=E_{0}+E_{1}$, $\left\Vert E_{1}\right\Vert \leq\rho\mu\psi\mathcal{O}\left(\epsilon\right)$,
and $\rho$, $\mu$, and $\psi$ are {}``pivot growth factors''
described in Section~\ref{sub:QR decomposition}. %
{} Now suppose that $L^{\text{T}}L+E_{2}$ has an LU factorization $L^{\text{T}}L+E_{2}=\widetilde{L_{B}}\widetilde{U_{B}}$.
By Lemma~\ref{lem:R_0, for main alg proof}, $R_{0}=D^{-2}\widehat{R}$
satisfies \begin{equation}
\left\Vert DR_{0}D^{-1}\right\Vert \leq\left\Vert L_{B}\right\Vert ^{2}+\mathcal{O}\left(\epsilon\right),\,\,\,\left\Vert DR_{0}^{-1}D^{-1}\right\Vert \leq\left\Vert L_{B}^{-1}\right\Vert ^{2}+\mathcal{O}\left(\epsilon\right),\label{eq:D^2R_0, proof of alg}\end{equation}
where we used that $\left\Vert L_{B}\right\Vert =\left\Vert \widetilde{L_{B}}\right\Vert +\mathcal{O}\left(\epsilon\right)$
and $\left\Vert L_{B}^{-1}\right\Vert =\left\Vert \widetilde{L_{B}}^{-1}\right\Vert +\mathcal{O}\left(\epsilon\right)$.

Step~$3$ of Algorithm~\ref{Flo:con-eig comp of XDX^*} involves
computing an approximate SVD $\widehat{R}\approx\widehat{U_{l}}\widehat{\Sigma}\widehat{U_{r}}^{*}$
using the modified one-sided Jacobi algorithm, applied from the left.
Note that, from (\ref{eq:first perturbed matrix, proof of accuracy}),
if $\widehat{R}$ has the (exact) SVD $\widehat{R}=U_{l}\Sigma U_{r}^{*}$,
then $QU_{l}$ and $U_{r}$ are the matrices of left and right singular
vectors of $D\left(L^{\text{T}}L+E_{2}\right)D$. Moreover, if $\widetilde{\Sigma_{ii}}$
denotes the exact singular value of $\widehat{R}$ (and $D\left(L^{\text{T}}L+E_{2}\right)D$),
then Proposition~\ref{pro:perturb for sing vecs of DBD} ensures
that \[
\frac{\left|\widetilde{\Sigma_{ii}}-\Sigma_{ii}\right|}{\widetilde{\Sigma_{ii}}}\leq\mu_{0}\left(L_{B}\right)\mathcal{O}\left(\left\Vert E_{2}\right\Vert \right)\leq\rho\mu\psi\mu_{0}\left(L_{B}\right)\mathcal{O}\left(\epsilon\right).\]
Now let $\widehat{\Sigma_{ii}}$ denote the computed singular value
of $\widehat{R}$ obtained via the one-sided Jacobi algorithm. Then
from Theorem~\ref{thm:Jacobi bound one-sided scaling} and the equality
$\widehat{R}=D^{2}R_{0}$, we also have that \[
\frac{\left|\widetilde{\Sigma_{ii}}-\widehat{\Sigma_{ii}}\right|}{\widetilde{\Sigma_{ii}}}\leq\nu_{0}\mathcal{O}\left(\epsilon\right).\]
Combining the previous two inequalities yields the bound on the computed
con-eigenvalues (recall that the exact and computed con-eigenvalues
satisfy $\lambda_{i}=\Sigma_{ii}$ and $\widehat{\lambda_{i}}=\widehat{\Sigma_{ii}}$).

Now let $\widehat{u_{i}}$ and $\widehat{\Sigma}_{ii}$ denote the
$i$th computed left singular vector and singular value of $\widehat{R}$.
Then by Theorem~\ref{thm:Jacobi bound one-sided scaling} and the
equality $\widehat{R}=D^{2}R_{0}$, there is an exact left singular
vector $u_{i}^{(1)}$ of $\widehat{R}$, corresponding to singular
value $\widetilde{\Sigma}_{ii}$, such that \begin{equation}
\left\Vert D^{-1}\left(\widehat{u_{i}}-u_{i}^{(1)}\right)\widetilde{\Sigma}_{ii}^{1/2}\right\Vert \leq\frac{\nu}{\mbox{relgap}_{i}}\mathcal{O}\left(\epsilon\right),\label{eq:bound 1 for Jacobi, proof of accuracy}\end{equation}
where $\nu$ is described in Theorem~\ref{thm:Jacobi bound one-sided scaling}
in Section~\ref{sub:Modified-one-sided-Jacobi} (see, in particular,
(\ref{eq:def of nu})).

Now define $R_{1}=D^{-1}\widehat{R}D^{-1}$ and $\widehat{R_{1}}=\text{fl}\left(\widehat{D}^{-1}\widehat{R}\widehat{D}^{-1}\right)$.
We show that the computed vector $\widehat{y_{i}}\approx\widehat{R_{1}}^{-1}\left(\widehat{D}^{-1}\widehat{u_{i}}\widehat{\Sigma}_{ii}^{1/2}\right)$,
obtained from Step~$4$ of Algorithm~\ref{Flo:con-eig comp of XDX^*},
is close to $y_{i}^{(1)}=R_{1}^{-1}\left(D^{-1}u_{i}^{(1)}\widetilde{\Sigma}_{ii}^{1/2}\right)$.
In particular, Step~$4$ involves computing an approximation $\widehat{y_{i}}$
to the triangular system $\widehat{R_{1}}y_{i}^{(0)}=\widehat{x_{i}}$,
where $\widehat{x_{i}}=\text{fl}\left(\widehat{D}^{-1}\widehat{u_{i}}\widehat{\Sigma}_{ii}^{1/2}\right)$
and \[
\widehat{R_{1}}=\text{fl}\left(\widehat{D}^{-1}\widehat{R}\widehat{D}^{-1}\right)=R_{1}+\delta R_{1},\,\,\,\frac{\left\Vert \delta R_{1}\right\Vert }{\left\Vert R_{1}\right\Vert }\leq\mathcal{O}\left(\epsilon\right).\]
In the above inequality for $\left\Vert \delta R_{1}\right\Vert $,
we used (\ref{eq:bound on computed SVD of Cauchy}). We will also
need the following expression for $R_{1}$, which follows from Lemma~\ref{lem:R_0, for main alg proof}:
\begin{equation}
R_{1}=DR_{0}D^{-1},\,\,\,\kappa\left(R_{1}\right)\leq\kappa^{2}\left(L_{B}\right)+\mathcal{O}\left(\epsilon\right).\label{eq:DR_0D^-{1}}\end{equation}
 Now, recall that $y_{i}^{(1)}$ is the exact solution of $R_{1}y_{i}^{(1)}=x_{i}$,
$x_{i}=D^{-1}u_{i}^{(1)}\widetilde{\Sigma}_{ii}^{1/2}$. Since $\widehat{y_{i}}$
is computed from the system $\widehat{R_{1}}y_{i}^{(0)}=\widehat{x_{i}}$
by backsubstitution, there is a matrix $\delta R_{2}$ such that $\left\Vert \delta R_{2}\right\Vert /\left\Vert R_{2}\right\Vert =\mathcal{O}\left(\epsilon\right)$
and $\left(R_{1}+\delta R_{2}\right)\widehat{y_{i}}=x_{i}+\left(\widehat{x_{i}}-x_{i}\right)$.
Therefore, \begin{eqnarray}
\frac{\left\Vert y_{i}^{(1)}-\widehat{y_{i}}\right\Vert }{\left\Vert y_{i}^{(1)}\right\Vert } & \leq & \kappa\left(R_{1}\right)\left(\frac{\left\Vert \left(\widehat{x_{i}}-x_{i}\right)\right\Vert }{\left\Vert x_{i}\right\Vert }+\mathcal{O}\left(\epsilon\right)\right)\nonumber \\
 & \leq & \kappa^{2}\left(L_{B}\right)\left(\frac{\left\Vert \left(\widehat{x_{i}}-x_{i}\right)\right\Vert }{\left\Vert x_{i}\right\Vert }+\mathcal{O}\left(\epsilon\right)\right),\label{eq:y_i-y_hat, main proof}\end{eqnarray}
where we used (\ref{eq:DR_0D^-{1}}) in the last inequality. To bound
$\left\Vert \left(\widehat{x_{i}}-x_{i}\right)\right\Vert $, we compute
that

\begin{eqnarray*}
\widehat{x_{i}}\left(j\right) & = & \text{fl}\left(\widehat{D}_{jj}^{-1}\widehat{u_{i}}\left(j\right)\widehat{\Sigma}_{ii}^{1/2}\right)=\widehat{D}_{jj}^{-1}\widehat{u_{i}}\left(j\right)\widehat{\Sigma}_{ii}^{1/2}\left(1+\mathcal{O}\left(\epsilon\right)\right)\\
 & = & D_{jj}^{-1}\widehat{u_{i}}\left(j\right)\widetilde{\Sigma}_{ii}^{1/2}\left(\frac{\widehat{\Sigma}_{ii}}{\widetilde{\Sigma}_{ii}}\right)^{1/2}\left(1+\mathcal{O}\left(\epsilon\right)\right)\\
 & = & D_{jj}^{-1}\widehat{u_{i}}\left(j\right)\widetilde{\Sigma}_{ii}^{1/2}\left(1+\left\Vert L_{B}^{-1}\right\Vert ^{2}\kappa\left(L_{B}\right)\mathcal{O}\left(\epsilon\right)\right),\end{eqnarray*}
where (\ref{eq:bound on computed SVD of Cauchy}) is used in the second
equality to bound $D_{jj}/\widehat{D_{jj}}$ and Proposition~\ref{pro:perturb for sing vecs of DBD}
is used in the last equality to bound $\widetilde{\Sigma}_{ii}/\widetilde{\Sigma}_{ii}$.
We also have from Lemma~\ref{lem:R_0, for main alg proof} (with
$B=L^{\text{T}}L+E_{2}$, $\left\Vert E_{2}\right\Vert \leq\rho\mu\psi\mathcal{O}\left(\epsilon\right)$)
that \[
\left\Vert x_{i}\right\Vert ^{-1}=\left\Vert D^{-1}u_{i}^{(1)}\widetilde{\Sigma}_{ii}^{1/2}\right\Vert ^{-1}\leq\left\Vert L^{-1}\right\Vert \kappa\left(L_{B}\right)+\mathcal{O}\left(\epsilon\right).\]
 Therefore, by (\ref{eq:bound 1 for Jacobi, proof of accuracy}),
we have \begin{eqnarray*}
\frac{\left\Vert \widehat{x_{i}}-x_{i}\right\Vert }{\left\Vert x_{i}\right\Vert } & \leq & \frac{\left\Vert D^{-1}\left(\widehat{u_{i}}-u_{i}^{(1)}\right)\widetilde{\Sigma}_{ii}^{1/2}\right\Vert }{\left\Vert D^{-1}u_{i}^{(1)}\widetilde{\Sigma}_{ii}^{1/2}\right\Vert }+\left\Vert L_{B}^{-1}\right\Vert ^{2}\kappa\left(L_{B}\right)\mathcal{O}\left(\epsilon\right)\\
 & \leq & \left(\frac{\nu\left\Vert L^{-1}\right\Vert \kappa\left(L_{B}\right)}{\mbox{relgap}_{i}}+\left\Vert L_{B}^{-1}\right\Vert ^{2}\kappa\left(L_{B}\right)\right)\mathcal{O}\left(\epsilon\right).\end{eqnarray*}
It finally follows from (\ref{eq:y_i-y_hat, main proof}) and the
above inequality that \begin{eqnarray}
\frac{\left\Vert \widehat{y_{i}}-y_{i}^{(1)}\right\Vert }{\left\Vert y_{i}^{(1)}\right\Vert } & = & \left(\frac{\nu\left\Vert L^{-1}\right\Vert \kappa^{3}\left(L_{B}\right)}{\mbox{relgap}_{i}}+\left\Vert L_{B}^{-1}\right\Vert ^{2}\kappa^{3}\left(L_{B}\right)\right)\mathcal{O}\left(\epsilon\right).\label{eq:(y_i-y_hat_i)}\end{eqnarray}

Now, from Lemma~\ref{lem:R0 for main alg proof, 2}, there exists
a right singular vector $v_{i}^{(1)}$ of $\widehat{R}$ such that
$y_{i}^{(1)}=Dv_{i}^{(1)}\widetilde{\Sigma}_{ii}^{-1/2}$. Moreover,
from (\ref{eq:first perturbed matrix, proof of accuracy}), $v_{i}^{(1)}$
is also a right singular vector of $D\left(L^{\text{T}}L+E_{2}\right)D$,
$\left\Vert E_{2}\right\Vert \leq\rho\mu\psi\mathcal{O}\left(\epsilon\right)$.
Therefore, Proposition~\ref{pro:perturb for sing vecs of DBD} ensures
that there is a right singular vector $v_{i}$ of $G=D\left(L^{\text{T}}L\right)D$
such that \begin{eqnarray}
\left\Vert D^{-1}\left(v_{i}-v_{i}^{(1)}\right)\widetilde{\Sigma}_{ii}^{1/2}\right\Vert  & \leq & \frac{\mu_{2}\left(L_{B}\right)}{\mbox{relgap}_{i}}\mathcal{O}\left(\left\Vert E_{2}\right\Vert \right)\leq\frac{\rho\mu\psi\mu_{2}\left(L_{B}\right)}{\mbox{relgap}_{i}}\mathcal{O}\left(\epsilon\right),\label{eq:bound 2 for Jacobi, proof of accuracy}\end{eqnarray}
We also have from Proposition~\ref{pro:perturb for sing vecs of DBD}
that $\left\Vert D^{-1}v_{i}^{(1)}\widetilde{\Sigma}_{ii}^{1/2}\right\Vert ^{-1}\leq\left\Vert L^{-1}\right\Vert $.
Therefore, letting $y_{i}=D^{-1}v_{i}\Sigma_{ii}^{1/2}$, it follows
from (\ref{eq:bound 2 for Jacobi, proof of accuracy}) that \begin{eqnarray}
\frac{\left\Vert y_{i}-y_{i}^{(1)}\right\Vert }{\left\Vert y_{i}\right\Vert } & = & \frac{\left\Vert y_{i}-y_{i}^{(1)}\right\Vert }{\left\Vert y_{i}^{(1)}\right\Vert }\frac{\left\Vert y_{i}^{(1)}\right\Vert }{\left\Vert y_{i}\right\Vert }\label{eq:y_i-y_tilde_i}\\
 & \leq & \left\Vert L^{-1}\right\Vert \frac{\rho\mu\psi\mu_{2}\left(L_{B}\right)}{\mbox{relgap}_{i}}\mathcal{O}\left(\epsilon\right),\end{eqnarray}
where we use that $\left\Vert y_{i}^{(1)}\right\Vert =\left\Vert y_{i}\right\Vert +\mathcal{O}\left(\epsilon\right)$
and $\widetilde{\Sigma_{ii}}^{1/2}=\Sigma_{ii}^{1/2}+\mathcal{O}\left(\epsilon\right)$.
Combining (\ref{eq:y_i-y_tilde_i}) and (\ref{eq:(y_i-y_hat_i)}),
we finally obtain\[
\frac{\left\Vert y_{i}-\widehat{y_{i}}\right\Vert }{\left\Vert y_{i}\right\Vert }\leq\left(\frac{\mu_{3}\left(L_{B}\right)}{\mbox{relgap}_{i}}+\left\Vert L_{B}^{-1}\right\Vert ^{2}\kappa^{3}\left(L_{B}\right)\right)\mathcal{O}\left(\epsilon\right),\]
where \[
\mu_{3}\left(L_{B}\right)=\left\Vert L^{-1}\right\Vert \left(\rho\mu\psi\mu_{2}\left(L_{B}\right)+\nu\kappa^{3}\left(L_{B}\right)\right).\]

In the final step of the algorithm, we compute an approximation $\widehat{z_{i}}=\text{fl}\left(\widehat{L}\widehat{y_{i}}\right)$
to the true con-eigenvector $z_{i}=Ly_{i}$. From (\ref{eq:bound on computed SVD of Cauchy}),
we have $\widehat{z_{i}}=L\widehat{y_{i}}+e_{i}$, where $\left\Vert e_{i}\right\Vert \leq\left\Vert L\right\Vert \left\Vert \widehat{y_{i}}\right\Vert \mathcal{O}\left(\epsilon\right)$.
Therefore, we obtain\begin{eqnarray*}
\frac{\left\Vert z_{i}-\widehat{z_{i}}\right\Vert }{\left\Vert z_{i}\right\Vert } & \leq & \frac{\left\Vert Ly_{i}-L\widehat{y_{i}}\right\Vert }{\left\Vert Ly_{i}\right\Vert }+\left\Vert L\right\Vert \frac{\left\Vert \widehat{y_{i}}\right\Vert }{\left\Vert L\widehat{y_{i}}\right\Vert }\mathcal{O}\left(\epsilon\right)\\
 & \leq & \left\Vert L\right\Vert \left\Vert L^{-1}\right\Vert \left(\frac{\left\Vert y_{i}-\widehat{y_{i}}\right\Vert }{\left\Vert y_{i}\right\Vert }+\mathcal{O}\left(\epsilon\right)\right)\\
 & \leq & \kappa\left(L\right)\left(\frac{\mu_{3}\left(L_{B}\right)}{\mbox{relgap}_{i}}+\left\Vert L_{B}^{-1}\right\Vert ^{2}\kappa^{3}\left(L_{B}\right)\right)\mathcal{O}\left(\epsilon\right).\end{eqnarray*}

\end{proof}

\subsection{Proof of Theorem~\ref{thm:truncation theorem}}

We only prove the error bounds for the computed con-eigenvector components
in Theorem~\ref{thm:truncation theorem} (the error bounds for the
computed con-eigenvalues follow in a similar fashion).

We need the following well-known result describing the sensitivity
of the eigenvalue problem for diagonalizable matrices. 
\begin{lem}
\label{lem:standard perturbation eigenvector}Assuming that $A$ has
simple eigenvalues, we consider its perturbation $A+E$. Let $Z_{0}=\left(\begin{array}{ccc}
z_{1} & \dots & z_{n}\end{array}\right)$ denote a matrix of unit eigenvectors of $A$, with corresponding
eigenvalues $\lambda_{1},\ldots,\lambda_{n}$. Then, for small enough
$\left\Vert E\right\Vert $, the $i$th unit eigenvectors $z_{i}$
and $\widetilde{z_{i}}$ of $A$ and $\widetilde{A}$ may be chosen
so that \[
\left\Vert z_{i}-\widetilde{z_{i}}\right\Vert \leq\kappa\left(Z_{0}\right)\frac{\mathcal{O}\left(\left\Vert E\right\Vert \right)}{\text{absgap}_{i}},\,\,\,\mbox{where}\,\,\,\text{absgap}_{i}=\min_{j\neq i}\left|\lambda_{i}-\lambda_{j}\right|.\]

\end{lem}
%
{}~~~

The next result shows that the matrix of eigenvectors of $\overline{C}C$
is well-conditioned. Recall that $\mathcal{O}\left(1\right)$ denotes
a constant that depends on $n$ only ($C$ has dimensions $n\times n$).
\begin{lem}
\label{lem:bound of eigenvector matrix}Let $C$ denote a positive-definite
Cauchy matrix, and let $Z_{0}$ denote the matrix of unit eigenvectors
of $\overline{C}C$. Then we have $\kappa\left(Z_{0}\right)\leq\kappa\left(L\right)\mu_{1}^{2}\left(L_{B}\right)\mathcal{O}\left(1\right)$.\end{lem}
\begin{proof}
From Section~\ref{sub:Derivation-of-algorithm}, we know that the
matrix $Z$ of (unnormalized) eigenvectors of $\overline{C}C$ is
given by $Z=P\overline{L}\left(D\overline{V}\Sigma^{-1/2}\right)$,
where $V$ is the matrix of right singular vectors of $D\left(L^{\text{T}}L\right)D$
and $L$ is the lower triangular matrix in the Cholesky factorization
$C=\left(PL\right)D^{2}\left(PL\right)^{*}$. Now define the matrix,
$Z_{0}$, of normalized eigenvectors $Z_{0}=Z\Omega^{-1}$, where
$\Omega_{ii}=\left\Vert z_{i}\right\Vert $. 

By Proposition~\ref{pro:perturb for sing vecs of DBD}, we have that
\[
\max\left\{ \left\Vert DV\Sigma^{-1/2}\right\Vert ,\left\Vert D^{-1}V\Sigma^{1/2}\right\Vert \right\} \leq\mu_{1}\left(L_{B}\right),\]
where $\mu_{1}\left(L_{B}\right)$ is defined in (\ref{eq:constants for bounds in proof})
and $\Sigma^{2}$ denotes the matrix of eigenvalues of $\overline{C}C$.
Therefore, $\left\Vert \Omega\right\Vert \leq\sqrt{n}\left\Vert Z\right\Vert \leq\sqrt{n}\left\Vert L\right\Vert \mu_{1}\left(L_{B}\right)$.
Also,\begin{eqnarray*}
\left\Vert Z^{-1}\right\Vert  & = & \left\Vert \left(Z^{-1}\right)^{*}\right\Vert =\left\Vert \left(\left(\Sigma^{1/2}\overline{V}^{-1}D^{-1}\right)\left(P\overline{L}\right)^{-1}\right)^{*}\right\Vert \\
 & \leq & \left\Vert L^{-1}\right\Vert \left\Vert \left(\Sigma^{1/2}\overline{V}^{-1}D^{-1}\right)^{*}\right\Vert =\left\Vert L^{-1}\right\Vert \left\Vert \left(D^{-1}\overline{V}\Sigma^{1/2}\right)\right\Vert \leq\left\Vert L^{-1}\right\Vert \mu_{1}\left(L_{B}\right),\end{eqnarray*}
where we used that $\overline{V}^{-1}=\overline{V}^{*}$ in the last
equality. It follows that\[
\left\Vert Z_{0}^{-1}\right\Vert \leq\left\Vert \Omega\right\Vert \left\Vert Z^{-1}\right\Vert \leq\sqrt{n}\kappa\left(L\right)\mu_{1}^{2}\left(L_{B}\right).\]
Finally, using the above inequality and the bound $\left\Vert Z_{0}\right\Vert \leq\sqrt{n}$
(recall that the column norms of $Z_{0}$ are unity), \[
\kappa\left(Z_{0}\right)\leq n\kappa\left(L\right)\mu_{1}^{2}\left(L_{B}\right).\]
 
\end{proof}
The next lemma is the key to proving Theorem~\ref{thm:truncation theorem}.
\begin{prop}
\label{pro:perturbation from partial cholesky}Suppose that Algorithm~\ref{Flo:Cauchy factorization alg-1}
produces, in exact arithemtic, the partial Cholesky factorization
$\widetilde{C}=\left(\widetilde{P}\widetilde{L}\right)\widetilde{D}^{2}\left(\widetilde{P}\widetilde{L}\right)$,
where $\widetilde{P}$ has dimension $m\times n$, $\widetilde{L}$
has dimension $n\times m$, and $\widetilde{D}$ has dimension $m\times m$.
Also assume that $\widetilde{D}_{mm}^{2}\leq\epsilon\Sigma_{ii}^{2}$
for some (simple) eigenvalue $\Sigma_{ii}^{2}$ of $\overline{C}C$
($1\leq i\leq m$). 

Then the $i$th unit eigenvectors $z_{i}$ and $\widetilde{z_{i}}$
of \textup{$\overline{C}C$ and $\overline{\widetilde{C}}\widetilde{C}$
may be chosen so that} \[
\left\Vert \widehat{z_{i}}-\widetilde{z_{i}}\right\Vert \leq\frac{\left\Vert C\right\Vert \kappa\left(L\right)\mu_{1}^{2}\left(L_{B}\right)}{\mbox{relgap}_{i}}\mathcal{O}\left(\epsilon\right),\,\,\,\mbox{where}\,\,\,\mbox{relgap}_{i}=\min_{j\neq i}\frac{\left|\Sigma_{ii}-\Sigma_{jj}\right|}{\Sigma_{ii}+\Sigma_{jj}},\]
where $\mu_{1}\left(L_{B}\right)$ is defined in (\ref{eq:constants for bounds in proof}).\end{prop}
\begin{proof}
After $m$ steps of Gaussian elimination with complete pivoting, \[
\widetilde{P}^{\text{T}}C\widetilde{P}=\left(\begin{array}{cc}
F_{11} & 0\\
F_{21} & F_{22}\end{array}\right)\left(\begin{array}{cc}
F_{11}^{*} & F_{21}^{*}\\
0 & F_{22}^{*}\end{array}\right)=\left(\begin{array}{cc}
F_{11}F_{11}^{*} & F_{11}F_{21}^{*}\\
F_{21}F_{11}^{*} & G_{22}\end{array}\right),\]
where $G_{22}=F_{21}F_{21}^{*}+F_{22}F_{22}^{*}$ and $\left|F_{22}\left(i,j\right)\right|\leq F_{11}\left(m,m\right)$,
$i,j=m+1,\ldots,n$ (recall that $C>0$). Now, $\widetilde{L}\widetilde{D}=\left(\begin{array}{cc}
F_{11} & F_{21}\end{array}\right)^{\text{T}}$ and \[
\widetilde{P}^{\text{T}}\widetilde{C}\widetilde{P}=\left(\begin{array}{c}
F_{11}\\
F_{21}\end{array}\right)\left(\begin{array}{cc}
F_{11}^{*} & F_{21}^{*}\end{array}\right)=\left(\begin{array}{cc}
F_{11}F_{11}^{*} & F_{11}F_{21}^{*}\\
F_{21}F_{11}^{*} & F_{21}F_{21}^{*}\end{array}\right).\]
Therefore, since $D_{mm}=F_{11}\left(m,m\right)$, \[
\left\Vert C-\widetilde{C}\right\Vert =\left\Vert P_{1}^{\text{T}}CP_{1}-P_{1}^{\text{T}}\widetilde{C}P_{1}\right\Vert \leq D_{mm}^{2}\mathcal{O}\left(1\right),\]
and \[
\left\Vert \overline{C}C-\overline{\widetilde{C}}\widetilde{C}\right\Vert \leq\left\Vert C\right\Vert D_{mm}^{2}\mathcal{O}\left(1\right).\]
It follows from Lemmas~\ref{lem:standard perturbation eigenvector}~and~\ref{lem:bound of eigenvector matrix}
that the $i$th eigenvectors $z_{i}$ and $\widetilde{z_{i}}$ of
$\overline{C}C$ and $\overline{\widetilde{C}}\widetilde{C}$ may
be chosen so that \begin{eqnarray*}
\left\Vert \widehat{z_{i}}-\widetilde{z_{i}}\right\Vert  & \leq & \max_{j\neq i}\frac{1}{\left|\Sigma_{ii}^{2}-\Sigma_{jj}^{2}\right|}\kappa\left(Z\right)\left\Vert C\right\Vert D_{mm}^{2}.\\
 & \leq & \frac{\kappa\left(Z\right)}{\mbox{relgap}_{i}}\left\Vert C\right\Vert \frac{D_{mm}^{2}}{\Sigma_{i}^{2}}\leq\frac{\left\Vert C\right\Vert \kappa\left(L\right)\mu_{1}^{2}\left(L_{B}\right)}{\mbox{relgap}_{i}}\mathcal{O}\left(\epsilon\right).\end{eqnarray*}
In the last inequality, we used that (see \cite{MATHIA:1997})\[
\max_{j\neq i}\frac{\Sigma_{ii}^{2}+\Sigma_{jj}^{2}}{\left|\Sigma_{ii}^{2}-\Sigma_{jj}^{2}\right|}\leq\max_{j\neq i}\frac{\Sigma_{ii}+\Sigma_{jj}}{\left|\Sigma_{ii}-\Sigma_{jj}\right|}.\]
 %
{}
\end{proof}
~~~

Finally, we are ready to prove Theorem~\ref{thm:truncation theorem}.
\begin{proof}
The proof of Theorem~\ref{thm:high accuracy of algorithm-1} shows
that the $i$th computed unit eigenvector $\widehat{\widetilde{z_{i}}}$
of $\overline{\widetilde{C}}\widetilde{C}$, $\widetilde{C}=\left(\widetilde{P}\widetilde{L}\right)\widetilde{D}^{2}\left(\widetilde{P}\widetilde{L}\right)^{\text{T}}$,
is close to an exact unit eigenvector $\widetilde{z_{i}}$, i.e.,
\[
\left\Vert \widehat{z_{i}}-\widetilde{z_{i}}\right\Vert \leq\kappa\left(L\right)\left(\frac{\mu_{3}\left(L_{B}\right)}{\mbox{relgap}_{i}}+\left\Vert L^{-1}\right\Vert ^{2}\kappa^{3}\left(L_{B}\right)\right)\mathcal{O}\left(\epsilon\right),\]
where $\mu_{3}\left(L_{B}\right)$ is defined in (\ref{eq:constants for bounds, 2}).
Also, Proposition~\ref{pro:perturbation from partial cholesky} implies
that there is an exact unit unit eigenvector $z_{i}$ of $\overline{C}C$
such that \[
\left\Vert \widetilde{z_{i}}-z_{i}\right\Vert \leq\frac{\left\Vert C\right\Vert \kappa\left(L\right)\mu_{1}^{2}\left(L_{B}\right)}{\mbox{relgap}_{i}}\mathcal{O}\left(\epsilon\right).\]
The claim follows from the above inequalities.
\end{proof}

\section{Comparison with related approaches for constructing optimal rational
approximations\label{sec:Comparison-with-related}}

Numerical approaches for finding near optimal rational approximations
originate in theoretical results of Adamyan, Arov, and Krein \cite{AD-AR-KR:1968,AD-AR-KR:1968a,AD-AR-KR:1971}.
In particular, given a periodic function $f\left(e^{2\pi ix}\right)\in L^{\infty}(0,1)$,
AAK theory yields an optimal {}``rational-like'' approximation $r_{M}\left(e^{2\pi ix}\right)$,\begin{equation}
r_{M}\left(z\right)=\frac{a_{0}+a_{1}z+a_{2}z^{2}+\dots}{\left(z-\zeta_{1}\right)\dots\left(z-\zeta_{M}\right)},\,\,\,\,\left|\zeta_{j}\right|<1,\label{eq:rational-like approx, AAK}\end{equation}
constructed from the left and right singular vectors corresponding
to the $M$th singular value, $\sigma_{M}$, of the infinite Hankel
matrix $H_{ij}=\hat{f}\left(i+j-1\right),\,\,\, i,j=1,2,\ldots$.
The numerator of $r_{M}\left(z\right)$ in (\ref{eq:rational-like approx, AAK})
is analytic in the unit disk. The approximation error satisfies \[
\max_{x}\left|f\left(e^{2\pi ix}\right)-r_{M}\left(e^{2\pi ix}\right)\right|=\sigma_{M},\]
where the number of poles $\zeta_{j}$ in (\ref{eq:rational-like approx, AAK})
equals the index $M$ of the singular value $\sigma_{M}$ (index counting
starts from zero). Moreover, the $L^{\infty}$-norm approximation
error is optimal among all functions of the form (\ref{eq:rational-like approx, AAK}).

In order to use AAK theory to compute near optimal rational approximations,
standard numerical approaches compute singular vectors of a truncated
 Hankel matrix. The poles of the rational approximation are obtained
as roots of a polynomial whose coefficients are the entries of the
singular vector. Such approaches have a long history of their own
and, in particular, let us mention the pioneering papers \cite{TREFET:1981,TRE-GUT:1983,TREFET:1983}.
A recent version (incorporating additional ideas) can be found in
\cite{DEU-TRE:2011}.

Instead of truncating the Hankel matrix, the approach of this paper
is based on the observation that it is always possible (see e.g. \cite{BEY-MON:2005,BEY-MON:2010,BE-LE-MO:2011p,D-B-H-M:2012p})
to construct a sub-optimal rational approximation, i.e., an approximation
with excessive number of poles for a desired accuracy. This leads
us to specialize AAK theory to proper rational functions $f\left(e^{2\pi ix}\right)$,
and to formulate the reduction problem (see Section~\ref{sub:Reduction-procedure}
and \cite[Section 6]{BEY-MON:2005}). Importantly, this results in
a con-eigenvalue problem of finite size and with no additional approximations.
Moreover, this formulation allows us to develop a numerical calculus
based on rational functions (numerical operations such as addition
and multiplication increase the number of poles; the reduction algorithm
is applied to keep their number near optimally small, see \cite{HA-BE-MO:2012}).
Early approaches of this type can be found in \cite{KUNG:1980,DE-GE-KA:1981,YOUNG:1983};
however, these algorithms may require extended precision for high
accuracy and also scale cubically in the number of original poles.

Comparing our approach with that in e.g. \cite{DEU-TRE:2011}, we
make two observations. First, to justify the truncation of an infinite
Hankel matrix, the Fourier coefficients have to decay below the desired
accuracy of approximation. Thus, for functions that have sharp transitions
(as in the example of Section~$3.2$) or singularities (as in the
example of Section~$3.1$), where the Fourier coefficients decay
slowly, this would require computing singular values of very large
matrices. In the examples of Sections~$3.1$~and~$3.2$, Hankel
matrices of size $\approx10^{7}\times10^{7}$ and $\approx10^{6}\times10^{6}$
would be needed in order to attain a comparable accuracy. This approach
would also require finding roots of polynomials with $\approx10^{7}$
and $\approx10^{6}$ coefficients, respectively.

Our second observation is that using Hankel matrices may require extended
precision arithmetic if high accuracy is desired, as is the case in
examples of Sections~$3.1$~and~$3.2$. Indeed, existing SVD algorithms
do not accurately compute small singular values of Hankel matrices.
Also, the roots of high degree polynomials (determined at the SVD
step) may be sensitive to perturbations in their coefficients. However,
when limited to approximating smooth functions, these {}``truncated
Hankel'' methods can yield surprisingly high accuracy since the errors
in the poles may be compensated by the residues. As far as we are
aware, truncated Hankel methods for constructing optimal rational
approximations for functions with singularities generally do not achieve
approximation errors better than $\approx10^{-4}$. In contrast, in
Section~$3.1$ we show that the reduction algorithm approximates
piecewise smooth functions with errors close to machine precision.

We also note that the results in \cite{HA-BE-MO:2012} (illustrated
in Section~$3.2$) demonstrate an effective numerical calculus based
on the reduction algorithm, capable of computing highly accurate solutions
to viscous Burgers' equation for viscosity as small as $10^{-5}$.
These solutions exhibit moving transitions regions of width $\approx10^{-5}$,
and computing them with high accuracy over long time intervals is
a nontrivial task for any numerical method. The con-eigenvalue algorithm
of this paper is critical to the high accuracy and efficiency of this
numerical calculus.

\section{Background on algorithms for high relative accuracy\label{sec:Background-on-algorithms}}

Here we provide necessary background on computing highly accurate
SVDs, as well some error bounds that are needed for the analysis of
the con-eigenvalue algorithm. Although the results we need in \cite{DEM-VES:1992,MATHIA:1996,D-G-E-S-V-D:1997,MATHIA:1997,DEMMEL:1999,HIGHAM:2000}
are only stated there for real-valued matrices, they carry over to
complex-valued matrices with minor modifications and are formulated
as such.

\subsection{\label{sub:Highly-accurate-SVDs}Accurate SVDs of matrices with rank-revealing
decompositions}

According to the usual perturbation theory for the SVD (see e.g. \cite{DAV-KAH:1969}),
perturbations $\delta A$ of a matrix $A$ change the $i$th singular
value $\sigma_{i}$ by $\delta\sigma_{i}$ and corresponding unit
eigenvector $u_{i}$ by $\delta u_{i}$, where (assuming for simplicity
that $\sigma_{i}$ is simple), \begin{equation}
\left|\delta\sigma_{i}\right|/\sigma_{1}\leq\left\Vert \delta A\right\Vert ,\,\,\,\,\left\Vert \delta u_{i}\right\Vert \leq\frac{\left\Vert \delta A\right\Vert }{\text{absgap}_{i}},\,\,\,\,\text{absgap}_{i}=\min_{i\neq j}\left|\sigma_{i}-\sigma_{j}\right|/\sigma_{1}.\label{eq:usual perturbation bound}\end{equation}
Therefore, small perturbations in the elements of $A$ may lead to
large relative changes in the small singular values and the associated
singular vectors. Moreover, since standard algorithms compute an SVD
of some nearby matrix $A+\delta A$, where $\left\Vert \delta A\right\Vert /\left\Vert A\right\Vert =\mathcal{O}\left(\epsilon\right)$,
the perturbation bound (\ref{eq:usual perturbation bound}) shows
that the computed small singular values and corresponding singular
vectors will be inaccurate. 

In contrast, the authors in \cite{D-G-E-S-V-D:1997} show that, for
many structured matrices, the $i$th singular value $\sigma_{i}\ll\sigma_{1}$
and the associated singular vector are robust with respect to small
perturbations of the matrix that preserve its underlying structure.
The sensitivity is instead governed by the $i$th \textit{relative}
gap \[
\mbox{relgap}_{i}=\min_{i\neq j}\frac{\left|\sigma_{i}-\sigma_{j}\right|}{\sigma_{i}+\sigma_{j}}.\]

More precisely, let us consider the class of matrices for which a
rank-revealing decomposition $A=XDY^{*}$ is available and may be
computed accurately. Here $X$ and $Y$ are $n\times m$ well-conditioned
matrices and $D$ is an $m\times m$ diagonal matrix that contains
any possible ill-conditioning of $A$. As is shown in \cite{D-G-E-S-V-D:1997},
a perturbation of $A=XDY^{*}$ that is of the form $A+\delta A=\left(X+\delta X\right)\left(D+\delta D\right)\left(Y+\delta Y\right)^{*}$,
where \begin{equation}
\frac{\left\Vert \delta X\right\Vert }{\left\Vert X\right\Vert }=\mathcal{O}\left(\epsilon\right),\,\,\,\,\frac{\left\Vert \delta Y\right\Vert }{\left\Vert Y\right\Vert }=\mathcal{O}\left(\epsilon\right),\,\,\,\,\frac{\left|\delta D_{ii}\right|}{\left|D_{ii}\right|}=\mathcal{O}\left(\epsilon\right),\label{eq:perturb of RRD factors}\end{equation}
changes the $i$th singular value $\sigma_{i}$ and associated left
(or right) singular vector $u_{i}$ by amounts $\delta\sigma_{i}$
and $\delta u_{i}$ bounded by \begin{equation}
\frac{\left|\delta\sigma_{i}\right|}{\sigma_{i}}\leq\max\left(\kappa\left(X\right),\kappa\left(Y\right)\right)\mathcal{O}\left(\epsilon\right),\,\,\,\,\left\Vert \delta u_{i}\right\Vert \leq\frac{\max\left(\kappa\left(X\right),\kappa\left(Y\right)\right)}{\mbox{relgap}_{i}}\mathcal{O}\left(\epsilon\right),\label{eq:bound for perturbed SVD, background}\end{equation}
where $\kappa(X)=\left\Vert X\right\Vert \left\Vert X^{\dagger}\right\Vert $
and $X^{\dagger}$ denotes the pseudo-inverse of $A$. One reason
this class of matrices is so useful is that Gaussian elimination with
complete pivoting (GECP) (or simple modifications) computes accurate
rank-revealing decompositions of many types of structured matrices
(see \cite{D-G-E-S-V-D:1997} and \cite{DEMMEL:1999}). Moreover,
small perturbations of such matrices that preserve their underlying
structure lead to small perturbations in the rank-revealing factors
and, therefore, small relative perturbations of the singular values.

Given the decomposition $A=XDY^{*}$, it is shown in \cite[Algorithm 3.1]{D-G-E-S-V-D:1997}
that an SVD of $A$ may be computed with high relative accuracy, and
with about the same cost as standard, less accurate SVD algorithms
for dense matrices. The key to this algorithm is the one-sided Jacobi
algorithm (briefly reviewed in Section~\ref{sub:Modified-one-sided-Jacobi}),
which, with an appropriate stopping criterion, accurately computes
the SVD of matrices of the form $DB$, where $D$ is diagonal (and
typically highly ill-conditioned) and $B$ is well-conditioned (see
\cite{DEM-VES:1992} and \cite{MATHIA:1996}). In particular, the
algorithm in \cite[Algorithm 3.1]{D-G-E-S-V-D:1997} yields computed
singular values $\widehat{\sigma_{i}}$ and left (or right) singular
vectors $\widehat{u_{i}}$ that satisfy\begin{equation}
\frac{\left|\sigma_{i}-\widehat{\sigma_{i}}\right|}{\sigma_{i}}\leq\max\left(\kappa\left(X\right),\kappa\left(Y\right)\right)\mathcal{O}\left(\epsilon\right),\label{eq:bound for computed SVD, background}\end{equation}
\begin{equation}
\left\Vert u_{i}-\widehat{u_{i}}\right\Vert \leq\frac{\max\left(\kappa\left(X\right),\kappa\left(Y\right)\right)}{\mbox{relgap}_{i}}\mathcal{O}\left(\epsilon\right),\label{eq:bound for computed sing. vect. SVD, background}\end{equation}

\subsection{LDU factorization of Cauchy matrices\label{sub:Accurate-factorization-of}}

In this section we review how a modification of GECP computes accurate
rank-revealing decompositions of Cauchy matrices \cite{DEMMEL:1999}. 

We describe Demmel's algorithm (see Algorithms~$3$~and~$4$ in
\cite{DEMMEL:1999} and Algorithm~$2.5$ in \cite{BO-KA-OL:2002})
for computing an accurate rank-revealing decomposition of a $n\times n$
positive-definite Cauchy matrix $C_{ij}=a_{i}b_{j}/\left(x_{i}+y_{j}\right)$
(note that Demmel refers to such matrices as quasi-Cauchy). The algorithm
is based on a modification of Gaussian elimination for computing,
in $\mathcal{O}\left(n^{2}\right)$ operations, the Cholesky factorization
$C=\left(PL\right)D\left(PD\right)^{*}$ of a positive-definite Cauchy
matrix (more generally, the algorithm computes an LDU factorization
for an arbitrary Cauchy matrix in $\mathcal{O}\left(n^{3}\right)$
operations). Here $P$ is a permutation matrix, $L$ is a unit lower
triangular matrix, and $D$ is a diagonal matrix with positive diagonal
elements. It is shown in \cite{DEMMEL:1999} that, remarkably, the
components of the LDU factors $\widehat{L}$, $\widehat{U}$, and
$\widehat{D}$ are computed to high relative accuracy, \begin{equation}
\left|\widehat{L}_{ij}-L_{ij}\right|\leq\left|L_{ij}\right|c_{n}\epsilon,\,\,\,\,\left|\widehat{U}_{ij}-U_{ij}\right|\leq c_{n}\left|U_{ij}\right|\epsilon,\,\,\,\,\left|\widehat{D}_{ii}-D_{ii}\right|\leq c_{n}\left|D_{ii}\right|\epsilon,\label{eq:bound on computed Cholesky Cauchy}\end{equation}
where $c_{n}$ is a modest-sized function of $n$. The basic reason
the algorithm achieves high relative accuracy is that the only operations
involved are multiplication and division of floating point numbers
(additions and subtractions in the algorithm involve only $x_{i}$
and $y_{j}$, which are assumed to be exact). 

We now review the basic idea behind the algorithm in \cite{DEMMEL:1999}.
First, ignoring pivoting for a moment, we assume that, after $k$
steps of Gaussian elimination, the Cauchy matrix is transformed to
the matrix $G^{(k)}$,

\[
G^{(k)}=\left(\begin{array}{cc}
G_{11}^{(k)} & G_{12}^{(k)}\\
0 & G_{22}^{(k)}\end{array}\right).\]
The elements of the Schur complement $G_{22}^{(k+1)}$ may be computed
from those of $G_{22}^{(k)}$ by using the recursion\begin{eqnarray}
G_{ij}^{(k)} & = & \left(\frac{x_{i}-x_{k}}{x_{i}+y_{k}}\right)\left(\frac{y_{j}-y_{k}}{y_{j}+x_{k}}\right)G_{ij}^{(k-1)},\,\,\, i,j=k+1,\ldots,n.\label{eq:schur complement recursion}\end{eqnarray}
Introducing pivoting, we observe that the matrix $G^{(k)}$ may be
obtained by applying Gaussian elimination to a Cauchy matrix $C^{(k)}=C^{(k)}\left(a^{(k)},b^{(k)},x^{(k)},y^{(k)}\right)$,
where $a^{(k)}$, $b^{(k)}$, $x^{(k)}$ and $y^{(k)}$ are permutations
of $a$, $b$, $x$ and $y$ corresponding to the row and column pivoting
of $C$. As long as the vectors $a$, $b$, $x$ and $y$ are permuted
according to the pivoting of $G^{(k)}$, the recursive formula (\ref{eq:schur complement recursion})
still holds. 

It is observed in \cite{DEMMEL:1999} that if $C$ is positive-definite
(and, therefore, only diagonal pivoting is needed), then the pivot
order may be determined in advance in $\mathcal{O}\left(n^{2}\right)$
operations by computing $\text{diag}\left(G^{(k)}\right)$ from formula
(\ref{eq:schur complement recursion}). Once the correct pivot order
is known, we do not need to compute the entire Schur complement $G^{(k)}$
to extract the components of $L$ and $U$, but only its $k$th row
and $k$th column. Indeed, we may use Algorithm~$2.5$ in \cite{BO-KA-OL:2002},
which uses the displacement structure of $C$, to compute an accurate
Cholesky decomposition in $\mathcal{O}\left(n^{2}\right)$ operations.
To see how, note that it easily follows from (\ref{eq:schur complement recursion})
that the Schur complement of a Cauchy matrix is a Cauchy matrix, \begin{equation}
G^{(k)}\left(i,j\right)=\frac{\alpha_{i}^{(k)}\beta_{j}^{(k)}}{x_{i}+y_{j}},\,\,\,\, i,j=k+1,\ldots,n,\label{eq:recursion for schur complement}\end{equation}
where the parameters $\alpha_{i}^{(k)}$ and $\beta_{i}^{(k)}$ satisfy
the recursion \begin{equation}
\alpha_{i}^{(k)}=\frac{x_{i}-x_{k}}{x_{i}+y_{k}}\alpha_{i}^{(k-1)},\,\,\,\,\beta_{i}^{(k)}=\frac{y_{i}-y_{k}}{y_{i}+x_{k}}\beta_{i}^{(k-1)},\,\,\,\, i=k+1,\ldots,n.\label{eq:recursion for alpha's, beta's}\end{equation}
Since the $k$th column $L\left(:,k\right)$ may be extracted from
$G^{(k)}\left(:,k\right)$, we therefore only require $\mathcal{O}\left(n\right)$
operations at each step of Gaussian elimination to compute $L\left(:,k\right)$.
Updating $\alpha_{i}^{(k)}$ and $\beta_{i}^{(k)}$ also requires
only $\mathcal{O}\left(n\right)$ operations. In Section~\ref{sub:Accurate-con-eigenvalue-decompositions}
(see Algorithms~\ref{Flo:algorithm for pivot order-1}~and~\ref{Flo:Cauchy factorization alg-1}),
we present an $\mathcal{O}\left(n\left(\log\delta^{-1}\right)^{2}\right)$
algorithm to compute con-eigenvalues greater than a user specified
cutoff $\delta$ and, as a result, yielding a fast algorithm for obtaining
nearly optimal rational approximations. Once an accurate LDU factorization
$C\approx\left(P\widehat{L}\right)\widehat{D}\left(P\widehat{D}\right)^{*}$
is available, an accurate SVD of $C$ may be obtained using the algorithm
in \cite[Algorithm 3.1]{D-G-E-S-V-D:1997}.

\subsection{\label{sub:QR decomposition}Rank-revealing decompositions of graded
matrices\label{sub:Rank-revealing-decompositions-of}}

We also review how a variant of the QR Householder algorithm with
complete pivoting computes accurate rank-revealing decompositions
of graded matrices \cite{HIGHAM:2000}. 

It is shown in \cite{HIGHAM:2000} that the Householder QR algorithm
with complete pivoting may be used to compute a rank-revealing decomposition
of a graded matrix of the form $A=D_{1}BD_{2}$. Here $D_{1}$ and
$D_{2}$ are diagonal matrices that account for the ill-conditioning
of $A$. Recall that the Householder QR algorithm uses repeated applications
of orthogonal matrices to reduce $A$ to an upper-triangular matrix
$R$. On the first step, the parameter $\beta_{1}$ and the vector
$v_{1}$ of the Householder reflection matrix $Q^{(1)}=I-\beta_{1}v_{1}v_{1}^{*}$
are chosen so that\[
Q^{(1)}\left(\begin{array}{c}
a_{11}\\
a_{21}\\
\vdots\\
a_{n1}\end{array}\right)=\left(\begin{array}{c}
a_{11}^{(1)}\\
0\\
\vdots\\
0\end{array}\right).\]
Consequently, the first application of $Q^{(1)}$ to $A$ results
in a matrix of the form \[
A^{(1)}=Q^{(1)}A=\left(\begin{array}{cccc}
a_{11}^{(1)} & a_{12}^{(1)} & \dots & a_{1n}^{(1)}\\
0 & a_{22}^{(1)} & \dots & a_{2n}^{(1)}\\
\vdots & \vdots & \ddots & \vdots\\
0 & a_{n2}^{(1)} & \dots & a_{nn}^{(1)}\end{array}\right).\]
This process is repeated on the $\left(n-1\right)\times\left(n-1\right)$
lower block $\left[a_{ij}^{(1)}\right]_{2\leq i,j\leq n}$ and, after
$n-1$ such steps, $A^{(n-1)}=Q^{(n-1)}\dots Q^{(1)}A=R$, where $R$
is upper triangular. In the version considered in \cite{HIGHAM:2000},
the rows of $A$ are first pre-sorted so that so that $\left\Vert A\left(1,:\right)\right\Vert _{\infty}\geq\dots\geq\left\Vert A\left(n,:\right)\right\Vert _{\infty}$.
The algorithm then proceeds as above, except that at each step, $k$,
column pivoting is performed to ensure that $\left\Vert A^{(k)}\left(k:n,k\right)\right\Vert _{2}\geq\dots\geq\left\Vert A^{(k)}\left(k:n,n\right)\right\Vert _{2}$.
Letting $P_{1}$ denote the row permutation matrix that pre-sorts
the rows of $A$, and letting $P_{2}$ denote the column permutation
matrix corresponding to the column pivoting, the QR Householder algorithm
produces the QR factorization $P_{1}AP_{2}=QR$. 

Following \cite{HIGHAM:2000}, we consider the error analysis of the
Householder algorithm (without pivoting) applied to $P_{1}AP_{2}$,
where $P_{1}$ and $P_{2}$ are chosen so that no column or row exchanges
are necessary (e.g. the matrix $A$ is pre-pivoted). Assume that the
matrix $P_{1}AP_{2}$ may be factored as $P_{1}AP_{2}=D_{1}BD_{2}$,
where $D_{1}$ and $D_{2}$ are diagonal matrices, and that the Householder
algorithm, applied to the row-scaled matrix $C=D_{1}B$, produces
intermediate matrices $C^{(k)}$ with columns $c_{j}^{(k)}$. Finally,
define the quantities $\rho$, $\mu$, and $\psi$ by \begin{equation}
\rho=\max_{i}\frac{\max_{j,k}\left|c_{ij}^{(k)}\right|}{\max_{j}\left|c_{ij}\right|},\,\,\,\,\mu=\max_{k}\max_{j\geq k}\frac{\left\Vert c_{j}^{(k)}\left(k:m\right)\right\Vert }{\left\Vert c_{k}^{(k)}\left(k:m\right)\right\Vert },\,\,\,\,\psi=\max_{{1\leq i\leq n\atop i\leq k\leq n}}\frac{\max_{j}\left|c_{kj}\right|}{\max_{j}\left|c_{ij}\right|}.\label{eq:def of pivots for QR analysis}\end{equation}
The above quantities measure the extent to which the Householder algorithm
preserves the scaling in the intermediate matrices $A^{(k)}$, and
are almost always small (this is analogous to the pivot growth factor
in Gaussian elimination with row pivoting). It is shown in \cite{HIGHAM:2000}
that 
\begin{thm}
\label{thm:backward error for QR}Suppose that $A$ is pre-pivoted,
and the Householder algorithm is used to compute the upper triangular
matrix $\widehat{R}$ of the QR decomposition. Then there is an orthogonal
matrix $Q$ such that $Q\widehat{R}=D_{1}\left(B+\delta B\right)D_{2},$
where $\delta B$ satisfies  \[
\left\Vert \delta B\right\Vert \leq\rho\psi\mu\left\Vert B\right\Vert \mathcal{O}\left(\epsilon\right),\]
and $\rho$, $\mu$, and $\psi$ are defined in (\ref{eq:def of pivots for QR analysis}).
\end{thm}
In \cite{HIGHAM:2000} Theorem~\ref{thm:backward error for QR} is
combined with the theory developed in \cite{D-G-E-S-V-D:1997} (e.g.,
see Theorems~$4.1$~and~$4.2$ in \cite{D-G-E-S-V-D:1997}) to
show that the QR algorithm with complete pivoting produces accurate
rank revealing decompositions of graded matrices of the form $A=D_{1}BD_{2}$,
as long as the principal minors of $B$ are well-conditioned and the
diagonal elements of $D_{1}$ and $D_{2}$ are approximately decreasing
in magnitude. 
\begin{rem*}
Instead of pre-sorting the rows of $A$ and applying the Householder
algorithm with column pivoting, one may also use a version of the
Householder algorithm in which both row and column pivoting is employed
(see \cite{HIGHAM:2000} for more details). Gaussian elimination with
complete pivoting may also be used to obtain accurate rank-revealing
decompositions of graded matrices \cite{D-G-E-S-V-D:1997}. 
\end{rem*}

\subsection{Modified one-sided Jacobi algorithm \label{sub:Modified-one-sided-Jacobi}}

The heart of the algorithm in \cite[Algorithm 3.1]{D-G-E-S-V-D:1997}
is the modified one-sided Jacobi algorithm, which accurately computes
the SVD of matrices of the form $DB$ and $BD$, where $D$ is diagonal
and typically highly graded, and $B$ is well-conditioned (see \cite{DEM-VES:1992},
\cite{MATHIA:1996}, \cite{DRM-VES:2007,DRM-VES:2007a}). Although
we focus on the one-sided Jacobi algorithm as applied to $G=BD$,
analogous considerations apply to $G=DB$ by replacing $G$ by $G^{*}$.
The one-sided Jacobi algorithm works by applying a sequence of Jacobi
matrices $J_{1},\ldots,J_{M}$ to $G$ from the right (i.e., the same
side as the scaling, which ensures that components of the right singular
vectors are computed with high relative accuracy). Each Jacobi matrix
$J$ is chosen to orthogonalize two selected columns, and one sweep
consists of orthogonalizing columns in the order $(1,1),(1,2),\ldots,(1,n)$,
followed by columns $(2,3),(2,4),\ldots,(2,n)$, and so on. Sweeps
are repeated until all the columns are orthogonal to each other to
within the bound \[
G\left(J_{1}\cdots J_{M}\right)=W,\,\,\,\frac{\left|w_{i}^{*}w_{j}\right|}{\left|w_{i}^{*}w_{i}\right|^{1/2}\left|w_{i}^{*}w_{i}\right|^{1/2}}\leq n\epsilon,\,\,\,\text{if }i\neq j.\]
This stopping criterion is used to ensure that even the smallest singular
values are computed with high relative accuracy. The SVD of $G=U\Sigma V^{*}$
immediately follows by taking $\Sigma_{ii}=W\left(:,i\right)$, $V=W/\Sigma$,
and $U=\left(J_{1}J_{2}\cdots J_{M}\right)^{*}$. 

It will be crucial for the error bounds developed in this paper that
the components of the left singular vectors of $DB$ (or the right
singular vectors of $BD$) scale in a way similar to $D$, and are
computed accurately relative to this scaling. At each step $m$ of
the Jacobi algorithm, we write $\left(J_{0}\cdots J_{m}\right)G=B_{m}D_{m}$,
where the columns of $B_{m}$ have unit $l^{2}$-norm and the matrix
$D_{m}$ is diagonal. Defining \begin{equation}
\nu_{0}=\max_{1\leq m\leq M}\kappa_{2}\left(B_{m}\right),\label{eq:definition of nu, jacobi}\end{equation}
we then have the following result from \cite{MATHIA:1996} and \cite{DEM-VES:1992}.
\begin{thm}
\label{thm:Jacobi bound one-sided scaling}Let $G=DB$ be a $n\times n$
full-rank, complex-valued matrix, where the diagonal matrix $D$ is
chosen so that the $l^{2}$-norm of each column of $B$ is unity.
Suppose that one-sided Jacobi algorithm is used to compute an approximation
$\widehat{u_{i}}$ to the $i$th left singular vector $u_{i}$ of
$G$, corresponding to singular value $\Sigma_{ii}$, and the iteration
converges after $M$ sweeps. Then the following error bound holds
on the computed components of $u_{i}$: \begin{equation}
\left|u_{i}\left(j\right)-\widehat{u_{i}}\left(j\right)\right|\leq\min\left\{ \frac{D_{jj}}{\sqrt{\Sigma_{ii}}},\frac{\sqrt{\Sigma_{ii}}}{D_{jj}}\right\} \left(\frac{\rho\left(M,n\right)\nu_{0}^{2}}{\mbox{relgap}_{i}}\epsilon+\mathcal{O}\left(\epsilon^{2}\right)\right),\label{eq: estimateForLeftSingularComponents}\end{equation}
where \textup{\[
\mbox{relgap}_{i}=\frac{\left|\sigma_{i}-\sigma_{j}\right|}{\sigma_{i}+\sigma_{j}},\]
} $\rho\left(M,n\right)$ is proportional to $M\cdot n^{3/2}$, and
$\nu_{0}$ in defined in (\ref{eq:definition of nu, jacobi}). Moreover,
the computed singular value $\widetilde{\Sigma_{ii}}$ satisfies\[
\frac{\left|\Sigma_{ii}-\widetilde{\Sigma_{ii}}\right|}{\Sigma_{ii}}\leq\nu_{0}\mathcal{O}\left(\epsilon\right).\]

\end{thm}
~~~For convenience, we define\begin{equation}
\nu=\rho\left(M,n\right)\nu_{0}^{2}.\label{eq:def of nu}\end{equation}
 The following result (see Theorem~$3.6$ of \cite{MATHIA:1997})
will also be needed. 
\begin{thm}
\label{thm:perturbation, D^2A}Let $G=DB$, where $B$ is non-singular
and $D$ is diagonal. Then the $i$th left singular vector of $G$,
corresponding to the simple singular value $\Sigma_{ii}$, satisfies
\[
\left|u_{i}\left(j\right)\right|\leq\frac{\max\left\{ \left\Vert B^{-1}\right\Vert ,\left\Vert B\right\Vert \right\} }{\text{relgap}_{i}}\min\left\{ \frac{D_{jj}}{\sqrt{\Sigma_{ii}}},\frac{\sqrt{\Sigma_{ii}}}{D_{jj}}\right\} .\]
 Moreover, there is a left singular vector $\widetilde{u_{i}}$ of
the perturbed matrix $G+\delta G=D\left(B+\delta B\right)$ that satisfies\begin{eqnarray*}
\left|u_{i}\left(j\right)-\widetilde{u_{i}}\left(j\right)\right| & \leq & \frac{\max\left\{ \left\Vert B^{-1}\right\Vert ^{2},\kappa\left(B\right)\right\} }{\text{relgap}_{i}}\min\left\{ \frac{D_{jj}}{\sqrt{\Sigma_{ii}}},\frac{\sqrt{\Sigma_{ii}}}{D_{jj}}\right\} \mathcal{O}\left(\left\Vert \delta B\right\Vert \right).\end{eqnarray*}

\end{thm}
~~~

\section{Appendix: Proof of Proposition~\ref{pro:perturb for sing vecs of DBD}}

\subsection{Overview of the proof  }

Proposition~\ref{pro:perturb for sing vecs of DBD} concerns how
graded perturbations of the form $D\left(B+\delta B\right)D$ perturb
the singular vectors of $DBD$, where the diagonal of $D>0$ is decreasing
and $B$ is complex symmetric and non-singular. As in \cite{D-G-E-S-V-D:1997},
we analyze the SVD of $DBD$ through the LU factorization $B=L_{B}L_{B}^{\text{T}}$.
In particular, we show that perturbations $D\left(\delta B\right)D$
of $DBD$ result in graded perturbations $\delta U$ of the singular
vectors: \[
\max\left\{ \left\Vert D\left(\delta U\right)\Sigma^{-1/2}\right\Vert ,\left\Vert D^{-1}\left(\delta U\right)\Sigma^{1/2}\right\Vert \right\} \leq c_{\delta U}\left(L_{B},U_{B}\right)\mathcal{O}\left(\left\Vert \delta B\right\Vert \right).\]
Here $c_{\delta U}\left(L_{B}\right)$ depends on the condition number
of $L_{B}$.

We now discuss the basic scaling considerations behind the proof of
Proposition~\ref{pro:perturb for sing vecs of DBD}. First, if $\kappa\left(L_{B}\right)$
is not too large, then $DBD$ has the rank-revealing decomposition
$DBD=XD^{2}Y$, where $X=\left(DL_{B}D^{-1}\right)$ and $Y=D^{-1}L_{B}^{\text{T}}D$.
Therefore, letting $X=QR$ denote a QR factorization of $X$ and setting
$A=\left(D^{-2}RD^{2}\right)Y^{*}$, it follows that $DBD=QD^{2}A$,
where \[
\kappa\left(X\right)\leq\kappa\left(L_{B}\right),\,\,\,\kappa\left(Y\right)\leq\kappa\left(L_{B}\right),\,\,\,\kappa\left(A\right)\leq\kappa^{2}\left(L_{B}\right).\]

From theory developed in \cite{MATHIA:1997} and \cite{DEM-VES:1992}
(see Lemma~\ref{lem:grading of sing vecs, D^2A}), the matrix of
left singular vectors $U_{l}$ of $D^{2}A=U_{l}\Sigma U_{r}^{*}$
is graded in the sense that \begin{equation}
\max\left\{ \left\Vert D^{-1}U\Sigma^{1/2}\right\Vert ,\left\Vert DU\Sigma^{-1/2}\right\Vert \right\} \leq\max\left\{ \left\Vert L_{B}^{-1}\right\Vert ^{2},\left\Vert L_{B}\right\Vert ^{2}\right\} .\label{eq:overview, grading 1}\end{equation}
It also can be shown (Lemma~\ref{lem:bound on Q components-1}) that
in the QR factorization $QR=\left(DL_{B}D^{-1}\right)$, the elements
of $Q$ are also graded in the sense that \begin{equation}
\max\left\{ \left\Vert D^{-1}QD\right\Vert ,\left\Vert DQD^{-1}\right\Vert \right\} \leq\kappa\left(L_{B}\right).\label{eq:overview, grading 2}\end{equation}
Finally, from (\ref{eq:overview, grading 1}) and (\ref{eq:overview, grading 2}),
it is not difficult to show that the matrix $U=QU_{l}$ of left singular
vectors of $DBD$ is also graded, \[
\max\left\{ \left\Vert D^{-1}\left(QU_{l}\right)\Sigma^{1/2}\right\Vert ,\left\Vert D\left(QU_{l}\right)\Sigma^{-1/2}\right\Vert \right\} \leq\mu_{1}\left(L_{B}\right),\]
where the factor $\mu_{1}\left(L_{B}\right)$ is defined by \[
\mu_{1}\left(L_{B}\right)=\kappa\left(L_{B}\right)\max\left\{ \left\Vert L_{B}^{-1}\right\Vert ^{2},\left\Vert L_{B}\right\Vert ^{2}\right\} .\]

\subsection{Preliminary Lemmas}

Before proceeding to the proof of Proposition~\ref{pro:perturb for sing vecs of DBD},
we first need several lemmas on graded matrices. 
\begin{lem}
\label{lem:bound on Q components-1}Let $L$ denote a nonsingular
lower triangular matrix, and let $D$ denote a diagonal matrix with
positive, decreasing diagonal elements. Then if $A=DLD^{-1}$ has
the QR factorization $A=QR$, the elements of $Q$ are bounded by
\[
\max\left\{ \left\Vert D^{-1}QD\right\Vert ,\left\Vert DQD^{-1}\right\Vert \right\} \leq\kappa\left(L\right).\]
\end{lem}
\begin{proof}
Rearranging $QR=DLD^{-1}$,\[
D^{-1}QD=L\left(D^{-1}R^{-1}D\right).\]
Since $R^{-1}$ is upper triangular and the elements of $D>0$ are
decreasing, Lemma~\ref{lem:graded triangular matrices} implies that\[
\left\Vert D^{-1}QD\right\Vert \leq\left\Vert L\right\Vert \left\Vert R^{-1}\right\Vert \leq\kappa\left(L\right).\]
In the last inequality, we used that $\left\Vert R^{-1}\right\Vert =\left\Vert DLD^{-1}\right\Vert \leq\left\Vert L\right\Vert $
(again, by Lemma~\ref{lem:graded triangular matrices}). 

To bound $\left\Vert DQD^{-1}\right\Vert $, we rearrange $QR=DLD^{-1}$
as \[
D^{-1}Q^{-1}D=\left(D^{-1}RD\right)L^{-1}.\]
Also, since $Q^{-1}=Q^{*},$ we calculate that \begin{eqnarray*}
\left\Vert D^{-1}Q^{-1}D\right\Vert  & = & \left\Vert D^{-1}Q^{*}D\right\Vert =\left\Vert \left(DQD^{-1}\right)^{*}\right\Vert \\
 & = & \left\Vert \left(DQD^{-1}\right)\right\Vert .\end{eqnarray*}
Therefore,\[
\left\Vert \left(DQD^{-1}\right)\right\Vert =\left\Vert \left(D^{-1}RD\right)L^{-1}\right\Vert \leq\left\Vert R\right\Vert \left\Vert L^{-1}\right\Vert \leq\kappa\left(L\right),\]
where we again used Lemma~\ref{lem:graded triangular matrices}.\end{proof}
\begin{lem}
\label{lem:mult of graded matrices-1}Let $D$ and $\Sigma$ denote
diagonal matrices, and suppose that the matrices $Q$ and $U$ satisfy
\[
\max\left\{ \left\Vert D^{-1}QD\right\Vert ,\left\Vert DQD^{-1}\right\Vert \right\} \leq c_{Q},\,\,\,\,\max\left\{ \left\Vert D^{-1}U\Sigma^{1/2}\right\Vert ,\left\Vert DU\Sigma^{-1/2}\right\Vert \right\} \leq c_{U}.\]
Then \[
\max\left\{ \left\Vert D^{-1}\left(QU\right)\Sigma^{1/2}\right\Vert ,\left\Vert D\left(QU\right)\Sigma^{-1/2}\right\Vert \right\} \leq c_{Q}c_{U}.\]
\end{lem}
\begin{proof}
We have that \[
\left\Vert D\left(QU\right)\Sigma^{-1/2}\right\Vert =\left\Vert \left(DQD^{-1}\right)\left(DU\Sigma^{-1/2}\right)\right\Vert \leq\left\Vert \left(DQD^{-1}\right)\right\Vert \left\Vert \left(DU\Sigma^{-1/2}\right)\right\Vert \leq c_{Q}c_{U}.\]
The reverse inequality follows from \[
\left\Vert D^{-1}\left(QU\right)\Sigma^{1/2}\right\Vert =\left\Vert \left(D^{-1}QD\right)\left(D^{-1}U\Sigma^{1/2}\right)\right\Vert \leq c_{Q}c_{U}.\]

\end{proof}
The following lemma (see Corollary $3.3$ and Proposition~3.6 in
\cite{MATHIA:1997}) shows that the matrix of left singular vectors
of $D^{2}A$ are graded in a particular way.
\begin{lem}
\label{lem:grading of sing vecs, D^2A}Suppose that $D$ is a positive-definite
diagonal matrix, and $A$ is non-singular. Then if $D^{2}A$ has the
SVD $D^{2}A=U_{l}\Sigma U_{r}$, \[
\left\Vert D^{-2}U_{l}\Sigma\right\Vert \leq\left\Vert A\right\Vert ,\,\,\,\left\Vert \left(D^{2}U_{l}\Sigma^{-1}\right)\right\Vert \leq\left\Vert A^{-1}\right\Vert ,\]
and \[
\left|\left(U_{l}\right)_{ij}\right|\leq\kappa\left(A\right)\min\left\{ \frac{D_{ii}^{2}}{D_{jj}^{2}},\frac{D_{jj}^{2}}{D_{ii}^{2}}\right\} .\]
%
{}
\end{lem}
~~~
\begin{lem}
\label{lem:factorization lemma, perturb}Suppose that $B$ is complex
symmetric, and has the LU factorization $B=L_{B}L_{B}^{\text{T}}$.
Define\textup{ \[
\mu_{1}\left(L_{B}\right)=\kappa\left(L_{B}\right)\max\left\{ \left\Vert L_{B}^{-1}\right\Vert ^{2},\left\Vert L_{B}\right\Vert ^{2}\right\} .\]
}Then the right singular vectors of $DBD$ and $D\left(B+\delta B\right)D$
coincide with the right singular vectors of $D^{2}A$ and $D^{2}\left(A+\delta A\right)$,
where the matrices $A$ and $\delta A$ satisfy \[
\left\Vert A\right\Vert \leq\left\Vert L_{B}\right\Vert ^{2},\,\,\,\,\left\Vert A^{-1}\right\Vert \leq\left\Vert L_{B}^{-1}\right\Vert ^{2},\,\,\,\,\left\Vert \delta A\right\Vert \leq\kappa^{3}\left(L_{B}\right)\mathcal{O}\left(\left\Vert \delta B\right\Vert \right).\]
\end{lem}
\begin{proof}
Note that \begin{eqnarray*}
DBD & = & D\left(L_{B}L_{B}^{\text{T}}\right)D\\
 & = & \left(DL_{B}D^{-1}\right)D^{2}\left(D^{-1}L_{B}^{\text{T}}D\right)\\
 & = & =L_{1}D^{2}L_{1}^{\text{T}}.\end{eqnarray*}
Now let $L_{1}=QR$ denote the QR factorization of $L_{1}$, and define
$A=\left(D^{-2}RD^{2}\right)L_{1}^{\text{T}}$. Then it follows that
$QD^{2}A=L_{1}D^{2}L_{1}^{\text{T}}$, and the singular values of
$D^{2}A$ and $DBD$ coincide. Also, using Lemma~\ref{lem:graded triangular matrices}
we calculate that \[
\left\Vert A\right\Vert \leq\left\Vert L_{B}\right\Vert ^{2},\,\,\,\,\left\Vert A^{-1}\right\Vert \leq\left\Vert L_{B}^{-1}\right\Vert ^{2}.\]

We now consider $D\left(B+\delta B\right)D$. First note that $B+\delta B$
has the LU factorization $B+\delta B=\left(L_{B}+\delta L_{B}\right)\left(L_{B}^{\text{T}}+\delta U_{B}\right)$,
where (see \cite{STEWAR:1993}) \begin{equation}
\max\left\{ \left\Vert \delta L_{B}\right\Vert ,\left\Vert \delta U_{B}\right\Vert \right\} \leq\left\Vert L_{B}^{-1}\right\Vert \kappa\left(L_{B}\right)\left\Vert \delta B\right\Vert .\label{eq:}\end{equation}
Therefore, \[
D\left(B+\delta B\right)D=\left(L_{1}+\delta L_{1}\right)D^{2}\left(L_{1}^{\text{T}}+\delta U_{1}\right),\]
where \begin{equation}
\left\Vert \delta L_{1}\right\Vert =\left\Vert D\delta L_{B}D^{-1}\right\Vert \leq\left\Vert \delta L_{B}\right\Vert \leq\left\Vert L_{B}^{-1}\right\Vert \kappa\left(L_{B}\right)\left\Vert \delta B\right\Vert .\label{eq:deltaL_{1}}\end{equation}
In the first inequality above, we used Lemma~\ref{lem:graded triangular matrices}.
Similarly, \begin{equation}
\left\Vert \delta U_{1}\right\Vert \leq\left\Vert L_{B}^{-1}\right\Vert \kappa\left(L_{B}\right)\left\Vert \delta B\right\Vert .\label{eq:deltaU_{1}}\end{equation}
Now let $L_{1}+\delta L_{1}$ have the QR factorization $L_{1}+\delta L_{1}=\widetilde{Q}\left(R+\delta R\right)$.
Then from \cite{STEWAR:1993},\[
\left\Vert \delta R\right\Vert \leq\kappa\left(L_{1}\right)\left\Vert \delta L_{1}\right\Vert \leq\left\Vert L_{B}^{-1}\right\Vert \kappa^{2}\left(L_{B}\right)\left\Vert \delta B\right\Vert .\]
Therefore, if we define \begin{eqnarray*}
A+\delta A & = & \left(D^{-2}\left(R+\delta R\right)D^{2}\right)\left(L_{1}^{\text{T}}+\delta U_{1}\right)\\
 & = & A+\left(D^{-2}\delta RD^{2}\right)L_{1}^{\text{T}}+\left(D^{-2}RD^{2}\right)\delta U_{1},\end{eqnarray*}
we may use (\ref{eq:deltaL_{1}}) and (\ref{eq:deltaU_{1}}) to bound\begin{eqnarray*}
\left\Vert \delta A\right\Vert  & \leq & \left\Vert L_{B}^{-1}\right\Vert \kappa^{2}\left(L_{B}\right)\left\Vert L_{B}\right\Vert \left\Vert \delta B\right\Vert +\left\Vert L_{B}\right\Vert \left\Vert L_{B}^{-1}\right\Vert \kappa\left(L_{B}\right)\left\Vert \delta B\right\Vert \\
 & = & \kappa^{3}\left(L_{B}\right)\mathcal{O}\left(\left\Vert \delta B\right\Vert \right).\end{eqnarray*}

\end{proof}
~~~~
\begin{lem}
\label{lem:bound on eig vec components}Suppose that $B$ is complex
symmetric, and has the LU factorization $B=L_{B}L_{B}^{\text{T}}$.
Define\textup{ \[
\mu_{1}\left(L_{B}\right)=\kappa\left(L_{B}\right)\max\left\{ \left\Vert L_{B}^{-1}\right\Vert ^{2},\left\Vert L_{B}\right\Vert ^{2}\right\} .\]
}Then the matrices $U$ and $V$ of left and right singular vectors
of $DBD$ satisfy\textup{ \begin{equation}
\max\left\{ \left\Vert D^{-1}U\Sigma^{1/2}\right\Vert ,\left\Vert DU\Sigma^{-1/2}\right\Vert \right\} \leq\mu_{1}\left(L_{B}\right),\label{eq:graded for left sing vec}\end{equation}
and \begin{equation}
\max\left\{ \left\Vert D^{-1}V\Sigma^{1/2}\right\Vert ,\left\Vert DV\Sigma^{-1/2}\right\Vert \right\} \leq\mu_{1}\left(L_{B}\right).\label{eq:graded for right sing vec}\end{equation}
Also, \[
\left|\left(U_{l}\right)_{ij}\right|\leq\kappa\left(A\right)\min\left\{ \frac{D_{ii}}{D_{jj}},\frac{D_{jj}}{D_{ii}}\right\} .\]
}\end{lem}
\begin{proof}
%
{}Since $B$ is complex symmetric, $B$ has the SVD $B=U\Sigma V^{*}$,
$V=\overline{U}$. Thus, it suffices to consider the matrix $U$ of
left singular vectors.

As in the proof of Lemma~\ref{lem:factorization lemma, perturb},
we may factor $DBD=QD^{2}A$, where $Q$ is orthogonal and \[
\left\Vert A\right\Vert \leq\left\Vert L_{B}\right\Vert ^{2},\,\,\,\,\left\Vert A^{-1}\right\Vert \leq\left\Vert L_{B}^{-1}\right\Vert ^{2}.\]
Now suppose that $D^{2}A$ has the SVD $D^{2}A=U_{l}\Sigma V_{r}^{*}$.
From Lemma~\ref{lem:grading of sing vecs, D^2A}, the matrix $U_{l}$
of left singular vectors satisfies\[
\left\Vert D^{-2}U_{l}\Sigma\right\Vert \leq\left\Vert A\right\Vert ,\,\,\,\left\Vert \left(D^{2}U_{l}\Sigma^{-1}\right)\right\Vert \leq\left\Vert A^{-1}\right\Vert .\]
In particular, \[
\frac{\Sigma_{jj}}{D_{ii}^{2}}\left|\left(U_{l}\right)_{ij}\right|\leq\left\Vert A\right\Vert ,\,\,\,\frac{D_{ii}^{2}}{\Sigma_{jj}}\left|\left(U_{l}\right)_{ij}\right|\leq\left\Vert A^{-1}\right\Vert ,\]
and so \begin{eqnarray*}
\left|\left(U_{l}\right)_{ij}\right| & \leq & \max\left\{ \left\Vert A^{-1}\right\Vert ,\left\Vert A\right\Vert \right\} \min\left\{ \frac{D_{ii}^{2}}{\Sigma_{jj}},\frac{\Sigma_{jj}}{D_{ii}^{2}}\right\} ,\\
 & \leq & \max\left\{ \left\Vert L_{B}^{-1}\right\Vert ^{2},\left\Vert L_{B}\right\Vert ^{2}\right\} \min\left\{ \frac{D_{ii}}{\sqrt{\Sigma_{jj}}},\frac{\sqrt{\Sigma_{jj}}}{D_{ii}}\right\} .\end{eqnarray*}
We can write the inequality above as\[
\max\left\{ \left\Vert D^{-1}U_{l}\Sigma^{1/2}\right\Vert ,\left\Vert DU_{l}\Sigma^{-1/2}\right\Vert \right\} \leq\max\left\{ \left\Vert L_{B}^{-1}\right\Vert ^{2},\left\Vert L_{B}\right\Vert ^{2}\right\} .\]
Also, from Lemma~\ref{lem:bound on Q components-1}, the unitary
matrix $Q$ from the QR factorization $L_{1}=QR$ satisfies \begin{equation}
\max\left\{ \left\Vert D^{-1}QD\right\Vert ,\left\Vert DQD^{-1}\right\Vert \right\} \leq\kappa\left(L_{B}\right)\label{eq:bound on Q, graded}\end{equation}
Since $DBD=QD^{2}A$, Lemma~\ref{lem:mult of graded matrices-1}
shows that the left singular vectors $U=QU_{l}$ of $DBD$ satisfy
\begin{equation}
\max\left\{ \left\Vert D^{-1}U\Sigma^{1/2}\right\Vert ,\left\Vert DU\Sigma^{-1/2}\right\Vert \right\} \leq\mu_{1}\left(L_{B}\right),\label{eq:left singular vectors of DBD}\end{equation}
where $\mu_{1}\left(L_{B}\right)$ is defined in the statement of
the lemma. \end{proof}
\begin{lem}
\label{lem:main bounds on eig components}Let $B$ be complex symmetric,
with an LU factorization $B=L_{B}L_{B}^{\text{T}}$. Let $D>0$ be
a positive-definite diagonal matrix, and define \begin{eqnarray*}
D_{1}B_{1}D_{1} & = & \left(\begin{array}{cc}
D & 0\\
0 & D\end{array}\right)\left(\begin{array}{cc}
0 & B\\
B^{*} & 0\end{array}\right)\left(\begin{array}{cc}
D & 0\\
0 & D\end{array}\right).\end{eqnarray*}
Then the $i$th eigenvector of $D_{1}B_{1}D_{1}$, corresponding to
eigenvalue $\Sigma_{ii}$, may be chosen so that the following component-wise
bounds hold: \textup{\begin{equation}
\left|x_{i}\left(j\right)\right|\leq\sqrt{2}\mu_{1}\left(L_{B}\right)\min\left\{ \frac{D_{jj}}{\sqrt{\Sigma_{ii}}},\frac{\sqrt{\Sigma_{ii}}}{D_{jj}}\right\} .\label{eq:main bound on x_i(j)}\end{equation}
The following norm-wise bound also holds: \[
\left\Vert D_{1}x_{i}\Sigma_{ii}^{-1/2}\right\Vert \geq\left\Vert B\right\Vert ^{-1/2}.\]
}\end{lem}
\begin{proof}
Since $B$ is complex symmetric, $B$ has the SVD $B=U\Sigma V^{*}$,
$V=\overline{U}$. Therefore, $D_{1}B_{1}D_{1}$ has the eigenvalue
decomposition \begin{eqnarray*}
D_{1}B_{1}D_{1} & = & \frac{1}{\sqrt{2}}\left(\begin{array}{cc}
\overline{U} & \overline{U}\\
U & -U\end{array}\right)\left(\begin{array}{cc}
\Sigma & 0\\
0 & -\Sigma\end{array}\right)\frac{1}{\sqrt{2}}\left(\begin{array}{cc}
\overline{U} & \overline{U}\\
U & -U\end{array}\right)^{*}\\
 & = & U_{1}\Sigma_{1}U_{1}^{*}.\end{eqnarray*}
 From (\ref{eq:graded for left sing vec}) in Lemma~\ref{lem:bound on eig vec components},\begin{equation}
\max\left\{ \left\Vert D^{-1}U\Sigma^{1/2}\right\Vert ,\left\Vert DU\Sigma^{-1/2}\right\Vert \right\} \leq\kappa\left(L_{B}\right)\max\left\{ \left\Vert L_{B}^{-1}\right\Vert ^{2},\left\Vert L_{B}\right\Vert ^{2}\right\} .\label{eq:graded of U for main lemma bound}\end{equation}
Since \begin{eqnarray*}
D_{1}^{-1}U_{1}\left|\Sigma_{1}\right|^{1/2} & = & \frac{1}{\sqrt{2}}\left(\begin{array}{cc}
D & 0\\
0 & D\end{array}\right)\left(\begin{array}{cc}
\overline{U} & \overline{U}\\
U & -U\end{array}\right)\left(\begin{array}{cc}
\left|\Sigma\right|^{1/2} & 0\\
0 & \left|\Sigma\right|^{1/2}\end{array}\right)\\
 & = & \frac{1}{\sqrt{2}}\left(\begin{array}{cc}
D\overline{U}\left|\Sigma\right|^{1/2} & D\overline{U}\left|\Sigma\right|^{1/2}\\
DU\left|\Sigma\right|^{1/2} & -DU\left|\Sigma\right|^{1/2}\end{array}\right),\end{eqnarray*}
we have from (\ref{eq:graded of U for main lemma bound}) that \[
\left\Vert D_{1}^{-1}U_{1}\left|\Sigma_{1}\right|^{1/2}\right\Vert \leq\sqrt{2}\kappa\left(L_{B}\right)\max\left\{ \left\Vert L_{B}^{-1}\right\Vert ^{2},\left\Vert L_{B}\right\Vert ^{2}\right\} .\]
An analogous calculation shows that \[
\left\Vert D_{1}U_{1}\left|\Sigma_{1}\right|^{-1/2}\right\Vert \leq\sqrt{2}\kappa\left(L_{B}\right)\max\left\{ \left\Vert L_{B}^{-1}\right\Vert ^{2},\left\Vert L_{B}\right\Vert ^{2}\right\} .\]
The previous two inequalities show that the $i$th eigenvector $x_{i}$
of $D_{1}B_{1}D_{1}$ (e.g., the $i$th column of $U_{1}$) satisfies
the component-wise bound \[
\left|x_{i}\left(j\right)\right|=\left|\left(U_{1}\right)_{ji}\right|\leq\sqrt{2}\mu_{1}\left(L_{B}\right)\min\left\{ \frac{D_{jj}}{\sqrt{\Sigma_{ii}}},\frac{\sqrt{\Sigma_{ii}}}{D_{jj}}\right\} .\]

We now consider $\left\Vert D_{1}x_{i}\right\Vert /\sqrt{\Sigma_{ii}}$.
Since $x_{i}=\left(1/\sqrt{2}\right)\left[\overline{u_{i}}u_{i}\right]^{\text{T}}$,
we have from the norm-wise bounds in Proposition~\ref{pro:perturb for sing vecs of DBD}
that \[
\left\Vert D_{1}x_{i}\Sigma_{ii}^{-1/2}\right\Vert =\frac{1}{\sqrt{2}}\left\Vert \left(\begin{array}{c}
D\overline{u_{i}}\Sigma_{ii}^{-1/2}\\
Du_{i}\Sigma_{ii}^{-1/2}\end{array}\right)\right\Vert =\left\Vert D_{1}u_{i}\Sigma_{ii}^{-1/2}\right\Vert \geq\left\Vert B\right\Vert ^{-1/2}.\]

\end{proof}

\subsection{Proof of the main result}
\begin{proof}
From Lemma~\ref{lem:factorization lemma, perturb}, the right singular
vectors of $DBD$ and $D\left(B+\delta B\right)D$ agree with the
right singular vectors of $D^{2}A$ and $D^{2}\left(A+\delta A\right)$,
where the matrices $A$ and $E$ satisfy \begin{equation}
\left\Vert A\right\Vert \leq\left\Vert L_{B}\right\Vert ^{2},\,\,\,\,\left\Vert A^{-1}\right\Vert \leq\left\Vert L_{B}^{-1}\right\Vert ^{2},\,\,\,\left\Vert \delta A\right\Vert \leq\kappa^{3}\left(L_{B}\right)\mathcal{O}\left(\left\Vert \delta B\right\Vert \right).\label{eq:bounds in A,A^{-1}, main prop}\end{equation}
Also, Lemmas~\ref{lem:grading of sing vecs, D^2A}~and~,\ref{lem:bound on eig vec components}
ensure that the $i$th left and right singular vectors $u_{i}$ and
$v_{i}$ of $D^{2}A$ are bounded by \begin{equation}
\max\left\{ \left|u_{i}\left(j\right)\right|,\left|v_{i}\left(j\right)\right|\right\} \leq\mu_{1}\left(L_{B}\right)\min\left\{ \frac{D_{jj}}{\sqrt{\Sigma_{ii}}},\frac{\sqrt{\Sigma_{ii}}}{D_{jj}}\right\} .\label{eq:bound on x_i(j), main prop}\end{equation}

To prove the component-wise bounds, we now proceed as in Theorem $2.21$
of \cite{DEM-VES:1992} (except now we may use the above component-wise
bounds for both $v_{i}\left(j\right)$ and $u_{i}\left(j\right)$).
Namely, let $\delta A=\eta E$, where $\eta=\left\Vert \delta A\right\Vert $,
and define \[
D_{1}=\left(\begin{array}{cc}
D^{2} & 0\\
0 & I\end{array}\right),\,\,\, E_{1}=\left(\begin{array}{cc}
0 & E^{*}\\
E & 0\end{array}\right),\,\,\, A_{1}=\left(\begin{array}{cc}
0 & A^{*}\\
A & 0\end{array}\right),\,\,\, x_{i}^{\pm}=\frac{1}{\sqrt{2}}\left(\begin{array}{c}
v_{i}\\
\pm u_{i}\end{array}\right).\]
Let $x_{i}^{\pm}\left(\eta\right)$ denote the eigenvectors of $G\left(\eta\right)$,
\begin{eqnarray*}
G\left(\eta\right) & = & \left(\begin{array}{cc}
0 & \left(A+\eta E\right)^{*}D^{2}\\
D^{2}\left(A+\eta E\right) & 0\end{array}\right)=D_{1}\left(A_{1}+\eta E_{1}\right)D_{1}.\end{eqnarray*}
%
{} Then by standard perturbation theory,\[
x_{i}^{\pm}\left(\eta\right)=x_{i}+\eta\sum_{\pm k\neq\pm i}\frac{\left(x_{i}^{\pm}\right)^{*}D_{1}E_{1}D_{1}x_{k}^{\pm}}{\pm\sigma_{i}\mp\sigma_{k}}x_{k}^{\pm}+\mathcal{O}\left(\eta^{2}\right),\]
and so \[
\left|x_{i}^{\pm}\left(\eta\right)\left(j\right)-x_{i}\left(j\right)\right|\leq\mathcal{O}\left(\eta\right)\sum_{\pm k\neq\pm i}\frac{\left|\left(x_{i}^{\pm}\right)^{*}D_{1}E_{1}D_{1}x_{k}^{\pm}\right|}{\pm\sigma_{i}\mp\sigma_{k}}\left|x_{k}^{\pm}\left(j\right)\right|.\]
%
{}Now,\begin{eqnarray*}
\left|\left(x_{i}^{\pm}\right)^{*}D_{1}E_{1}D_{1}x_{k}^{\pm}\right| & \leq & \frac{1}{2}\left|\left(v_{i},D^{2}E^{*}u_{k}\right)\right|+\frac{1}{2}\left|\left(u_{i},ED^{2}v_{k}\right)\right|\\
 & = & \frac{1}{2}\left|\left(D^{2}v_{i},E^{*}u_{k}\right)\right|+\frac{1}{2}\left|\left(u_{i},ED^{2}v_{k}\right)\right|\\
 & \leq & \frac{\left\Vert D^{2}v_{i}\right\Vert +\left\Vert D^{2}v_{k}\right\Vert }{2}.\end{eqnarray*}
Now, %
{} we have that $D^{2}Av_{k}=\sigma_{k}u_{k}$, and so\[
\sigma_{k}=\left\Vert D^{2}Av_{k}\right\Vert \geq\left\Vert A^{-1}\right\Vert ^{-1}\left\Vert D^{2}v_{k}\right\Vert ,\,\,\,\left\Vert D^{2}v_{k}\right\Vert \leq\sigma_{k}\left\Vert A^{-1}\right\Vert .\]
Similarly, $\left\Vert D^{2}v_{i}\right\Vert \leq\sigma_{i}\left\Vert A^{-1}\right\Vert $.
Therefore, \begin{equation}
\left|\left(x_{i}^{\pm}\right)^{*}D_{1}E_{1}D_{1}x_{k}^{\pm}\right|\leq\frac{\left\Vert D^{2}v_{i}\right\Vert +\left\Vert D^{2}v_{k}\right\Vert }{2}\leq\left\Vert A^{-1}\right\Vert \frac{\sigma_{i}+\sigma_{k}}{2}.\label{eq:bound on inner prod, main prop}\end{equation}
Finally, from (\ref{eq:bound on x_i(j), main prop}) and (\ref{eq:bounds in A,A^{-1}, main prop}),
\begin{eqnarray*}
\left|x_{i}^{\pm}\left(\eta\right)\left(j\right)-x_{i}\left(j\right)\right| & \leq & \mathcal{O}\left(\eta\right)\sum_{\pm k\neq\pm i}\frac{\left|\left(x_{i}^{\pm}\right)^{*}D_{1}E_{1}D_{1}x_{k}^{\pm}\right|}{\pm\sigma_{i}\mp\sigma_{k}}\left|x_{k}^{\pm}\left(j\right)\right|\\
 & \leq & \mathcal{O}\left(\eta\right)\left\Vert A^{-1}\right\Vert \mu_{1}\left(L_{B}\right)\sum_{\pm k\neq\pm i}\frac{\sigma_{i}+\sigma_{k}}{\left|\sigma_{i}-\sigma_{k}\right|}\min\left\{ \frac{D_{jj}}{\sqrt{\Sigma_{ii}}},\frac{\sqrt{\Sigma_{ii}}}{D_{jj}}\right\} \\
 & \leq & \mathcal{O}\left(\eta\right)\left\Vert L_{B}^{-1}\right\Vert ^{2}\mu_{1}\left(L_{B}\right)\left(\max_{k\neq i}\frac{\sigma_{i}+\sigma_{k}}{\left|\sigma_{i}-\sigma_{k}\right|}\right)\min\left\{ \frac{D_{jj}}{\sqrt{\Sigma_{ii}}},\frac{\sqrt{\Sigma_{ii}}}{D_{jj}}\right\} \\
 & \leq & \mathcal{O}\left(\left\Vert \delta B\right\Vert \right)\left\Vert L_{B}^{-1}\right\Vert ^{2}\mu_{1}\left(L_{B}\right)\kappa^{3}\left(L_{B}\right)\left(\max_{k\neq i}\frac{\sigma_{i}+\sigma_{k}}{\left|\sigma_{i}-\sigma_{k}\right|}\right)\min\left\{ \frac{D_{jj}}{\sqrt{\Sigma_{ii}}},\frac{\sqrt{\Sigma_{ii}}}{D_{jj}}\right\} ,\end{eqnarray*}
where we used (\ref{eq:bound on inner prod, main prop}) in the second
to last inequality and (\ref{eq:bounds in A,A^{-1}, main prop}) in
the last inequality. 

We now prove the norm-wise bound on $Du_{i}/\Sigma_{ii}$. To do so,
note that, since $B$ is complex symmetric, the SVD of $B$ may be
written as $B=U\Sigma U^{\text{T}}$ (e.g., $B$ has a Takagi factorization).
Now suppose that $DBD\overline{u_{i}}=\sigma_{i}u_{i}$. Then the
bound follows from \[
\sigma_{i}=\left|\left(DBD\overline{u_{i}},u_{i}\right)\right|=\left|\left(BD\overline{u_{i}},Du_{i}\right)\right|\leq\left\Vert B\right\Vert \left\Vert Du_{i}\right\Vert ^{2}.\]

Finally, the bounds on the ratio $\widetilde{\Sigma_{ii}}/\Sigma_{ii}$
of the singular values of $DBD$ and $D\left(B+\delta B\right)D$
follow from Theorem~$4.1$ of \cite{D-G-E-S-V-D:1997}. 
\end{proof}
~~~~

\bibliographystyle{plain}

\end{document}